\documentclass[final,onefignum,onetabnum]{siamart171218}

\usepackage{amssymb}
\usepackage{caption}
\usepackage{hyperref}
\newfloat{algorithm}{t}{lop}

\theoremstyle{plain}
\newtheorem{Theorem}{Theorem}
\newtheorem{Definition}{Definition}

\newtheorem{Corollary}{Corollary}[Theorem]

\newtheorem{Lemma}{Lemma}
\newtheorem{Assumption}{Assumption}
\newtheorem{Model}{Model}

\newcommand\norm[1]{\left\lVert#1\right\rVert}

\usepackage{lipsum}
\usepackage{amsfonts}
\usepackage{graphicx}
\usepackage{epstopdf}
\usepackage{algorithmic}
\ifpdf
  \DeclareGraphicsExtensions{.eps,.pdf,.png,.jpg}
\else
  \DeclareGraphicsExtensions{.eps}
\fi

\newsiamremark{remark}{Remark}
\newsiamremark{hypothesis}{Hypothesis}
\crefname{hypothesis}{Hypothesis}{Hypotheses}
\newsiamthm{claim}{Claim}

\headers{Sharper Bounds for Proximal Gradient Algorithms with Errors}{A. Hamadouche}

\title{Sharper Bounds for Proximal Gradient Algorithms with Errors\thanks{{Work supported by UK's EPSRC (EP/T026111/1, EP/S000631/1), and the MOD University Defence     Research Collaboration.}}}
\author{Anis Hamadouche\thanks{Anis Hamadouche, Yun Wu, Andrew M.\ Wallace, and Jo\~ao F.\ C.\ Mota     are with the School of Engineering \& Physical Sciences, Heriot-Watt University, Edinburgh EH14 4AS,     UK. (e-mail: \{ah225,y.wu,a.m.wallace,j.mota\}@hw.ac.uk).}
\and Yun Wu\footnotemark[2]
\and Andrew M.\ Wallace\footnotemark[2]
\and Jo\~ao F.\ C.\ Mota\footnotemark[2]}

\usepackage{amsopn}

\makeatletter
\newcommand*{\addFileDependency}[1]{
  \typeout{(#1)}
  \@addtofilelist{#1}
  \IfFileExists{#1}{}{\typeout{No file #1.}}
}
\makeatother

\newcommand*{\myexternaldocument}[1]{%
    \externaldocument{#1}%
    \addFileDependency{#1.tex}%
    \addFileDependency{#1.aux}%
}

\ifpdf
\hypersetup{
  pdftitle={An Example Article},
  pdfauthor={D. Doe, P. T. Frank, and J. E. Smith}
}
\fi

\myexternaldocument{ex_supplement}

\begin{document}

\maketitle

\begin{abstract}
    \textbf{We analyse the convergence of the proximal gradient algorithm for convex composite problems in the presence of gradient and proximal computational inaccuracies. We derive  new tighter deterministic and probabilistic bounds that we use to verify a simulated (MPC) and a synthetic (LASSO) optimization problems solved on a reduced-precision machine in combination with an inaccurate proximal operator. We also show how the probabilistic bounds are more robust for algorithm verification and more accurate for application performance guarantees. Under some statistical assumptions, we also prove that some cumulative error terms follow a martingale property. And conforming to observations, e.g., in \cite{schmidt2011convergence}, we also show how the acceleration of the algorithm amplifies the gradient and proximal computational errors.}
\end{abstract}



\section{Introduction}
\label{Sec:Intro}
Many problems in science and engineering can be posed as \textit{composite optimization
problems}:
\begin{align}
  \label{Eq:Problem}
  \underset{x \in \mathbb{R}^n}{\text{minimize}} \,\,\,
  f(x) := g(x) + h(x)\,,
\end{align}
where the function $g\,:\, \mathbb{R}^{n} \to \mathbb{R}$ is real-valued
and differentiable, and the function $h:\mathbb{R}^n \to \mathbb{R} \cup
\{+\infty\}$ is not necessarily differentiable and is possibly infinite-valued,
enabling the inclusion of hard
constraints in~\eqref{Eq:Problem}. Examples
include various machine learning frameworks, e.g., logistic regression and
support vector machines \cite{cortes1995support}, sparse regression and inference
\cite{quinonero2005unifying,lawrence2003fast,lawrence2001sparse}, image processing \cite{afonso2010fast}, and discrete
optimal control \cite{nagahara2015maximum}. 

A popular class of algorithms to solve~\eqref{Eq:Problem} is
\textit{proximal gradient methods} \cite{beck2017first} which, in each
iteration, take a gradient step using the function $g$ and, subsequently,
evaluate the proximal operator of the function $h$ at the resulting point. Such
algorithms have been widely studied under different contexts, and several
guarantees have been established, both in the
convex~\cite{beck2009fast,beck2017first,bertsekas2015convex,combettes2005signal,palomar2010convex}
and nonconvex~\cite{bolte2018first,ochs2019non} cases. Stochastic versions of
the proximal gradient algorithm have also been proposed and shown to converge
in convex and nonconvex settings, e.g.,  
\cite{atchade2014stochastic,atchade2014stochastic,zhou2018distributed,nitanda2014stochastic,rosasco2019convergence,
davis2016sound,zhou2016convergence}.

All of these results, however, assume that computations are performed with
near-infinite precision, which is unrealistic when the computational platform
has limitations in power, precision, or both. Examples include applications that are associated with sensing and control of autonomous platforms, often using FPGAs or other finite precision computational hardware. With these applications in mind, we
analyze proximal gradient methods when both the gradient and the proximal
operator are computed approximately at each iteration, and obtain tight
performance bounds.
\subsection{Problem statement}
We consider a convex instantiation of~\eqref{Eq:Problem}. Namely, we assume that $g\,:\,
\mathbb{R}^{n} \to \mathbb{R}$ is a convex, differentiable function, and has a
Lipschitz-continuous gradient with constant $L > 0$, i.e., $\|\nabla g(x) - \nabla
g(y)\|_2 \leq L \|x - y\|_2$, for all $x$, $y \in \mathbb{R}^n$. We also assume that $h:\mathbb{R}^n \to \mathbb{R} \cup
\{+\infty\}$ is closed, proper, and convex. In this case, given arbitrary initial $x^1 = x^0 \in \mathbb{R}^n$, the accelerated proximal gradient
descent
algorithm applied to~\eqref{Eq:Problem} consists of iterating for $k>1$
\begin{equation}
  \label{Eq:PGNoiseless}
  \begin{split}
    y^k &= x^k +\beta_k(x^k-x^{k-1})
    \\
    x^{k+1} &=\text{prox}_{s_{k} h}\big(y^{k} - s_{k}\nabla g(y^{k})\big), 
  \end{split}
\end{equation}
where $0 < \beta_k \leq 1$ is the \textit{momentum} at
iteration $k$, which we assume takes the form $\beta_k = (\alpha_{k-1}-1)/\alpha_k$, where $\{\alpha_k\}_{k\geq1}$ is a given parameter sequence satisfying $\alpha_0 = 1$,  $\alpha_k \geq 1$, and $\alpha_{k}^2-\alpha_{k} = \alpha_{k-1}^2$ for all $k \geq 1$. As an example, we can use $\alpha_{k}=(k+2)/2$. The stepsize $s_{k}$ at iteration $k$ satisfies $0 < s_{k} \leq 1/L$, and
\begin{equation}
  \label{Eq:Prox}
    \text{prox}_{u}(y) 
    :=
    \underset{x \in \mathbb{R}^n}{\arg\min}\,\,\, 
    u(x) + \frac{1}{2}\|x - y\|_2^2
\end{equation}
is the proximal operator of $u:\mathbb{R}^n\to\mathbb{R}\cup \{+\infty\}$ at $y
\in \mathbb{R}^n$. The special case of $\beta_k = 0$ corresponds to the \textit{basic (unaccelerated) proximal gradient} algorithm 
\begin{equation}
  \label{Eq:BasicPGNoiseless}
    x^{k+1} =\text{prox}_{s_{k} h}\big(x^{k} - s_{k}\nabla g(x^{k})\big).
\end{equation}

We consider the case in which both the gradient of $g$
in \eqref{Eq:PGNoiseless} and~\eqref{Eq:BasicPGNoiseless} and the proximal operator~\eqref{Eq:Prox} are
computed approximately at each iteration.  Specifically, we consider the
\textit{approximate accelerated proximal gradient} algorithm
\begin{equation}
    \label{Eq:AcceleratedPG}
    \begin{split}
        y^k &= x^k +\beta_k(x^k-x^{k-1}),\\
        x^{k+1} 
        &\in 
        \text{prox}_{s_{k} h}^{\epsilon_{2}^{k}}
        \Big[y^{k} - s_{k}\big(\nabla g(y^{k})+\epsilon_1^{k}\big)\Big]\,,
    \end{split}
\end{equation}
and its original \textit{approximate basic (unaccelerated) proximal gradient} algorithm
\begin{equation}
    \label{Eq:UnacceleratedPG}
    x^{k+1} 
    \in 
    \text{prox}_{s_{k} h}^{\epsilon_{2}^{k}}
    \Big[x^{k} - s_{k}\big(\nabla g(x^{k})+\epsilon_1^{k}\big)\Big]\,
\end{equation}
where $\epsilon_1^{k} \in \mathbb{R}^n$ and $\epsilon_2^{k} \in \mathbb{R}_+$
model, respectively, the error in the gradient and the error in the computation
of the proximal operator at
iteration $k$. In~\eqref{Eq:AcceleratedPG} and ~\eqref{Eq:UnacceleratedPG}, $\text{prox}_{s_{k} h}^{\epsilon_{2}^{k}}(y)$ is the set of vectors that are $\epsilon_2^k$-suboptimal in the computation of
the proximal operator of $s_k h$ at a point $y \in \mathbb{R}^n$: 

\begin{equation}
  \label{Eq:ProxApproximate}    
    \text{prox}_{u}^{\epsilon}(y) 
    := 
    \Big\{
      x \in \mathbb{R}^n\,:\,
      u(x) + \frac{1}{2}\|x -y\|_2^2 \leq \epsilon + \underset{z}{\inf}\,\,
      u(z) + \frac{1}{2}\|z -
      y\|_2^2 
    \Big\}\,,
\end{equation}
which we will denote as the \textit{$\epsilon$-suboptimal} proximal of $u:\mathbb{R}^n\to\mathbb{R}\cup \{+\infty\}$ at $y
\in \mathbb{R}^n$.

While standard proximal gradient methods~\eqref{Eq:PGNoiseless} and~\eqref{Eq:BasicPGNoiseless} converge to a solution of~\eqref{Eq:Problem}
provided the stepsize $s_{k}$ is small enough, approximate
proximal gradient algorithms [\eqref{Eq:AcceleratedPG} and ~\eqref{Eq:UnacceleratedPG}] require, in addition, that the approximation
errors $\epsilon_1^k$ and $\epsilon_2^{k}$ satisfy some additional convergence criteria, for example, that they converge to zero along the iterations. 

Our goal is then \textit{to characterize the convergence of the approximate
proximal gradient [\eqref{Eq:AcceleratedPG} and ~\eqref{Eq:UnacceleratedPG}] to a solution
of~\eqref{Eq:Problem}}. 
Differently from prior work, we assume not only deterministic
errors, but also probabilistic ones, according to models suited to approximate
computing. 

\subsection{Our approach}
In the case of deterministic errors, we get inspiration
from~\cite{beck2017first} to derive, using simple arguments, upper bounds on
$f(x^k)$ throughout the iterations. The resulting bounds are simpler and
tighter than other bounds \cite{schmidt2011convergence},\cite{aujol2015stability}. 
In the case of probabilistic errors, our arguments rely on
concentration of measure results for martingale sequences and bypass
the need to assume that $\epsilon_1^k$ and $\epsilon_2^k$ converge to zero. We
believe this line of reasoning is novel in the analysis of approximate proximal gradient
algorithms. 

\subsection{Applications}
In order to validate our convergence results, we use the proposed error bounds to analyse the convergence of \eqref{Eq:AcceleratedPG} and \eqref{Eq:UnacceleratedPG} when applied to Model Predictive Control (MPC) \cite{hegrenaes2005spacecraft} with different levels of injected gradient and proximal computation errors. We also apply the same set of bounds to analyse the \textit{proximal gradient} algorithm for solving randomly generated LASSO problems \cite{tibshirani1996regression}. For the latter, instead of generating the errors from a known distribution as in the MPC test, we use the developed benchmark \cite{hamadouche2021sspd} to vary the fixed-point machine representation and the proximal computation precision so that we obtain more realistic error sequences.

\subsection{Contributions}
We summarize our contributions as follows:
\begin{itemize}

  \item 
    We establish convergence bounds for the proximal gradient algorithm with
    deterministic and probabilistic errors. Our bounds are simpler than prior
    bounds. 
  \item We conduct experiments on a discrete model predictive control problem to verify the sharpness of
    our bounds and compare them with the bounds
    in~\cite{schmidt2011convergence}. The models for the errors are inspired by
    approximate computing techniques suited for
    low-precision machines, such as reduced-precision accelerators on FPGA and battery-operated devices, in which  algorithms are typically run approximately in order to save processing time and/or power. We also run experiments on a real benchmark that uses fixed-point arithmetic and tunable CVX solver precision \cite{grant2009cvx}.
  \item We propose new models for the proximal and gradient errors
  that satisfy martingale properties in consistence with experimental results.
\end{itemize}

\subsection{Organization}
We start by reviewing prior work in Section~\ref{Sec:RelatedWork}. In
Section~\ref{Sec:MainResults}, we then describe our model, state our assumptions, and
present the main results. The proofs are in Section~\ref{Sec:Proofs}, and some
auxiliary results are relegated to the appendix.
Section~\ref{Sec:Experiments} then describes our experimental results, and we
conclude in Section~\ref{Sec:Conclusion}.

\section{Related Work}
\label{Sec:RelatedWork}
\subsection{Origins of proximal gradient}
Many optimization algorithms have been developed to address large-scale problems arising in data science and machine learning applications. For instance, gradient methods, which use the gradient of the function as a search direction to iteratively find points with lower (or larger) cost, are suited
 to smooth convex problems with simple and typically inexpensive gradient calculation. For constrained problems, projected gradient methods \cite{lin2007projected} involve an extra projection step onto the feasible set. The first instance of a gradient method can be traced back to Louis Augustin Cauchy \cite{cauchy1847methode}, who suggested the use of derivatives to solve optimization problems in 1847 \cite{lemarechal2012cauchy}. The convergence of the resulting method for nonlinear optimization problems, however, was established only in 1944. \cite{curry1944method}.

Subgradient and projected subgradient methods, originally developed by Shor, in the 1970s \cite{shor2012minimization}, generalize gradient and projected gradient methods for the case in which the objective functions are not differentiable (but still convex) \cite{akgul1984topics,boyd2003subgradient,beck2017first}.

The proximal operator \eqref{Eq:Prox} generalizes the projection operator \cite{moreau1965proximite}. Specifically, if we set $u(x)$ in  \eqref{Eq:Prox} as the indicator function of a (convex) set $S \subset \mathbb{R}^n$, $\text{prox}_u(y)$ becomes the projection of $y$ onto $S$. Proximal splitting algorithms iteratively apply the proximal operator of a function in combination with the gradient or proximal operator of other functions, which often results in simple algorithms with tolerable per-iteration complexity \cite{aujol2019rates}. It is surprising that the proximal gradient algorithm \eqref{Eq:BasicPGNoiseless} applied to composite problems \eqref{Eq:Problem} has the same convergence rate as the classical gradient algorithm [i.e., when $h(x)=0$ in \eqref{Eq:Problem}], which applies only to much simpler problems. In particular, in both problems the objective function decreases along the iterations $k$ as $O(1/k)$ \cite{nesterov2013gradient, beck2017first, beck2009fast}. 

According to \cite{nesterov1983method}, it is also possible to accelerate gradient-based methods to achieve higher convergence rates by evaluating the gradient at a linear combination of two consecutive iterates,  in~\eqref{Eq:PGNoiseless}. 
The accelerated proximal gradient descent algorithm~\eqref{Eq:PGNoiseless} was applied in \cite{beck2009fast} to LASSO, and the resulting algorithm, famously known as FISTA, was shown to converge at a rate of $O(1/k^2)$. The same rate applies to general convex problems \cite{nesterov1983method,tseng2008accelerated, beck2009fast,nesterov2013gradient,beck2017first}.

\subsection{Stochastic proximal gradient}
The iterations of stochastic proximal gradient are exactly as in standard proximal gradient \eqref{Eq:PGNoiseless} or~\eqref{Eq:BasicPGNoiseless}, but the gradient of $g$ is computed approximately in order to save computation or to avoid retrieving all the points in a database. Specifically, in many applications, $g$ is a sum of functions each of which depends on one (or a few) datapoints of a given dataset. For example, $g(x) = \sum_{i=1}^m g_i(x)$, where $m$ is the number of points in the dataset and $g_i\,:\, \mathbb{R}^n \to \mathbb{R}$ measures the error of a model on the $i$th point of the dataset. As $\nabla g(x^k) = \sum_{i=1}^m \nabla g_i(x)$, proximal gradient in \eqref{Eq:PGNoiseless} or \eqref{Eq:BasicPGNoiseless} requires visiting all the points of the dataset at each iteration, which can be time-consuming. To overcome this, stochastic proximal gradient approximates the sum $\sum_{i=1}^m \nabla g_i(x^k)$ at iteration $k$ by $\sum_{i\in \mathcal{S}_k} \nabla g_i(x^k)$, where $\mathcal{S}_k$ is a random, but small, subset of $\{1, \ldots, m\}$. Errors in stochastic proximal gradient thus stem from approximating a sum of gradients with a truncated sum. The problem we address is more general than the one addressed by stochastic proximal gradient, as we do not necessarily assume that $g$ in \eqref{Eq:Problem} is additive, i.e., $g(x) = \sum_{i=1}^m g_i(x)$.

In a recent convergence analysis of the stochastic proximal gradient algorithm,  \cite{rosasco2019convergence} considered stochastic perturbations (realizations of a Gaussian random variable with $0$ mean) of the gradient and bypassed the need to make the assumption that the gradient error is summable, i.e., $\sum_{i=1}^{k}\alpha_i\sqrt{\mathbb{E}\big[\norm{\epsilon_1^i}_2^2\big]} < \infty$, where $\{\alpha_i\}$ is a sequence related to the acceleration momentum $\beta_i$, and $\mathbb{E}[\cdot]$ represents the expected value of a random variable. Note that $\epsilon_1^k$ here results from replacing the exact gradient by a stochastic estimate \cite{rosasco2019convergence, atchade2014stochastic} rather than using finite-precision computations, as in our case. In our analysis we adopt a relaxation step different from  \cite{rosasco2019convergence} and obtain  more realistic upper bounds on the function values.

In our work we consider the deterministic proximal gradient algorithm with perturbed gradient as in \cite{schmidt2011convergence, aujol2019rates}, whose convergence proofs follow a slightly different line than the proofs of stochastic proximal gradient algorithm \cite{rosasco2019convergence, atchade2014stochastic}. Moreover, instead of implicitly assuming the availability of the closed-form expression of the proximal operator as in \cite{rosasco2019convergence, atchade2014stochastic, aujol2019rates}, we consider additional deterministic and random perturbations in the proximal operation step, as in \cite{schmidt2011convergence}, and analyse the manifestation of both types of errors. We also establish probabilistic upper bounds (see Theorem \ref{prop:2}, \ref{prop:3} and \ref{Theorem2}). Our analysis is thus fundamentally different from the analyses in \cite{schmidt2011convergence} and \cite{aujol2019rates}, both of which assumed deterministic error models.

Without any assumption on the decay of the error terms, if we use probabilistic error models and set $\epsilon_{2_\Omega}^k = 0$ in~\eqref{Eq:AcceleratedPG}, for all $k > 0$, for a given probability sample space $\Omega$, and we use inexact gradient computations $\nabla^{\epsilon_1^k}g$, then we retrieve the same worst case convergence rates of \cite{atchade2017perturbed} and \cite{rosasco2019convergence}, i.e,  $O({1}/{\sqrt{k}})$, but with better coefficients.
Similar conclusions to the ones in \cite{aujol2019rates,atchade2014stochastic,atchade2017perturbed} would follow if we further assume that $\{\alpha_k\epsilon_{1_\Omega}^{k}\}$ is summable, which is a weaker requirement than the square summability of \cite{aujol2019rates}, thereby the recovery of the optimal rate $O({1}/{k^2})$ in the approximate accelerated proximal gradient algorithm.

\subsection{Approximate proximal gradient}
One year after the seminal work in \cite{beck2009fast}, it was shown that the same nearly optimal rates can still be achieved when the computation of the gradients and proximal operators are approximate \cite{schmidt2011convergence}. This variant is known as the \textit{approximate} proximal gradient algorithm. The analysis in \cite{schmidt2011convergence} requires the errors $\epsilon_1^k$ and $\epsilon_2^k$ to decrease with iterations $k$ at rates $O({1}/{k^{\varsigma + 1}})$ for the basic proximal gradient~\eqref{Eq:UnacceleratedPG}, and $O({1}/{k^{\varsigma + 2}})$ for the accelerated proximal gradient~\eqref{Eq:AcceleratedPG}, for any $\varsigma > 0$, in order to satisfy the summability assumptions of both error terms. The work in \cite{schmidt2011convergence} established the following ergodic convergence bound in terms of function values of the averaged iterates for the basic approximate proximal gradient \eqref{Eq:UnacceleratedPG}:
  \begin{equation}
    \label{schmidt1}
    \begin{split}
    &f\bigg(\frac{1}{k}\sum_{i=1}^{k}x^{i}\bigg)  
    -
    f(x^\star) 
    \leq 
    \frac{L}{2k}
    \Big[\norm{x^\star-x^0}_{2}+ 2A_k + \sqrt{2B_k}\Big]^2\\
    &\quad A_k = \sum_{i=1}^{k}\Big(\frac{\|\epsilon_1^{i}\|_2}{L}+\sqrt{\frac{2\epsilon_2^{i}}{L}}\Big), \quad B_k = \sum_{i=1}^{k} \frac{\epsilon_2^{i}}{L},
    \end{split}
  \end{equation}
where $x^\star$ is an optimal solution of \eqref{Eq:Problem}, $L$ is the \textit{Lipschitz} constant of the gradient, and $x^0$ is the initialization vector. The same work also analyzed the \textit{inexact accelerated approximate  proximal gradient}~\eqref{Eq:AcceleratedPG} and obtained the following convergence result in terms of the function values of the iterates,
  \begin{equation}
    \label{schmidt2}
    \begin{split}
    &f\big(x^{i}\big)  
    -
    f(x^\star) 
    \leq 
    \frac{2L}{(k+1)^2}
    \Big[\norm{x^\star-x^0}_{2}+ 2\tilde{A}_k + \sqrt{2\tilde{B}_k}\Big]^2\\
    &\quad \tilde{A}_k = \sum_{i=1}^{k} i \Big(\frac{\|\epsilon_1^{i}\|_2}{L}+\sqrt{\frac{2\epsilon_2^{i}}{L}}\Big), \quad \tilde{B}_k = \sum_{i=1}^{k} \frac{i^2\epsilon_2^{i}}{L}.
    \end{split}
  \end{equation}
This is the most closely related work to ours; however, our work derives similar, yet sharper, convergence bounds for the \textit{inexact proximal gradient} algorithm. In addition, we derive probabilistic bounds that can be estimated before running the algorithm for given bounded proximal and gradient errors. Specifically, the bounds constants can be computed from the machine representation and software solver tolerances (for the computation of the proximal operator).

The work in \cite{aujol2015stability} extended the analysis of \cite{schmidt2011convergence} to a more general momentum parameter selection (and thus to a different algorithm) $\alpha_k =({(k+a-1)}/{a})^d$,  where $d \in [0,1] $ and $a > \text{max}(1,(2d)^\frac{1}{d})$, which becomes FISTA \cite{beck2009fast} when $d=1$. The works in \cite{aujol2015stability, villa2013accelerated} also considered two different types of approximation in the proximal operator computation. For example,  \cite[Proposition 3.3]{aujol2015stability} makes assumptions similar to ours, but establishes different bounds. The same paper also suggests slowing down the over-relaxations of FISTA to stabilize the algorithm and shows how to obtain a better trade-off between acceleration and error amplification by controlling the approximation errors. In contrast, we show that the basic approximate proximal gradient algorithm \eqref{Eq:UnacceleratedPG} converges to a constant predictable residual without any assumptions on the gradient error terms (see Theorem \ref{prop:3}). We also show that errors in the accelerated proximal gradient method cause the algorithm to eventually diverge as $O(k)$ in the worst case scenario, but converge sub-optimally, i.e., to a constant error term, using stronger assumptions on the proximal error and under a standard suitable choice of the momentum sequence $\{\beta_k\}$.
\section{Main Results}
\label{Sec:MainResults}
Before stating our convergence guarantees for the approximate proximal gradient algorithm, we specify our
assumptions and describe the class of algorithms that our analysis covers.
\subsection{Setup and algorithms}
Recall that we aim to solve convex \textit{composite optimization problems} with the
format of~\eqref{Eq:Problem}, repeated here for convenience:
\begin{align}
  \label{Eq:ProblemRepeated}
  \underset{x \in \mathbb{R}^n}{\text{minimize}} \,\,\,
  f(x) := g(x) + h(x)\,.
\end{align}
All of our results assume the following:
\begin{Assumption}[Assumptions on the problem]
  \label{Ass:OptimizationProblem}
  \hfill    
            
  \medskip  
  \noindent
  \begin{itemize}

    \item The function $h:\mathbb{R}^n \to \mathbb{R} \cup \{+\infty\}$ is
      closed, proper, and convex.

    \item The function $g\,:\, \mathbb{R}^{n} \to \mathbb{R}$ is convex and
      differentiable,
      and its gradient $\nabla g\,:\, \mathbb{R}^n \to \mathbb{R}^n$ is
      Lipschitz-continuous with constant $L > 0$, that is,
      \begin{equation}
        \label{Eq:LipschitzContinuity}
        \big\|\nabla g(y) - \nabla g(x)\big\|_2 \leq L\big\|y - x\big\|_2\,,
      \end{equation}
      for all $x$, $y \in \mathbb{R}^n$, where $\|\cdot\|_2$ stands for the
      standard Euclidean norm. 

    \item The set of optimal solutions of~\eqref{Eq:ProblemRepeated} is nonempty: 
      \begin{equation}
        \label{Eq:NonemptySetSolutions}
        X^\star := \big\{x \in \mathbb{R}^n\, :\, f(x) \leq
        f(z),\,\,\text{\emph{ for all} $z \in \mathbb{R}^n$}\big\} \neq
        \emptyset\,.
      \end{equation}
  \end{itemize}
\end{Assumption}
The above assumptions are standard in the analysis of proximal gradient algorithms and are
actually required for convergence to an optimal solution from an arbitrary
initialization \cite{beck2017first,bertsekas2015convex}. 

A consequence of~\eqref{Eq:LipschitzContinuity} that we will
often use in our results is that~\cite[Lem.\
1.2.3]{Nesterov04-IntroductoryLecturesConvexOptimization}
\begin{equation}
  \label{Eq:LipschitzContinuityAlt}
  g(y) \leq g(x) + \nabla g(x)^\top (y-x) + \frac{L}{2}\|y - x\|_2^2\,,
\end{equation}
for any $x$, $y \in \mathbb{R}^n$. Also, as $h$ is closed, proper, and convex,
the function $z \mapsto h(z) + (1/2)\|z - y\|_2^2$ is coercive, which implies
that the set defining the approximate proximal set
in~\eqref{Eq:ProxApproximate} is nonempty for all $\epsilon \geq 0$, and $y \in \mathbb{R}^n$:
\begin{equation*}
    \text{prox}_{h}^{\epsilon}(y) 
    := 
    \Big\{
      x \in \mathbb{R}^n\,:\,
      h(x) + \frac{1}{2}\|x -y\|_2^2 \leq \epsilon 
      + 
      \underset{z}{\inf}\,\,
      h(z) + \frac{1}{2}\|z -
      y\|_2^2 
    \Big\}
    \neq \emptyset
    \,.
\end{equation*}
When $\epsilon > 0$, this set may contain more than a single element, which
results in several possible instances of the accelerated approximate proximal gradient~\eqref{Eq:AcceleratedPG},
\begin{equation}
    \label{Eq:AcceleratedPGRepeated}
    \begin{split}
    &y^k = x^k +\beta_k(x^k-x^{k-1}),
    \\
    &x^{k+1} 
    \in 
    \text{prox}_{s_{k} h}^{\epsilon_{2}^{k}}
    \Big[y^{k} - s_{k}\big(\nabla g(y^{k})+\epsilon_1^{k}\big)\Big]\,,
    \end{split}
\end{equation}
whenever there exists a $k$ for which $\epsilon_2^{k} > 0$. However, as we
establish bounds on function values [i.e., $f(x^k)$], this ambiguity does not
affect our results. By setting $\beta_k = 0$, \eqref{Eq:AcceleratedPGRepeated} reduces to the basic approximate proximal gradient scheme, i.e.,
\begin{equation}
    \label{Eq:PGDApproximateRepeated}
    x^{k+1} 
    \in 
    \text{prox}_{s_{k} h}^{\epsilon_{2}^{k}}
    \Big[x^{k} - s_{k}\big(\nabla g(x^{k})+\epsilon_1^{k}\big)\Big].
\end{equation}
\subsection{Error models and assumptions}
In what follows we consider two models for the gradient error $\epsilon_1$.
\begin{Model}[Absolute Error Model]
\label{ErrModel1}
Under this model, each evaluation of the gradient of $g$ at a point $x$ is subject to additive noise $\epsilon_1$ whose magnitude is independent from the entries of $x$. Specifically, each evaluation of the gradient of $g$ in~\eqref{Eq:ProblemRepeated} yields
\begin{equation}
    \label{Eq:Model1}
    \nabla g^{\epsilon_1}(x) = \nabla g (x)+ \epsilon_1
\end{equation}
where
\begin{equation}
    |\epsilon_1| \leq \delta \mathbf{1}_n. 
    \label{Eq:Model1UpperDelta}
\end{equation}
$\mathbf{1}_n$ is an $n$-dimensional vector of ones, and $\delta$ is a positive scalar. $|.|$ stands for the vector componentwise absolute value. This can be used, for example, to model fixed-point representation approximations (see Appendix~\ref{AppendixFixedPoint}).
\end{Model}
\begin{Model}[Relative Error Model]
\label{ErrModel2}
Under this model, each evaluation of the gradient of $g$ at a point $x$ is subject to additive noise $\epsilon_1$ whose magnitude is proportional to the magnitude of the gradient $|\nabla g(x)|$. Specifically, the gradient of $g$ in \eqref{Eq:ProblemRepeated} is approximated by
\begin{equation}
    \label{Eq:Model2}
    \nabla g^{\epsilon_1}(x) = \nabla g (x)+ \epsilon_1
\end{equation}
where
\begin{align}
    & |\epsilon_1| \leq \delta|\nabla g(x)|. \label{Eq:Model2UpperDeltaNG}
\end{align} 
$\delta$ is a positive scalar, and $|.|$ stands for the vector componentwise absolute value. This can be used, for example, to model floating-point representation approximations (see Appendix \ref{AppendixFloatingPoint}).
\end{Model}
In both models the parameter $\delta$ is known as the machine precision.

For each of the above models, our analysis assumes two different scenarios:
\begin{enumerate}
  \item The sequences of errors $\{\epsilon_1^k\}_{k\geq 1}$ and
    $\{\epsilon_2^{k}\}_{k\geq 1}$ are deterministic, or

  \item The sequences of errors $\{\epsilon_1^k\}_{k\geq 1}$ and
    $\{\epsilon_2^{k}\}_{k\geq 1}$ are random, in which case we use
    $\epsilon_{1_\Omega}^k$ and $\epsilon_{2_\Omega}^k$ to denote the respective
    random vectors/variables of errors at iteration $k$, where $\Omega$ denotes the sample space of a given probability measure.
\end{enumerate}
In scenario 2, the sequences $\{x^k\}_{k \geq 1}$ and
$\{y^k\}_{k \geq 1}$ become random as well. And we also use $x^k_{_\Omega}$ and
$y^k_{_\Omega}$ to denote the respective random vectors at iteration $k$.
We make the following assumption in this case:
\begin{Assumption}
  \label{Assum:RandomVectors}
  In scenario 2, we assume that each random vector $\epsilon_{1_\Omega}^k$, for
  $k \geq 1$, satisfies
  \begin{subequations}
    \label{Eq:AssumptionsNoise}
    \begin{align}
      &\mathbb{E}\big[
      \epsilon_{1_\Omega}^k
      \,\big\vert\,
      \epsilon_{1_\Omega}^1, \ldots, \epsilon_{1_\Omega}^{k-1}
      \big]
      =
      \mathbb{E}\big[\epsilon_{1_\Omega}^k\big]
      =
      0\,,
      \label{Eq:AssumptionRandomZeroMean}
      \\
      &\mathbb{P}\big(|\epsilon_{1_\Omega j}^k| \leq \delta\big) = 1\,,\quad\quad\quad\quad
      \text{\emph{for all} $j = 1, \ldots, n$,}\quad(\text{Model 1})
      \label{Eq:AssumptionRandomBounded}
      \\
      &\mathbb{P}\big(|\epsilon_1| \leq \delta|\nabla g(x)|) = 1,\quad\quad
      \text{\emph{for all} $j = 1, \ldots, n$,}\quad(\text{Model 2})
      \label{Eq:AssumptionRandomBounded2}
      \\
      &\mathbb{E}\big[
      {\epsilon_{1_\Omega}^k }^\top x_{_\Omega}^k 
      \,\big\vert\,
      \epsilon_{1_\Omega}^1, \ldots, \epsilon_{1_\Omega}^{k-1},\,
      x_{1_\Omega}^1, \ldots, x_{1_\Omega}^{k-1}
      \big]
      =
      \mathbb{E}\big[
      {\epsilon_{1_\Omega}^k }^\top x_{_\Omega}^k 
      \big]
      =
      0\,,
      \label{Eq:AssumptionRandomInnerProduct}
    \end{align}
  \end{subequations}
  where  $\epsilon_{1_\Omega j}^k$ in~\eqref{Eq:AssumptionRandomBounded} denotes
  the $j$-th entry of $\epsilon_{1_\Omega}^k$, and $\delta > 0$ is the machine precision. 
\end{Assumption}
The first assumption, \eqref{Eq:AssumptionRandomZeroMean}, states that
$\epsilon_{1_\Omega}^k$ is independent from past realizations and has zero mean.
The second assumption, \eqref{Eq:AssumptionRandomBounded}, states that the
absolute value of each entry of $\epsilon_{1_\Omega}^k$ is bounded by $\delta$
almost surely. The third assumption, \eqref{Eq:AssumptionRandomBounded2}, states that the
absolute value of each entry of $\epsilon_{1_\Omega}^k$ is bounded by the corresponding entry of the gradient vector $\nabla g(x)$ scaled by $\delta$, almost surely. Finally, the fourth assumption, \eqref{Eq:AssumptionRandomInnerProduct}, states that the gradient error $\epsilon_{1_\Omega}^k$ and the iterate $x_{_\Omega}^k$ are orthogonal, and their inner product is conditional mean independent from past realizations. 

\begin{Assumption}
Let $\{x^k\}$ denote the sequence produced by~\eqref{Eq:AcceleratedPGRepeated} or~\eqref{Eq:PGDApproximateRepeated}. We define the residual error vector at iteration $k$ as
\begin{equation}
\label{Eq:ResidualVec}
r^k = x^k - \overline{x}^{k},   
\end{equation}
where $\overline{x}^{k}$ stands for the proximal error-free iterate
\begin{equation}
  \overline{x}^{k+1} 
  :=
  \text{prox}_{s h}
  \Big(x^k - s\big(\nabla g(x^k) + \epsilon_1^{k}\big)\Big).
\end{equation}
In scenario 2,  we assume
  \label{Assum:RandomVectors2}
  \begin{subequations}
    \label{Eq:AssumptionsResidual}
    \begin{align}
      &
      \mathbb{E}\big[
      r_{_\Omega}^k
      \,\big\vert\,
      r_{_\Omega}^1, \ldots, r_{_\Omega}^{k-1}
      \big]
      =
      \mathbb{E}\big[r_{_\Omega}^k\big]
      =
      0\,,
      \label{Eq:ResidAssumptionRandomZeroMean}
      \\
      &
      \mathbb{E}\big[
      {r_{_\Omega}^k }^\top x_{_\Omega}^k 
      \,\big\vert\,
      r_{_\Omega}^1, \ldots, r_{_\Omega}^{k-1},\,
      x_{1_\Omega}^1, \ldots, x_{1_\Omega}^{k-1}
      \big]
      =
      \mathbb{E}\big[
      {r_{_\Omega}^k }^\top x_{_\Omega}^k 
      \big]
      =
      0\,.
      \label{Eq:ResidAssumptionRandomInnerProduct}
    \end{align}
  \end{subequations}
Note that Lemma \ref{Lem:BackwardProxError}, stated in the appendix, bounds the norm of the residual vector $\big\|r^k\big\|_2$ as a function of $\epsilon_{2}^k$; therefore, bounding $\epsilon_{2}^k$ implies bounding $\big\|r^k\big\|_2$. 
\end{Assumption}
Before we proceed into the main theoretical results, we need the following Lemma which shows the effect of the error model on the geometry of the problem.
\begin{Lemma}
\label{Lem:LipsChitz+Errors}
If the exact gradient of $g$ is \textit{Lipschitz} with constant $L$, then the approximate gradient of $g$ is \textit{Lipschitz} with constant $L$ for absolute errors (Model \ref{ErrModel1}) and with constant $(1+\delta)L$ for relative errors (Model \ref{ErrModel2}).
\end{Lemma}
\begin{proof}
See Section \ref{Subsec:ProofLipsChitz+Errors}.
\end{proof}
\subsection{Approximate proximal gradient}
In this section, we consider the approximate proximal gradient algorithm in~\eqref{Eq:PGDApproximateRepeated}, i.e., without
acceleration. We start by considering deterministic error sequences
$\{\epsilon_1^k\}_{k\geq 1}$ and $\{\epsilon_2^{k}\}_{k\geq 1}$, as in Assumption \ref{Ass:OptimizationProblem}, and then we consider the
case in which these sequences are random, as in
Assumption~\ref{Assum:RandomVectors}.

\subsubsection{Deterministic errors}
Our first result provides a bound on the ergodic convergence of the sequence
of function values, and decouples the contribution of the errors in the
computation of gradient, $\epsilon_1^k$, and in the computation of the proximal operator, $\epsilon_2^k$ and $r^k$.

\begin{Theorem}[Approximate proximal gradient, deterministic errors]
  \label{prop:1}
  Consider problem~\eqref{Eq:ProblemRepeated} and let
  Assumption~\ref{Ass:OptimizationProblem} hold. Then, for arbitrary error
  sequences $\{\epsilon_1^k\}_{k\geq 1}$ and $\{\epsilon_2^{k}\}_{k\geq 1}$, the
  sequence generated by approximate proximal gradient in~\eqref{Eq:PGDApproximateRepeated}
  with constant stepsize $s_k := s$, for all $k$,
  satisfies
  \begin{equation}
    \label{Eq:Theorem1}
    \begin{split}
    &f\bigg(\frac{1}{k+1}\sum_{i=0}^{k}x^{i+1}\bigg)  
    -
    f(x^\star) 
    \leq 
    \frac{1}{k+1}
    \Bigg[
      \sum_{i=0}^{k}\epsilon_2^{i} 
      + 
      \sum_{i=0}^{k}\Big(\epsilon_1^{i}-\frac{1}{s}r^{i+1}\Big)^\top(x^\star-x^{i+1})
    \\&+ 
    \frac{1}{2s}\norm{x^\star-x^0}_{2}^2\Big]
    -
    \frac{1}{k+1}\Big[\frac{1}{2s}\sum_{i=0}^{k}\big\|r^{i+1}\big\|_2^2
    +
    \frac{1}{2s}\big\|x^\star-x^{k+1}\big\|_2^2 \Bigg],
    \end{split}
\end{equation}
  where $x^\star$ is any solution of~\eqref{Eq:ProblemRepeated}, $r^i$ is the residual vector associated with error $\epsilon_2^i$, defined by \eqref{Eq:ResidualVec},  $s \leq 1/L$ for absolute error Model \ref{ErrModel1}, and  $s \leq 1/(L+\delta)$ for relative error Model \ref{ErrModel2}.
\end{Theorem}
\begin{proof}See Section \ref{Subsec:Proofprop1}.\end{proof}
This result implies that the well-known $O(1/k)$ convergence rate for the gradient method without errors still holds when both $\{\epsilon_2^{k}\}$ and $\{(\epsilon_1^{k}-\frac{1}{s}r^{k+1})^\top (x^\star-x^k)\}$ are summable. Note that a faster convergence of these two errors will not improve the convergence rate but will yield better coefficients.

To obtain a bound independent of the particular sequence $\{x^k\}$, we can
apply Cauchy-Schwarz's inequality to the second term of the right-hand side
of~\eqref{Eq:Theorem1} followed by quasi-F\'ejer's bound (see Theorem~\ref{Thm:QuasiFejerPG} in the appendix):
\begin{Corollary} 
  \label{corol:1} 
Assume that for $k \geq k_0$, we have $\epsilon_2^{k} \leq c_2 \norm{{x}^{k+1}-{x}^{k}}_2 \leq c_2 \rho$ and $\epsilon_2^{k} \leq c_1 \norm{\nabla g(x^{k+1})-\nabla g(x^{k})}_2$ where $\rho$, $c_1$, $c_2$ and $k_0$ are constants. Then, for any $x^\star \in X^\star$ and $k\geq k_0$, under the same conditions as Theorem~\ref{prop:1}, the sequence generated by
approximate proximal gradient in~\eqref{Eq:PGDApproximateRepeated} satisfies
\begin{equation}
  \label{Eq:Corollary1.1}
  \begin{split}
  &f\bigg(\frac{1}{k+1}\sum_{i=0}^{k}x^{i+1}\bigg)  
  -
  f(x^\star) 
  \leq 
  \frac{1}{k+1}
  \Bigg[
    \sum_{i=0}^{k}\epsilon_2^{i} 
+ \sum_{i=1}^{k}\Bigg(\norm{\epsilon_1^{i}}_{2}+\sqrt{\frac{2\epsilon_2^{i}}{s}}\Bigg)\norm{x^\star - x^{0}}_2
  \\&+\frac{1}{2s}\norm{x^\star-x^0}_{2}^2\Bigg]
  -\frac{1}{k+1}\Bigg[\frac{1}{2s}\sum_{i=0}^{k}\big\|r^{i+1}\big\|_2^2 
  +\frac{1}{2s}\big\|x^\star-x^{k+1}\big\|_2^2 \Bigg]
  \\&+\frac{1}{k+1}\sum_{i=1}^{k}\Bigg(\norm{\epsilon_1^{i}}_{2}+\sqrt{\frac{2\epsilon_2^{i}}{s}}\Bigg) \big(\sum_{j=1}^iE^{j}+iC_\rho\big),
  \end{split}
\end{equation}    
where $E^j = \norm{r^{j}}_2+s \norm{\epsilon_1^{j-1}}_2$ is an absolutely summable sequence and $C_\rho = \sqrt{2Lc_2 \rho} + c_1 \rho$.  For small perturbations and very small suboptimality stopping criterion, i.e., $\rho \approx 0$\footnote{$C_\rho = 0$ if the optimum $x^\star$ is reached.}, \eqref{Eq:Corollary1.1} can be approximated by
\begin{equation}
\label{Eq:CorollaryPGDDeterministicApprox}
\begin{split}
  &f\bigg(\frac{1}{k+1}\sum_{i=0}^{k}x^{i+1}\bigg)  
  -
  f(x^\star) 
  \lessapprox 
  \frac{1}{k+1}
  \Bigg[
    \sum_{i=0}^{k}\epsilon_2^{i} 
+ \sum_{i=1}^{k}\Bigg(\norm{\epsilon_1^{i}}_{2}+\sqrt{\frac{2\epsilon_2^{i}}{s}}\Bigg)\norm{x^\star - x^{0}}_2
  \\&+\frac{1}{2s}\norm{x^\star-x^0}_{2}^2\Bigg],
\end{split}
\end{equation}    
where we have dropped the second order error terms. This result holds with $s \leq 1/L$ for absolute error Model \ref{ErrModel1}, and  $s \leq 1/(L+\delta)$ for relative error Model \ref{ErrModel2}.
\end{Corollary}
\begin{proof}See Section \ref{Subsec:Proofcorol1}.\end{proof}
Notice that the use of Cauchy-Schwarz's inequality followed by quasi-F\'ejer's inequality leads to a looser bound than the original result \eqref{Eq:Theorem1}. Nonetheless, the $O(1/k)$ convergence rate is still guaranteed with weaker summability assumptions of $\{\epsilon_2^{k}\}_{k\geq 1}$ and $\{\norm{\epsilon_1^{k}}_2\}_{k\geq 1}$. If we set both errors to zero for all $k\geq 1$, we recover the error-free optimal upper bound $\frac{1}{2sk}\norm{x^\star-x^0}_{2}^2$ \cite{beck2017first}.


The decoupled error terms in  \eqref{Eq:Theorem1} and~\eqref{Eq:Corollary1.1} completely eliminate any error redundancy. Although \eqref{Eq:Corollary1.1} shares the same error-free term with \eqref{schmidt1}, our proposed bound gives a better approximation of the discrepancy caused by perturbations, and consequently we obtain better error terms. For instance, let us consider the case where both proximal and gradient errors decrease as $O({1}/{k})$ (i.e., nonsummable). Then Corollary~\ref{corol:1} yields an overall convergence rate of $O(\log k/{k})$ which is less conservative than what would have been obtained from \eqref{schmidt1}, i.e, $O(\log^2 k/{k})$. Additionally, as a necessary condition for convergence, we only require the partial sums $\sum_{i=1}^{k}\epsilon_2^{i}$ and $\sum_{i=1}^{k}\norm{\epsilon_1^{i}}_2$ to be in $o(k)$ as compared to the stronger condition $o(\sqrt{k})$ of \cite{schmidt2011convergence}. 

\subsubsection{Random errors}
Let us now consider the case in which $\epsilon_1^k$, $\epsilon_2^{k}$ and
therefore $x^k$, are random, and let $\epsilon_{1_\Omega}^k$, $\epsilon_{2_\Omega}^k$ and $x_{\Omega}^k$ be the corresponding random variables/vectors.   

\begin{Theorem} [\textbf{Random errors}]\label{prop:2}
Consider problem~\eqref{Eq:ProblemRepeated} and let Assumption~\ref{Ass:OptimizationProblem} hold. Assume that the gradient error $\{\epsilon_{1_\Omega}^k\}_{k\geq 1}$ and residual proximal error $\{r_{_\Omega}^{k}\}_{k\geq 1}$ sequences satisfy Assumption~\ref{Assum:RandomVectors}, \ref{Assum:RandomVectors2} and $\mathbb{P}\big(\epsilon_{2_\Omega}^k \leq \varepsilon_0\big) = 1$, for all $k > 0$, and for some $\varepsilon_0 \in \mathbb{R}$. Let $\{x_{\Omega}^i\}$ denote a sequence generated by the approximate proximal gradient algorithm in~\eqref{Eq:PGDApproximateRepeated} with constant stepsize $s_k = s$, for all $k$. Assume that $\norm{x_{_\Omega}^k-x_{_\Omega}^\star}_{2}^2 \leq \norm{x_{_\Omega}^0-x_{_\Omega}^\star}_{2}^2$ hold with probability $p$, for all $k$. Then, for any $\gamma >0$,
\begin{equation}
\begin{split}
    &f\bigg(\frac{1}{k}\sum_{i=1}^{k}x_{_\Omega}^{i}  \bigg)-f(x^\star)
\leq \frac{1}{k}\sum_{i=1}^{k}\epsilon_{2_\Omega}^{i} + \frac{\gamma}{\sqrt{k}}\Bigg(\sqrt{n}M_{\nabla g}|\delta|+\sqrt{\frac{2\varepsilon_0}{s}}\Bigg)\norm{x^\star-x^0}_{2}
\\&+\frac{1}{2sk}\norm{x^\star-x^0}_{2}^2,
\end{split}
\label{Eq:Theorem2}
\end{equation}
with probability at least  $p^k\big(1-2\exp(-\frac{\gamma^2}{2})\big)$, where $x^\star$ is any solution of~\eqref{Eq:ProblemRepeated},  $M_{\nabla g} = 1$, $s \leq 1/L$ for absolute error Model \ref{ErrModel1}, and $M_{\nabla g} = \underset{i \in \mathbb{N}_{+}} \sup\bigg\{\norm{\nabla g(x^i)}_{\infty}\bigg\}$, $s_k := s \leq 1/(L+\delta)$ for relative error Model \ref{ErrModel2}. 
\end{Theorem}
\begin{proof}See Section \ref{Subsec:Proofprop2}\end{proof}
For large scale problems,\footnote{And for same levels of error magnitudes $\delta$ and $\varepsilon_0$.} we typically have $n \gg \frac{1}{s} \geq L$; therefore, we obtain the following approximated bound
\begin{equation}
    f\bigg(\frac{1}{k}\sum_{i=1}^{k}x_{_\Omega}^{i}  \bigg)-f(x^\star)
\lessapprox \frac{1}{k}\sum_{i=1}^{k}\epsilon_{2_\Omega}^{i} + \gamma M_{\nabla g}\sqrt{\frac{n}{k}}|\delta|\norm{x^\star-x^0}_{2}+
    \frac{1}{2sk}\norm{x^\star-x^0}_{2}^2
    \label{prop2approx},
\end{equation}
with the same probability. In the absence of computational errors, \eqref{Eq:Theorem2} coincides with the results of Theorem~\ref{prop:1} and Corollary~\ref{corol:1} which reduce to the deterministic noise-free convergence upper bound, i.e, $\frac{1}{2sk}\norm{x^\star-x^0}_{2}^2$. With exact proximal computation, i.e., $\epsilon_{2_\Omega}^k = 0$, for all $k$, if we let the machine precision $\delta$ to decrease as $O({1}/{k^{0.5+\varsigma}})$, i.e, progressively increase computation accuracy, then we retrieve back the optimal convergence rate $O({1}/{k})$. In order to recover the same convergence rate for the inexact proximal case, we also need the sum of the ensemble means $E(\epsilon_{2_\Omega}^{i})$ to decrease as $O({1}/{k^{1+\varsigma}})$ which is a weaker requirement than $O({1}/{k^{2+\varsigma}})$ in  \cite{schmidt2011convergence}. This result also suggests that a slower $O({1}/{\sqrt{k}})$ convergence rate (same as the noise-free subgradient method) is achieved when the sequence of ensemble means $\{E(\epsilon_{2_\Omega}^{k})\}$ is summable for all centered and bounded sequences $\{\epsilon_{1_\Omega}^{k}\}$, and consequently the proximal error is the main contributor to any divergence from the optimal set $X^\star$.

Notice that for a fixed machine precision $\delta$ and probability parameter $\gamma$ we obtain a predictable error residual rather than unpredictable  running error terms as in Theorem~\ref{prop:1}, Corollary~\ref{corol:1} or \eqref{schmidt1} without making any summability assumptions on $\{\norm{\epsilon_{1_\Omega}^{k}}_2\}$ (as in Corollary \ref{corol:1}) or $\{\epsilon_{1_\Omega}^{k}\}$ in general. 

Morevover, the effect of the dimension of the problem $n$ shows up explicitly in the convergence bound \eqref{Eq:Theorem2} which is missing in Theorem~\ref{prop:1} and Corollary \ref{corol:1} as well as in the original work of \cite{schmidt2011convergence}. The latter suggests that using progressively sparser gradient vectors\footnote{As is the case in \textit{proximal gradient} algorithm when applied to LASSO.} can potentially accelerate the convergence speed but never faster than the optimal (limit) speed of $O({1}/{k})$. Overall, better parameter design results in smaller constants, but not necessarily in exceeding the optimal convergence rate.

The following result applies if we assume statistical stationarity\footnote{Whose ensemble mean and variance are time-invariant.}. 
\begin{Theorem} [\textbf{Random stationary errors}]\label{prop:3}
Consider problem~\eqref{Eq:ProblemRepeated}, let Assumption~\ref{Ass:OptimizationProblem} hold and assume that the rounding error $\{\epsilon_{1_\Omega}^k\}_{k\geq 1}$ and residual error $\{r_{_\Omega}^{k}\}_{k\geq 1}$ sequences satisfy Assumption~\ref{Assum:RandomVectors} and that the proximal computation error is upper bounded, i.e $\mathbb{P}\big(\epsilon_{2_\Omega}^k \leq \varepsilon_0\big) = 1$  for all $k \geq 1$ and stationary with constant mean $E(\epsilon_{2_\Omega})$. Let $\{x_{\Omega}^i\}$ denote a sequence generated by the approximate proximal gradient algorithm in~\eqref{Eq:PGDApproximateRepeated} with constant stepsize $s_k = s$, for all $k$. Assume that $\norm{x_{_\Omega}^k-x_{_\Omega}^\star}_{2}^2 \leq \norm{x_{_\Omega}^0-x_{_\Omega}^\star}_{2}^2$ hold with probability $p$, for all $k$. Then, for any $\gamma >0$,
\begin{equation}
    f\bigg(\frac{1}{k}\sum_{i=1}^{k}x_{_\Omega}^{i}  \bigg)-f(x^\star) \leq E\big(\epsilon_{2_\Omega}\big) + \frac{\gamma}{\sqrt k}\bigg(\frac{\varepsilon_0}{2} +\sqrt{n}M_{\nabla g}|\delta|\norm{x^\star-x^0}_{2}\bigg)+
    \frac{1}{2sk}\norm{x^\star-x^0}_{2}^2
    \label{Eq:Theorem3},
\end{equation}
with probability at least $p^k\big(1-4\exp(-\frac{\gamma^2}{2})\big)$, where $x^\star$ is any solution of~\eqref{Eq:ProblemRepeated},  $M_{\nabla g} = 1$, $s \leq 1/L$ for absolute error Model \ref{ErrModel1}, and $M_{\nabla g} = \underset{i \in \mathbb{N}_{+}} \sup\bigg\{\norm{\nabla g(x^i)}_{\infty}\bigg\}$, $s_k := s \leq 1/(L+\delta)$ for relative error Model \ref{ErrModel2}.
\end{Theorem}
\begin{proof}See Section \ref{Subsec:ProofEq:Theorem3}\end{proof}
Once again, if both errors are forced to zero in \eqref{Eq:Theorem3} then the optimal convergence rate is obtained as in Theorem~\ref{prop:1} and Theorem~\ref{prop:2}. \eqref{Eq:Theorem3} also implies that we obtain a worst case convergence rate of $O(1)$, i.e, convergence up to a predicted constant residual $E(\epsilon_{2_\Omega})$.

Notice that if we further assume that the proximal error sequence is stationary with zero mean, i.e., $E(\epsilon_{2_\Omega}) = 0$, then we can recover the $O({1}/{\sqrt{k}})$ convergence rate. In practice, however, the latter assumption can be very misleading and a better approach would be to compensate for the algorithm to reduce or completely eliminate the positive bias term $E(\epsilon_{2_\Omega}) \geq 0$. Notice that, in the long run, a faster decay of  $\{\epsilon_{2_\Omega}^{k}\}$ will not improve the $O({1}/{\sqrt{k}})$ rate but will achieve better error coefficients in the short run. In other words, choosing progressively decreasing error sequence $\{\epsilon_{2_\Omega}^{k}\}$ improves (i.e., reduces) the coefficient of $\frac{1}{\sqrt{k}}$ and,  because of the nonnegativity of the latter it is practically impossible to design a zero-mean proximal error; therefore, the algorithm does not converge in the strict sense of the word but only converges to a neighborhood set around the optimal set $X^\star$ whose radius is determined by the residual $E(\epsilon_{2_{\Omega}})$.
\subsection{Accelerated Approximate PG}
\subsubsection{Deterministic errors} We now analyze the effect of computational inaccuracy on the modified version of the inexact PG algorithm, viz. the inexact accelerated PG. In what follows, we establish upper bounds on the convergence of the accelerated PG in the presence of deterministic errors in the computation of the gradient as well as in the proximal operation step.
\begin{Theorem} [\textbf{Accelerated with deterministic errors}]\label{Theorem1} Consider the approximate accelerated PG in~\eqref{Eq:AcceleratedPG} with constant stepsize $s_k := s$, arbitrary error   sequences $\{\epsilon_1^k\}_{k\geq 1}$ and $\{\epsilon_2^{k}\}_{k\geq 1}$, and with parameter $\beta_k = (\alpha_{k-1} - 1)/\alpha_k$, where $\alpha_k$ satisfies
\begin{itemize}
    \item $\alpha_k \geq 1 \quad \forall \quad k > 0$ and $\alpha_0 = 1$
    \item $\alpha_k^2-\alpha_k = \alpha_{k-1}$
    \item $\{\alpha_k\}_{k=0}^\infty$ is an increasing sequence and proportional to $k$ ($O(k)$)
\end{itemize}
Also, assume that $s \leq 1/L$, under error Model \ref{ErrModel1}, and  $s_k := s \leq 1/(L+\delta)$, under error Model \ref{ErrModel2}. Then, the sequence $\{x^k\}$ produced by ~\eqref{Eq:AcceleratedPG} satisfies
\begin{equation}
   f(x^{k+1})-f(x^\star) \leq \frac{1}{\alpha_k^2} \bigg[\sum_{i=0}^k \alpha_i^2\epsilon_2^{i} + \sum_{i=0}^k \alpha_i \bigg({\epsilon_1^{i}} -\frac{1}{s} r^{i+1}\bigg)^\top u^{i+1}+\frac{1}{2s}\norm{x^\star-x^{0}}_{2}^2.  \bigg]
   \label{Eq:Theorem4},
\end{equation}
where $u^{i} = x^\star- x_{}^i + (1-\alpha_{i-1})(x_{}^{i} -x^{i-1})$, $x^\star$ is any solution of~\eqref{Eq:ProblemRepeated}.
\end{Theorem}
\begin{proof}See Section \ref{Subsec:ProofEq:Theorem1}\end{proof}
To obtain a bound independent of the particular sequence $\{u^k\}$, we can
apply Cauchy-Schwarz's inequality to the second term of the right-hand side
of~\eqref{Eq:Theorem1} followed by F\'ejer's bound (see Theorem~\ref{Thm:Fejer} in the appendix):
\begin{Corollary}\label{corol:2}
Under the same conditions as Theorem \ref{Theorem1}, the sequence $\{x^k\}$ produced by ~\eqref{Eq:AcceleratedPG} satisfies  
\begin{equation}
\begin{split}
   &f(x^{k+1})-f(x^\star) \leq \frac{1}{\alpha_k^2} 
   \bigg[\sum_{i=0}^k \alpha_i^2\epsilon_2^{i} + \sum_{i=0}^k \alpha_i\norm{u^{i+1}}_2 \Big(\norm{\epsilon_1^{i}}_2  
   \\&+\sqrt{\frac{2\epsilon_2^{i}}{s}} \Big)+\frac{1}{2s}\norm{x^\star-x^{0}}_{2}^2\bigg],
\end{split}
\label{corol2}
\end{equation}
where $u^{i} = x^\star- x_{}^i + (1-\alpha_{i-1})(x_{}^{i} -x^{i-1})$, $x^\star$ is any solution of~\eqref{Eq:ProblemRepeated}.
\end{Corollary}
\begin{proof}See Section \ref{Subsec:Proofcorol2}\end{proof}
Ignoring second order error terms (for small square summable perturbations and very small suboptimality stopping criterion, i.e., $\rho \approx 0$), \eqref{corol2} can be approximated by
\begin{equation}
\begin{split}
   &f(x^{k+1})-f(x^\star) \lessapprox \frac{1}{\alpha_k^2} 
   \bigg[\sum_{i=0}^k \alpha_i^2\norm{\epsilon_2^{i}}_2 + \sum_{i=0}^k \alpha_i \Big(\norm{\epsilon_1^{i}}_2  
   +\sqrt{\frac{2\epsilon_2^{i}}{s}} \Big)\norm{x^{0}-x^\star}_2\\&+\frac{1}{2s}\norm{x^\star-x^{0}}_{2}^2\bigg].
\end{split}
\label{corol2:2}
\end{equation}

Notice that if we trivially choose $\beta_k = 0$ we recover back the nonaccelerated basic scheme. In the noise-free case, \eqref{Eq:Theorem1} and~\eqref{corol2} reduce to $\frac{1}{2s\alpha_k^2}\norm{x^\star-x^0}_{2}^2$, which coincides with the convergence rate of the accelerated proximal gradient algorithm \cite[Thm.\ 10.34]{beck2017first}, i.e, $O({1}/{k^2})$ if $\alpha_k$ is in the order of $O(k)$.

For non-zero errors, the latter convergence rate still holds when both $\{\alpha_{k}^2\epsilon_2^{k}\}$ and $\{\alpha_{k}\norm{\epsilon_1^{k}}_2\}$ are summable, which is a stronger requirement than what was required in the basic (nonaccelerated) case since at iteration $k$ the errors $\epsilon_1$ and $\epsilon_2$ get amplified by factors proportional to $O(k)$ and $O(k^2)$, respectively.  As a result, if $\alpha_k \propto k$, then a sufficient condition for convergence is for  $\norm{\epsilon_1}_2$ to decrease as $O({1}/{k^{2+\varsigma}})$ and  we only require $\sqrt{\epsilon_2}$ to decrease as $O({1}/{k^{1.5+\varsigma}})$\footnote{Equivalently $\epsilon_2 \propto O(\frac{1}{k^{3+\varsigma}})$} instead of $O({1}/{k^{2+\varsigma}})$ as in \cite[Proposition.\  2]{schmidt2011convergence} with $\varsigma > 0$. Also, we obtain a better factor in general. For instance, if both $\norm{\epsilon_1}_2$ and $\epsilon_2$ decrease as $O({1}/{k^2})$ then an overall convergence rate of $O({\log k}/{k^2})$ is obtained using \eqref{corol2} as compared to $O({\log ^2 k}/{k^2})$ in \cite[Proposition.\  2]{schmidt2011convergence}.

Finally, when both error sequences are nonsummable and inaccessible (cannot be measured) at iteration $k$, then none of the above obtained bounds nor the results of \cite{schmidt2011convergence} can give an accurate estimate of the convergence or divergence rates because of the dependence on error realizations.

\subsubsection{Random errors} The following result gives an estimate of the convergence rate when both errors are stochastic and bounded following a probabilistic analysis approach.
\begin{Theorem} [\textbf{Accelerated with random errors}]\label{Theorem2}
Consider problem~\eqref{Eq:ProblemRepeated} and let Assumption~\ref{Ass:OptimizationProblem} hold. Suppose that the rounding error $\{\epsilon_{1_\Omega}^k\}_{k\geq 1}$ and residual error $\{r_{_\Omega}^{k}\}_{k\geq 1}$ sequences satisfy Assumptions~\ref{Assum:RandomVectors} and \ref{Assum:RandomVectors2}, respectively. Let the norm of the iterative difference $\norm{x^{k}-x^{k-1}}_2$ be summable. Then, for all $\gamma > 0$, the the sequence generated by the approximate accelerated PG in~\eqref{Eq:AcceleratedPG} with constant stepsize $s_k := s$ and with the following choices:
\begin{itemize}
    \item $\beta_k = \frac{\alpha_{k-1}-1}{\alpha_k}$
    \item $\alpha_k \geq 1 \quad \forall \quad k > 0$ and $\alpha_0 = 1$
    \item $\alpha_k^2-\alpha_k = \alpha_{k-1}$
    \item $\{\alpha_k\}_{k=0}^\infty$ increases as $o(k)$
\end{itemize}
for all $k$, satisfies  
\begin{equation}
   f(x_{_\Omega}^{k+1})-f(x^\star) \leq \frac{1}{\alpha_k^2} \bigg[S_{\epsilon_{2_\Omega}}+S_{r_{_\Omega}}+S_{\epsilon_{1_\Omega}} +\frac{1}{2s}\norm{x^\star-x^0}_{2}^2\bigg]
   \label{Eq:Theorem5},
\end{equation}
where 
\begin{align}
    S_{\epsilon_{2_\Omega}}&=\mathbb{E}\big[\sum_{i=0}^{k}{i}^2\epsilon_{2_\Omega}^{i}\big] + \frac{\gamma}{2}\sqrt{\sum_{i=1}^{k}i^{4}(\epsilon_{2_\Omega}^i)^2},\\
    S_{\epsilon_{1_\Omega}}&=\gamma|\delta|M_{\nabla g}\sqrt{n\sum_{i=1}^{k} i^2\norm{u_{_\Omega}^i}_{2}^2},\\
    S_{r_{_\Omega}}&=\gamma\sqrt{\frac{2}{s}\sum_{i=1}^{k} i^2\norm{u_{_\Omega}^i}_{2}^2\epsilon_2^i}.
\end{align}
with probability at least $1-6\exp(-\gamma^2/2)$, where $u_{_\Omega}^{i} = x^\star- x_{}^i + (1-\alpha_{i-1})(x_{}^{i} -x^{i-1})$, $x^\star$ is any solution of~\eqref{Eq:ProblemRepeated},  $M_{\nabla g} = 1$, $s \leq 1/L$ for absolute error Model \ref{ErrModel1}, and $M_{\nabla g} = \underset{i \in \mathbb{N}_{+}} \sup\bigg\{\norm{\nabla g(x^i)}_{\infty}\bigg\}$, $s_k := s \leq 1/(L+\delta)$ for relative error Model \ref{ErrModel2}. $\mathbb{E}[.]$ stands for the expectation operator.
\end{Theorem}
\begin{proof}See Section \ref{Subsec:Proofprop5}\end{proof}
\begin{Corollary} [\textbf{Accelerated with random errors}]\label{Corol1OfTheorem2}
Consider problem~\eqref{Eq:ProblemRepeated} and let the assumptions of Theorem \ref{Theorem2} hold. Let the sequence $\{\norm{u_{_\Omega}^{i}}_2\}$ be upper bounded by $M_u\norm{x^0-x^\star}_2$ with a real finite constant $M_u < \infty$ and let $\varepsilon_0$ be an upper bound on the proximal error, i.e., $\epsilon_{2_\Omega}^k \leq \varepsilon_0$ for all $k$. Then we have, for all $k$,  
\begin{equation}
   f(x_{_\Omega}^{k+1})-f(x^\star) \leq \frac{1}{\alpha_k^2} \bigg[S_{\epsilon_{2_\Omega}}+S_{r_{_\Omega}}+S_{\epsilon_{1_\Omega}} +\frac{1}{2s}\norm{x^\star-x^0}_{2}^2\bigg]
   \label{Eq:Corollary5.1},
\end{equation}
where 
\begin{align}
    S_{\epsilon_{2_\Omega}}&=\varepsilon_0\frac{k(k + 1)(2k + 1)}{6} + \frac{\gamma}{2}\varepsilon_0\sqrt{\frac{k(k+1)(2k+1)(3k^2+3k-1)}{30}},\\
    S_{\epsilon_{1_\Omega}}&=\gamma|\delta|M_uM_{\nabla g}\norm{x^0-x^\star}_2\sqrt{\frac{nk(k + 1)(2k + 1)}{6}},\\
    S_{r_{_\Omega}}&=\gamma M_u\norm{x^0-x^\star}_2\sqrt{\frac{2s \varepsilon_0 k(k + 1)(2k + 1)}{6}}.
\end{align}
with probability at least $1-6\exp(-\gamma^2/2)$, where $u_{_\Omega}^{i} = x^\star- x_{}^i + (1-\alpha_{i-1})(x_{}^{i} -x^{i-1})$, $x^\star$ is any solution of~\eqref{Eq:ProblemRepeated},  $M_{\nabla g} = 1$, $s \leq 1/L$ for absolute error Model \ref{ErrModel1}, and $M_{\nabla g} = \underset{i \in \mathbb{N}_{+}} \sup\bigg\{\norm{\nabla g(x^i)}_{\infty}\bigg\}$, $s_k := s \leq 1/(L+\delta)$ for relative error Model \ref{ErrModel2}. 
\end{Corollary}
\begin{proof}
Substituting for $\sum_{i=1}^{k} i^{2}$ by
\begin{equation}
\sum_{i=1}^{k} i^{2} = \frac{k(k + 1)(2k + 1)}{6},
\end{equation}
and substituting for $\sum_{i=1}^{k} i^{4}$ by,
\begin{equation}
\sum_{i=1}^{k} i^{4} = \frac{k(k+1)(2k+1)(3k^2+3k-1)}{30},
\end{equation}
and using $\norm{u_{_\Omega}^{i}}_2 \leq M_u\norm{x^0-x^\star}_2$, $\norm{\epsilon_{1_\Omega}^{i}}_2 \leq |\delta|M_{\nabla g}\sqrt{n}$ in Theorem~\ref{Theorem2} completes the proof.
\end{proof}
It is not surprising that in the absence of any perturbations, both probabilistic and deterministic analyses lead to the optimal convergence rate of $O({1}/{k^2})$ for the accelerated scheme \eqref{Eq:Theorem4}-\eqref{Eq:Theorem5}. However, as stated previously in Corollary \ref{corol:2}, under the influence of computational inexactness and due to error amplification, acceleration has a counter-effect in the Nesterov's sense \cite{nesterov1983method} and the method becomes more sensitive to noise whenever we want to speed up the PG algorithm.

Although computational errors are deterministic in nature \cite{higham2002accuracy}, probabilistic results such as \eqref{Eq:Theorem5} give us practical convergence bounds when errors cannot be measured or are undetectable but with known upper bounds. If the ensemble mean $\mathbb{E}[\epsilon_{2_\Omega}^{i}]$ is constant for all $i \geq 1$ in \eqref{Eq:Theorem5}, i.e, the error sequence $\{\epsilon_{2_\Omega}^k\}$ is stationary, then \eqref{Eq:Theorem5} becomes totally independent from the instantaneous running errors $\epsilon_{1_\Omega}^{i}$, $\epsilon_{2_\Omega}^{i}$ as well as from the running iterates $x_{_\Omega}^{i}$ and would be only determined by the machine precision $\delta$, the tolerance $\mathbb{E}[\epsilon_{2_\Omega}^{}]$ and the given probability parameter $\gamma$. The factor $\alpha_k$ is designed to be proportional to the iteration counter $o(k)$.

Note that for this design, i.e., $\alpha_k \propto o(k)$, the gradient error term \\$\frac{\gamma|\delta|}{\alpha_k^2}\norm{x^\star-x^0}_{2}\sqrt{\frac{nk(k + 1)(2k + 1)}{6}}$ in \eqref{Eq:Theorem5} continues to decrease as $O({1}/{\sqrt{k}})$ (since ${1}/{\alpha_k^2} \propto O(1/k^2)$) as in the basic case without summability assumption on $\{\epsilon_1^k\}$ but we still cannot guarantee convergence unless $\{\alpha_{k}^2E\big(\epsilon_{2_\Omega}^{k}\big)\}$ is summable or equivalently $E\big(\epsilon_{2_\Omega}^{k}\big)$ decreases as $O({1}/{k^{2+\varsigma}})$.
We also note that the absolute convergence of the algorithm rests upon the summability of $\{\alpha_{k}^2E\big(\epsilon_{2_\Omega}^{k}\big)\}$  and errors boundedness without any additional (stronger) requirements. Furthermore, in order to achieve an optimal convergence rate of $O({1}/{k^2})$ at a fixed dimensionality of the problem $n$, we require $\delta$ to decrease (i.e., increase machine precision) as $O({1}/{k^{1.5+\varsigma}})$ instead of $O({1}/{k^{2+\varsigma}})$ \cite{schmidt2011convergence}.  Consequently, if we let $\norm{\epsilon_{1_\Omega}^{k}}_2$ (or equivalently $\delta$) and $\sqrt{\epsilon_{2_\Omega}^{k}}$ to decrease as $O({1}/{k^2})$ we recover the optimal rate $O({1}/{k^2})$ instead of $O({\log^2 k}/{k^2})$ \cite{schmidt2011convergence}. To the lower limit, if $\norm{\epsilon_{1_\Omega}^{k}}_2$ (or equivalently $\delta$) and $\sqrt{\epsilon_{2_\Omega}^{k}}$ decrease as $O({1}/{k^{1.5}})$ we achieve an asymptotic convergence rate of $O({\log k}/{k^2})$ which is still less conservative than the latter even for worse computational errors. The algorithm fails to converge with non-summable proximal error. 

In summary, although boundedness of the gradient error is sufficient for the gradient error term $S_{\epsilon_{1_\Omega}}$ to asymptotically vanish, the algorithm fails to converge without the summability of the proximal error term $\{\alpha_{k}^2E\big(\epsilon_{2_\Omega}^{k}\big)\}$. 
\section{Proofs}
\label{Sec:Proofs}
\label{Sec:Proofs}
\subsection{Proof of Lemma~\ref{Lem:LipsChitz+Errors}}
\label{Subsec:ProofLipsChitz+Errors}
Let us show that \textit{Lipschitz} continuity of the gradient still holds with absolute errors. We have for absolute gradient error $\epsilon_1^{k}$
\begin{align}
    \norm{\nabla^{\epsilon_1^{k}} g(y^{k})-\nabla^{\epsilon_1^{k}} g(z^{k})}_2 &=     \norm{\nabla g(y^{k})+\epsilon_1^{k}-\nabla g(z^{k})-\epsilon_1^{k}}_2
    \\
    &= \norm{\nabla g(y^{k})-\nabla g(z^{k}) }_2\leq L  \norm{y^{k}-z^{k}}_2\notag.
\end{align}
The latter cancellation occurs because at instant $k$ the absolute error is the same no matter what the argument of $g$ is. However, for relative errors, i.e., when $\epsilon_1^{k} = \epsilon_1^{k}(y^k)= \kappa \odot \nabla g(y^{k})$ where $|\kappa| \leq \delta$, where $\odot$ stands for the vector element-wise Hadamard product, we have
\begin{align}
    \norm{\nabla^{\epsilon_1^{k}} g(y^{k})-\nabla^{\epsilon_1^{k}} g(z^{k})}_2 &=     \norm{\nabla g(y^{k})+\epsilon_1^{k}(y^{k})-\nabla g(z^{k})-\epsilon_1^{k}(z^{k})}_2
    \\
    &=\norm{\nabla g(y^{k})(1+\kappa)-\nabla g(z^{k})(1+\kappa)}_2\notag
    \\
    &= |1+\kappa|\norm{\nabla g(y^{k})-\nabla g(z^{k}) }_2\leq (1+\delta)L\norm{y^{k}-z^{k}}_2.\notag 
\end{align}
\subsection{Proof of Theorem~\ref{prop:1}}
\label{Subsec:Proofprop1}

Recall the definition of $\epsilon$-suboptimal proximal operator
in~\eqref{Eq:ProxApproximate}:
\begin{equation}
  \label{Eq:ProxApproximate-Repetition}    
    \text{prox}_{u}^{\epsilon}(y) 
    := 
    \Big\{
      x \in \mathbb{R}^n\,:\,
      u(x) + \frac{1}{2}\|x -y\|_2^2 \leq \epsilon + \underset{z}{\inf}\,\,
      u(z) + \frac{1}{2}\|z -
      y\|_2^2 
    \Big\}\,.
\end{equation}
Because this is a set, the point $x^{k+1}$ in approximate
proximal gradient~\eqref{Eq:PGDApproximateRepeated} is not defined uniquely.  To bound the
effect of the error $\epsilon_2^{k}$, we will therefore compute its difference
with respect to the case where $\epsilon_2^{k} = 0$, as measured by a function that
we will define shortly.  Recall that $\overline{x}^{k+1}$ is the noiseless
computation of the proximal operator in~\eqref{Eq:PGDApproximateRepeated} at
$x^k$ with constant stepsize $s$:
\begin{align}
  \overline{x}^{k+1} 
  :&=
  \text{prox}_{s h}
  \Big[x^k - s\big(\nabla g(x^k) + \epsilon_1^{k}\big)\Big]
  \label{Eq:ProofThm1-Init1},
  \\
  &=
  \text{prox}_{s h}
  \Big[x^k - s\nabla^{\epsilon_1^{k}} g(x^k)\Big]
  \label{Eq:ProofThm1-Init2}
  \\
  &=
  \underset{x}{\arg\min} \,\,\,  
  g(x^{k}) + \nabla^{\epsilon_1^{k}} g(x^k)^\top (x-x^k)
  +
  \frac{1}{2s}\big\|x-x^k\big\|_{2}^2+h(x) 
  \label{Eq:ProofThm1-Init3}
  \\
  &=:
  \underset{x}{\arg\min} \,\,\,  
  G\big(x,\, x^k\big)\,.
  \label{Eq:ProofThm1-Init4}
\end{align}
From~\eqref{Eq:ProofThm1-Init1} to~\eqref{Eq:ProofThm1-Init2}, we used
$\nabla^{\epsilon_1^{k}}g(x^k) := \nabla g(x^k) + \epsilon_1^{k}$ as the inexact
gradient of $g$ at $x^k$. From~\eqref{Eq:ProofThm1-Init2}
to~\eqref{Eq:ProofThm1-Init3}, we developed the squared $\ell_2$-norm term in
the definition of the proximal operator [cf.\ \eqref{Eq:Prox}] and added
$g(x^k)$ to the objective function. Finally,
from~\eqref{Eq:ProofThm1-Init3} to~\eqref{Eq:ProofThm1-Init4}, we defined
\begin{equation}
  \label{G0}
  G\big(x,\,x^k\big) 
  := 
  g(x^{k}) + \nabla^{\epsilon_1^{k}}
  g(x^k)^\top (x-x^k)+\frac{1}{2s}\big\|x-x^k\big\|_{2}^2
  +
  h(x)\,.
\end{equation}

As $h$ is convex [cf.\ Assumption~\ref{Ass:OptimizationProblem}], the quadratic
term in~\eqref{G0} makes the function $G(\cdot,\, x^k)$ strongly convex
with parameter $1/s$~\cite{beck2017first}.

Recall that $\overline{x}^{k+1}$ is the optimal solution of~\eqref{Eq:ProofThm1-Init4} and that
$x^{k+1}$ is the actual, noisy iterate in~\eqref{Eq:PGDApproximateRepeated}.
Therefore, according to~\eqref{Eq:PGDApproximateRepeated} and to the definition of
the $\epsilon$-suboptimal proximal operator
in~\eqref{Eq:ProxApproximate-Repetition}, 
\begin{align}
  &
  h\big(x^{k+1}\big)
  +
  \frac{1}{2s}\Big\|x^{k+1} - x^k + s \nabla^{\epsilon_1^{k}} g(x^k)\Big\|_2^2
  \leq
  \epsilon_2^{k}
  +
  h\big(\overline{x}^{k+1}\big)
  \label{Eq:ProofTheorem1-RelBetweenNoiselessNoisy1}\\
  \notag&+
  \frac{1}{2s}\Big\|\overline{x}^{k+1} - x^k + s \nabla^{\epsilon_1^{k}}
  g(x^k)\Big\|_2^2
  \\
  \label{Eq:ProofTheorem1-RelBetweenNoiselessNoisy2}  
  \Longleftrightarrow\qquad
  &
  h\big(x^{k+1}\big)
  +
  \frac{1}{2s}\big\|x^{k+1} - x^k\big\|_2^2
  +
  \nabla^{\epsilon_1^{k}} g(x^k)^\top \big(x^{k+1} - x^k\big)
  \leq
  \\
  &\epsilon_2^{k}
  +
  h\big(\overline{x}^{k+1}\big)
  +
  \frac{1}{2s}\big\|\overline{x}^{k+1} - x^k\big\|_2^2
  +\nabla^{\epsilon_1^{k}} g(x^k)^\top \big(\overline{x}^{k+1} - x^k\big)
  \notag
  \\
  \Longleftrightarrow\qquad
  &G\big(x^{k+1},\,x^k\big)
  -
  G\big(\overline{x}^{k+1},\,x^k\big) 
  \leq \epsilon_2^{k}\,.
  \label{suboptimal}
\end{align}
From~\eqref{Eq:ProofTheorem1-RelBetweenNoiselessNoisy1}
to~\eqref{Eq:ProofTheorem1-RelBetweenNoiselessNoisy2}, we developed the
squared-norm terms and cancelled the common term.
From~\eqref{Eq:ProofTheorem1-RelBetweenNoiselessNoisy2} to~\eqref{suboptimal},
we added the constant $g(x^k) - \frac{s}{2}\big\|\nabla g(x^k)\big\|_2^2$ to both sides and used the definition~\eqref{G0}. Notice
that~\eqref{suboptimal} bounds the distance between $x^{k+1}$ and
$\overline{x}^{k+1}$ as measured by $G(\cdot,\, x^k)$.

Because $G(\cdot,\,x^k)$ is strongly convex, Theorem~\ref{thm:2} in the
Appendix establishes that 
\begin{equation}
  \label{Eq:ProofTheorem1-ApplicationOfStrongConv}
  G\big(x,\, x^k\big) - G\big(\overline{x}^{k+1},\, x^k\big) 
  \geq 
  \frac{1}{2s} \big\|x - \overline{x}^{k+1}\big\|_2^2\,,
\end{equation}
for any $x \in \mathbb{R}^n$. In particular, it holds for any optimal solution
$x^\star$ of~\eqref{Eq:ProblemRepeated}.

Thus, subtracting~\eqref{Eq:ProofTheorem1-ApplicationOfStrongConv} with $x =
x^\star$ from~\eqref{suboptimal} yields
\begin{align}
  &
  G\big(x^{k+1},\,x^k\big)
  -
  G\big(x^\star,\,x^k\big) 
  \leq 
  \epsilon_2^{k}
  -
  \frac{1}{2s}\big\|x^\star-\overline{x}^{k+1}\big\|_2^2
  \label{27}
  \\
  \Longleftrightarrow\qquad
  &
  g(x^k) 
  + 
  \nabla^{\epsilon_1^{k}} g(x^k)^\top \big(x^{k+1}-x^k\big)
  +
  \frac{1}{2s}\big\|x^{k+1}-x^k\big\|_{2}^2 + h\big(x^{k+1}\big)
  \label{Eq:ProofTheorem1-ApplyingDefinitionG}
  \\
  \notag
  &
  - 
  G\big(x^\star,\,x^k\big)
  \leq 
  \epsilon_2^{k}
  -
  \frac{1}{2s}\big\|x^\star-\overline{x}^{k+1}\big\|_2^2
  \\
  \Longleftrightarrow\qquad
    &g(x^k) 
  + 
  \nabla g(x^k)^\top \big(x^{k+1}-x^k\big) 
  + {\epsilon_1^{k}}^\top \big(x^{k+1}-x^k\big)
  \label{Eq:ProofTheorem1-BeforeConvex}
  \\
  \notag
  &
  +
  \frac{1}{2s}\big\|x^{k+1}-x^k\big\|_{2}^2 + h\big(x^{k+1}\big)
  - 
  G\big(x^\star,\,x^k\big)
  \leq 
  \epsilon_2^{k} 
  -
  \frac{1}{2s}\big\|x^\star-\overline{x}^{k+1}\big\|_2^2.
\end{align}
From~\eqref{27} to~\eqref{Eq:ProofTheorem1-ApplyingDefinitionG}, we simply used
the definition of $G(x,\, x^k)$ in~\eqref{G0} with $x = x^{k+1}$ and we also used
$\nabla^{\epsilon_1^{k}}g(x^k) := \nabla g(x^k) + \epsilon_1^{k}$ in \eqref{Eq:ProofTheorem1-BeforeConvex}. 
From~\eqref{Eq:ProofTheorem1-ApplyingDefinitionG}
to~\eqref{Eq:ProofTheorem1-BeforeConvex}, we
used~\eqref{Eq:LipschitzContinuityAlt}, which follows from the fact that $g$
has a \textit{Lipschitz} gradient with constant $L$ (and $s \leq 1/L$), with
$x=x^k$ and $y = x^{k+1}$.

Applying \eqref{Eq:LipschitzContinuityAlt} to \eqref{Eq:ProofTheorem1-BeforeConvex} (with $s \leq 1/L$) and using $f:=g+h$, we obtain
\begin{align}
  &
  g\big(x^{k+1}\big) + h\big(x^{k+1}\big)
  - 
  G\big(x^\star,\,x^{k}\big) 
  \leq
  \epsilon_2^{k}
  -
  \frac{1}{2s}\big\|x^\star-\overline{x}^{k+1}\big\|_2^2
  \label{Eq:ProofTheorem1-StrongConvg}
  \\
  \notag
  &
  + 
  {\epsilon_1^{k}}^\top \big(x^k - x^{k+1}\big),
  \\
  \Longleftrightarrow\qquad
  &
  f\big(x^{k+1}\big)
  - 
  G\big(x^\star,\,x^{k}\big) 
  \leq
  \epsilon_2^{k}
  -
  \frac{1}{2s}\big\|x^\star-\overline{x}^{k+1}\big\|_2^2
  + 
  {\epsilon_1^{k}}^\top \big(x^k - x^{k+1}\big)\,.
  \label{G-G(34)}
\end{align}
We now expand $G(x^\star,\,x^k)$ in~\eqref{G-G(34)} as follows 
\begin{equation}
\begin{split}
  f(x^{k+1})
  -
  g(x^k) - \nabla^{\epsilon_1^{k}} g(x^k)^\top (x^\star-x^k)
  -
  \frac{1}{2s}\big\|x^\star-x^k\big\|_{2}^2
  -
  h(x^\star)
  \\
  \leq
  \epsilon_2^{k}
  -
  \frac{1}{2s}\big\|x^\star-\overline{x}^{k+1}\big\|_2^2
  + 
  {\epsilon_1^{k}}^\top \big(x^k - x^{k+1}\big)\,.
\end{split}
\end{equation}
Rearranging and subtracting $g(x^\star)$ from both sides yields
\begin{equation}
\begin{split}
  \label{Eq:ProofTheorem1-Aggregating1}
  f(x^{k+1}) 
  -
  h(x^\star)
  -
  g(x^\star) 
  &\leq
  -g(x^\star) 
  + 
  \epsilon_2^{k}
  -
  \frac{1}{2s}\big\|x^\star-\overline{x}^{k+1}\big\|_2^2 
  + 
  g(x^k) 
  \\
  &\quad+ 
  \nabla^{\epsilon_1^{k}} g(x^k)^\top (x^\star-x^k)
  +
  \frac{1}{2s}\big\|x^\star-x^k\big\|_2^2
  \\
  &\quad+{\epsilon_1^{k}}^\top \big(x^k - x^{k+1}\big)\,.
\end{split}
\end{equation}
Using the definitions $f:= g + h$ and $\nabla^{\epsilon_1^{k}} g(x^k) = \nabla
g(x^k) +  \epsilon_1^{k}$ in~\eqref{Eq:ProofTheorem1-Aggregating1}, we obtain 
\begin{equation}
\begin{split}
  f(x^{k+1}) - f(x^\star) 
  &\leq 
  \epsilon_2^{k}
  - g(x^\star) 
  + g(x^k) 
  + \nabla g(x^k)^\top \big(x^\star-x^k\big) 
  \\
  &\quad-\frac{1}{2s}\big\|x^\star-\overline{x}^{k+1}\big\|_2^2
  +\frac{1}{2s}\big\|x^\star-x^k\big\|_2^2 
  + {\epsilon_1^{k}}^\top\big(x^\star-x^{k}\big)
  \\
  &\quad+{\epsilon_1^{k}}^\top \big(x^k - x^{k+1}\big)\,
  \\
  &\leq
  \epsilon_2^{k}
  -
  \frac{1}{2s}\big\|x^\star-\overline{x}^{k+1}\big\|_2^2
  +
  \frac{1}{2s}\big\|x^\star-x^k\big\|_2^2
  +
  {\epsilon_1^{k}}^\top\big(x^\star-x^{k+1}\big)\,,
  \label{Eq:ProofTheorem1-LastStepBeforeSumming}
\end{split}
\end{equation}
where in the second inequality we used the fact that $g$ is convex, i.e., $g(x^\star)
\geq g(x^k) + \nabla g(x^k)^\top  (x^\star - x^k)$.
Summing both sides of~\eqref{Eq:ProofTheorem1-LastStepBeforeSumming} from $0$
to $k$, 
\begin{equation}
	\begin{split}
	\sum_{i=0}^{k}\big[f(x^{i+1}) - f(x^\star)\big] 
		&\leq 
		\sum_{i=0}^{k}\epsilon_2^{i}
		+ 
		\sum_{i=0}^{k}{\epsilon_1^{i}}^\top\big(x^\star-x^{i+1}\big)
    	\\
    	&\quad
	    +	\frac{1}{2s}\sum_{i=0}^{k}
		\Big[
			\big\|x^\star-x^i\big\|_{2}^2
			-			\big\|x^\star-\overline{x}^{i+1}\big\|_2^2
		\Big],
		\\
		&=  
		\sum_{i=0}^{k}\epsilon_2^{i}
		+ 
		\sum_{i=0}^{k}{\epsilon_1^{i}}^\top\big(x^\star-x^{i+1}\big)
		+\frac{1}{2s}\sum_{i=0}^{k}
		\Big[
			\big\|x^\star-x^i\big\|_{2}^2
			\\
		    &\quad 
			-\big(\big\|x^\star-x^{i+1}\big\|_2^2+
			\big\|x^{i+1}-\overline{x}^{i+1}\big\|_2^2
			\\
			&\quad+2
			(x^{i+1}-\overline{x}^{i+1})^\top(x^\star-x^{i+1})
		\big)\Big],
		\\
		&=  
		\sum_{i=0}^{k}\epsilon_2^{i}
		+ 
		\sum_{i=0}^{k}{\epsilon_1^{i}}^\top\big(x^\star-x^{i+1}\big)
	    +\frac{1}{2s}\sum_{i=0}^{k}
		\Big[
			\big\|x^\star-x^i\big\|_{2}^2
			\\
		    &\quad
			-\big(\big\|x^\star-x^{i+1}\big\|_2^2+
			\big\|r^{i+1}\big\|_2^2+2
			(r^{i+1})^\top(x^\star-x^{i+1})
		\big)\Big],
		\\
		&=
		\sum_{i=0}^{k}\epsilon_2^{i}
		+ 
		\sum_{i=0}^{k}(\epsilon_1^{i}-\frac{1}{s}r^{i+1})^\top\big(x^\star-x^{i+1}\big)
		+
		\frac{1}{2s}
		\Big[
			\big\|x^\star-x^0\big\|_{2}^2
			\\
		    &\quad
			-
			\big\|x^\star-x^{k+1}\big\|_2^2
		\Big]
		 - \frac{1}{2s}\sum_{i=0}^{k}\big\|r^{i+1}\big\|_2^2\,,
		\label{Eq:ProofTheorem1-StepBeforeJensen}
	\end{split}
\end{equation}
where in the second-to-last inequality we used the definition of $r^i$ in \eqref{Eq:ResidualVec}, and in the last equality we noticed that the quadratic terms involving $x^\star$ formed a telescopic sequence. Rearranging and moving negative terms to the left hand side results in 
\begin{equation}
	\label{Eq:ProofTheorem1-StepBeforeJensen}
	\begin{split}
		\sum_{i=0}^{k}\big[f(x^{i+1}) - f(x^\star)\big] 
		+\frac{1}{2s}\sum_{i=0}^{k}\big\|r^{i+1}\big\|_2^2 +\frac{1}{2s}\big\|x^\star-x^{k+1}\big\|_2^2
		&\leq 
		\sum_{i=0}^{k}\epsilon_2^{i}
		\\
		+\sum_{i=0}^{k}(\epsilon_1^{i}-\frac{1}{s}r^{i+1})^\top\big(x^\star-x^{i+1}\big)
		+\frac{1}{2s}\big\|x^\star-x^0\big\|_{2}^2.
	\end{split}
\end{equation}
Since $f$ is a convex function, Jensen's inequality implies
\begin{equation*}
    f\bigg(\frac{1}{k+1}\sum_{i=0}^{k}x^{i+1}\bigg)  -f(x^\star) \leq
    \frac{1}{k+1} \sum_{i=0}^{k}\big[f(x^{i+1})  -f(x^\star)\big]\,,
\end{equation*}
which, applied to~\eqref{Eq:ProofTheorem1-StepBeforeJensen} and together with the fact that the last two terms of the left-hand side of \eqref{Eq:ProofTheorem1-StepBeforeJensen} are nonnegative, yields the
statement of the theorem:
\begin{align}
    f\bigg(\frac{1}{k+1}\sum_{i=1}^{k}x^{i+1}\bigg)  -f(x^\star)    +\frac{1}{2(k+1)s}\sum_{i=0}^{k}\big\|r^{i+1}\big\|_2^2 +\frac{1}{2(k+1)s}\big\|x^\star-x^{k+1}\big\|_2^2
    \leq \notag\\ 
    \frac{1}{k+1}
    \Big[
      \sum_{i=0}^{k}\epsilon_2^{i} 
      + 
      \sum_{i=0}^{k}(\epsilon_1^{i}-\frac{1}{s}r^{i+1})^\top(x^\star-x^{i+1})
    +
    \frac{1}{2s}\norm{x^\star-x^0}_{2}^2\Big]\,. \label{Eq:MainDeterministic}
\end{align}
\hfill\qed

\subsection{Proof of Corollary~\ref{corol:1}}
\label{Subsec:Proofcorol1}
We will use Lemma~\ref{Lem:BackwardProxError} to bound the norm of the residual error $r^{k} = x^{k}-\overline{x}^{k}$ resulting from the proximal error $\epsilon_2^{k}$.
Using Cauchy-Shwartz inequality and the bound from Lemma \ref{Lem:BackwardProxError}, for all $i$, we obtain 
\begin{equation}
    \label{Ineq:residualErrorBoundCS}
    \begin{split}
        (\epsilon_1^{i}-\frac{1}{s}r^{i+1})^\top(x^\star-x^i) &\leq \Big(\norm{\epsilon_1^{i}}_{2}+\frac{1}{s}\norm{r^{i+1}}_{2}\Big)\norm{x^\star-x^i}_{2}
        \\
        &\leq \Big(\norm{\epsilon_1^{i}}_{2}+\sqrt{\frac{2\epsilon_2^{i}}{ s}}\Big)\norm{x^\star-x^i}_{2}.
    \end{split}
\end{equation}

Using  \eqref{Ineq:residualErrorBoundCS} in \eqref{Eq:Theorem1} yields
\begin{equation}
    \begin{split}
        f\bigg(\frac{1}{k+1}\sum_{i=1}^{k}x^{i+1}\bigg)  -f(x^\star) & \leq \frac{1}{k+1} \sum_{i=1}^{k}\epsilon_2^{i}
        \\
        &\quad +\frac{1}{k+1}\sum_{i=1}^{k}\Big(\norm{\epsilon_1^{i}}_{2}+\sqrt{\frac{2\epsilon_2^{i}}{s}}\Big)\norm{x^\star-x^i}_{2}
        \\
        &\quad+\frac{1}{2s(k+1)}\norm{x^\star-x^0}_{2}^2
	\end{split}
\end{equation}
Applying Quasi-F\'ejer (Theorem \ref{Thm:QuasiFejerPG}) recursively gives
\begin{equation}
    \begin{split}
        f\bigg(\frac{1}{k+1}\sum_{i=1}^{k}x^{i+1}\bigg)  -f(x^\star) &\leq \frac{1}{k+1} \sum_{i=1}^{k}\epsilon_2^{i}+\frac{1}{2s(k+1)}\norm{x^\star-x^0}_{2}^2\\
        &\quad+\frac{1}{k+1}\sum_{i=1}^{k}\Big(\norm{\epsilon_1^{i}}_{2}
        +\sqrt{\frac{2\epsilon_2^{i}}{s}}\Big)\norm{x^\star - x^{i-1}}_2\\
        &\quad+\frac{1}{k+1}\sum_{i=1}^{k}\Big(\norm{\epsilon_1^{i}}_{2}
        +\sqrt{\frac{2\epsilon_2^{i}}{s}}\Big)(E^{i} + C_\rho)\\
        &\leq \frac{1}{k+1} \sum_{i=1}^{k}\epsilon_2^{i} +\frac{1}{2s(k+1)}\norm{x^\star-x^0}_{2}^2\\
        &\quad+\frac{1}{k+1}\sum_{i=1}^{k}\Big(\norm{\epsilon_1^{i}}_{2}+\sqrt{\frac{2\epsilon_2^{i}}{s}}\Big)\norm{x^\star - x^{0}}_2\\ &\quad+\frac{1}{k+1}\sum_{i=1}^{k}\Big(\norm{\epsilon_1^{i}}_{2}+\sqrt{\frac{2\epsilon_2^{i}}{s}}\Big)(\sum_{j=1}^i E^{j}+ i C_\rho),
    \end{split}
\end{equation}
where $E^j = \norm{r^{j}}_2+s_{j-1} \norm{\epsilon_1^{j-1}}_2$ and $C_\rho = 0$ if the optimum $x^\star$ is reached. This completes the proof of Corollary \ref{corol:1}. 

\subsection{Proof of Theorem~\ref{prop:2}}
\label{Subsec:Proofprop2}
This result is about the basic version of approximate PGD, but with random proximal computation error $\epsilon_{2_\Omega}$, component-wise bounded gradient error $\epsilon_{1_\Omega}$ and bounded residuals $\norm{x_{_\Omega}^k-x^\star}_2$. As the algorithm generates a sequence of random vectors $\{x_{_\Omega}^k\}$, the residual vector sequence $\{r_{_\Omega}^k\}$ will also be a random.

Let $T_k$ denote the second error term in the bound of \eqref{Eq:Theorem1} [Theorem \ref{prop:1}], i.e.,
\begin{equation}
\label{martingale}
T_k =
\left\{
\begin{array}{ll}
0 &,\,\, k=0
\\
\sum_{i=0}^k (\epsilon_{1_\Omega}^{i}-\frac{1}{s}r_{_\Omega}^{i+1})^\top (x^\star - x_{_\Omega}^{i+1})&,\,\,k = 1, 2, \ldots\,,
\end{array}
\right.
\end{equation}
The first step is to show that $\{T_k\}$ is a martingale. Recall that  a sequence of random variables $T_0,T_1,\dots$ is a martingale with respect to the sequence $X_0,X_1,\dots$ if, for all $k\geq0$, the following conditions hold:
\begin{itemize}
    \item $T_k$ is a function of $X_0,X_1,\dots,X_k$;
    \item $E(|T_k|) < \infty$;
    \item $E(T_{k+1}|X_0,X_1,\dots,X_k) = T_k$.
\end{itemize}
A sequence of random variables $T_0,T_1,\dots$ is called a martingale when it is a martingale with respect to itself. That is, $E(|T_k|) < \infty$, and $E(T_{k+1}|T_0,T_1,\dots,T_k) = T_k$.

Let $\nu_{_\Omega}^k = \epsilon_{1_\Omega}^{k-1}-\frac{1}{s}r_{_\Omega}^{k}$ and recall the definition of $r_{_\Omega}^{k}$ in \eqref{Eq:ResidualVec}:
\begin{equation}
    r^k = x^k - \overline{x}^{k}.   
\end{equation}
Rewriting \eqref{martingale} in terms of $\nu_{_\Omega}^k$ yields
\begin{equation}
        \label{Eq:MartingaleNu}
        T_k = T_{k-1} + {\nu_{_\Omega}^k}^\top (x^\star-x_{_\Omega}^k).
\end{equation}
We now show that Assumptions \ref{Assum:RandomVectors} and \ref{Assum:RandomVectors2} imply that $\{T_k\}_{k\geq 0}$ is a martingale. Specifically, \eqref{Eq:AssumptionRandomZeroMean} and \eqref{Eq:ResidAssumptionRandomZeroMean}, we have
\begin{equation}
\mathbb{E}\big[\nu_{_\Omega}^{k}\big\vert \nu_{_\Omega}^{1}\dots \nu_{_\Omega}^{k-1}\big] =  \mathbb{E}[\nu_{_\Omega}^{k}] = 0.\notag 
\end{equation}
And from \eqref{Eq:AssumptionRandomInnerProduct} and \eqref{Eq:ResidAssumptionRandomInnerProduct}, we have
\begin{equation}
\mathbb{E}\big[{\nu_{_\Omega}^{k}}^\top x_{_\Omega}^k\big\vert \nu_{_\Omega}^{1}\dots \nu_{_\Omega}^{k-1},x_{_\Omega}^{1}\dots x_{_\Omega}^{k-1}\big] = \mathbb{E}\big[{\nu_{_\Omega}^{k}}^\top x_{_\Omega}^k\big] = 0.\notag 
\end{equation}
Taking the expected value of both sides of \eqref{Eq:MartingaleNu} conditioned on $\{T_i\}_{i=1}^{k-1}$ gives
\begin{align}
        \mathbb{E}\big[T_k\big\vert T_1\dots T_{k-1}\big] &= \mathbb{E}\big[T_{k-1} + {\nu_{_\Omega}^{k}}^\top (x^\star-x_{_\Omega}^k)\big\vert T_1\dots T_{k-1}\big]\notag
    \\
         &= \mathbb{E}\big[T_{k-1}\big\vert T_1\dots T_{k-1}\big]  + \mathbb{E}\big[{\nu_{_\Omega}^{k}}^\top (x^\star-x_{_\Omega}^k)\big\vert T_1\dots T_{k-1}\big]\notag
    \\
         &= T_{k-1}  + \mathbb{E}\big[{\nu_{_\Omega}^{k}}^\top (x^\star-x_{_\Omega}^k)\big\vert T_1\dots T_{k-1}\big]\notag
    \\
         &= T_{k-1}  + \mathbb{E}\big[{\nu_{_\Omega}^{k}}^\top x^\star\big\vert T_1\dots T_{k-1}\big] -\mathbb{E}\big[{\nu_{_\Omega}^{k}}^\top x_{_\Omega}^k\big\vert T_1\dots T_{k-1}\big]\notag
    \\
         &= T_{k-1}  + \mathbb{E}\big[\nu_{_\Omega}^{k}\big\vert T_1\dots T_{k-1}\big]^\top x^\star -\mathbb{E}\big[{\nu_{_\Omega}^{k}}^\top x_{_\Omega}^k\big\vert T_1\dots T_{k-1}\big]\notag
    \\
         &= T_{k-1}  + \mathbb{E}\big[\nu_{_\Omega}^{k}\big\vert \nu_{_\Omega}^{1}\dots \nu_{_\Omega}^{k-1},x_{_\Omega}^{1}\dots x_{_\Omega}^{k-1}\big]^\top x^\star\label{53}  
    \\     
         &\quad-\mathbb{E}\big[{\nu_{_\Omega}^{k}}^\top x_{_\Omega}^k\big\vert \nu_{_\Omega}^{1}\dots \nu_{_\Omega}^{k-1},x_{_\Omega}^{1}\dots x_{_\Omega}^{k-1}\big]\notag
    \\
         &= T_{k-1}  + \mathbb{E}\bigg[\epsilon_{1_\Omega}^{k-1}-\frac{1}{s}r_{_\Omega}^{k}\bigg]^\top x^\star -\mathbb{E}\bigg[(\epsilon_{1_\Omega}^{k-1}-\frac{1}{s}r_{_\Omega}^{k})^\top x_{_\Omega}^k\bigg] \label{54}
    \\
         &= T_{k-1}  + \mathbb{E}\bigg[\epsilon_{1_\Omega}^{k-1}-\frac{1}{s}r_{_\Omega}^{k}\bigg]^\top x^\star -\mathbb{E}\bigg[\mathbb{E}\bigg[\epsilon_{1_\Omega}^{k-1}-\frac{1}{s}r_{_\Omega}^{k}\big\vert x_{_\Omega}^k\bigg]^\top x_{_\Omega}^k\bigg] \label{55}
    \\
         &= T_{k-1}  - \mathbb{E}\bigg[\epsilon_{1_\Omega}^{k-1}-\frac{1}{s}r_{_\Omega}^{k}\bigg]^\top x_{_\Omega}^k \label{56}
    \\
         &= T_{k-1} \label{57}.
\end{align}
From \eqref{53} to \eqref{54}, we used the error mean independence assumption, i.e,\\ $E\big[\nu_{_\Omega}^{k}\big\vert \nu_{_\Omega}^{1}\dots \nu_{_\Omega}^{k-1}\big] =E\big[\nu_{_\Omega}^{k}\big]$ as well as the data mean independence assumption, i.e, $E\big[{\nu_{_\Omega}^{k}}^\top x_{_\Omega}^k\big\vert \nu_{_\Omega}^{1}\dots \nu_{_\Omega}^{k-1},x_{_\Omega}^{1}\dots x_{_\Omega}^{k-1}\big]$ $=$ $E\big[{\nu^{k}}^\top x_{_\Omega}^k\big] $.  From \eqref{56} to \eqref{57}, we used the zero mean error assumption, i.e, $E\big[\nu_{_\Omega}^{k}\big] = 0$. Therefore, $T_1,T_2,\dots,T_{k}$ is a martingale. 

In what follows, we establish upper bounds on the absolute value of the martingale $\{T_k\}$. To do that, we use the Azuma-Hoeffding inequality in Lemma \ref{lem:3}, noticing that $\big\vert T_k - T_{k-1}\big\vert = \big\vert{\nu_{_\Omega}^k}^\top (x^\star-x_{_\Omega}^k)\big\vert \leq \Big(\sqrt{n}\delta M_{\nabla g}+\sqrt{2\epsilon_2^k/s}\Big)\norm{x_{_\Omega}^\star-x_{_\Omega}^k}_{2}$, where we have used Cauchy-Schwarz, etc. Lemma \ref{lem:3} then yields
\begin{equation}     
    \text{Pr}\bigg(|T_k-T_0|\geq \gamma\sqrt{\sum_{i=1}^k\big(\sqrt{n}M_{\nabla g}|\delta|+\sqrt{\frac{2\epsilon_2^i}{s}}\Big)^2\norm{x_{_\Omega}^\star-x_{_\Omega}^i}_{2}^2}\bigg)\leq 2\exp(-\frac{\gamma^2}{2}).
\end{equation}

Since $\epsilon_2^k \leq \varepsilon_0$, then the following also holds 
\begin{equation}
    \text{Pr}\bigg(|T_k-T_0|\geq \gamma\big(\sqrt{n}M_{\nabla g}|\delta|+\sqrt{\frac{2\varepsilon_0}{s}}\big)\sqrt{\sum_{i=1}^k\norm{x_{_\Omega}^\star-x_{_\Omega}^i}_{2}^2}\bigg)\leq 2\exp(-\frac{\gamma^2}{2}).
\end{equation}
And since $T_0 = 0$ we obtain
\begin{equation}
    \text{Pr}\bigg(|T_k|\geq \gamma\big(\sqrt{n}M_{\nabla g}|\delta|+\sqrt{\frac{2\varepsilon_0}{s}}\big)\sqrt{\sum_{i=1}^k\norm{x_{_\Omega}^\star-x_{_\Omega}^i}_{2}^2}\bigg)\leq 2\exp(-\frac{\gamma^2}{2}).
\end{equation}
Or, equivalently, that
\begin{equation}
    |T_k|\leq
    \gamma\big(\sqrt{n}M_{\nabla g}|\delta|+\sqrt{\frac{2\varepsilon_0}{s}}\big)\sqrt{\sum_{i=1}^k\norm{x_{_\Omega}^\star-x_{_\Omega}^i}_{2}^2}
\end{equation}
holds for all $k \geq 1$ with probability at least $1-2\exp(-\frac{\gamma^2}{2})$. Expanding $T_k$ we obtain
\begin{equation}
    \bigg|\sum_{i=0}^{k}(\epsilon_{1_\Omega}^{i-1}-\frac{1}{s}r_{_\Omega}^{i})^\top(x_{_\Omega}^\star-x_{_\Omega}^i)\bigg|\leq
    \gamma\Bigg(\sqrt{n}M_{\nabla g}|\delta|+\sqrt{\frac{2\varepsilon_0}{s}}\Bigg)\sqrt{\sum_{i=1}^k\norm{x_{_\Omega}^\star-x_{_\Omega}^i}_{2}^2}.
    \label{Ineq:BasicProbUpperBound0}
\end{equation}
By assumption, we have that $\norm{x_{_\Omega}^\star-x_{_\Omega}^i}_{2}^2 \leq \norm{x_{_\Omega}^\star-x_{_\Omega}^0}_{2}^2$ holds with probability $p$, for each $i$. Then,
\begin{equation}     \bigg|\sum_{i=0}^{k}(\epsilon_{1_\Omega}^{i-1}-\frac{1}{s}r_{_\Omega}^{i})^\top(x_{_\Omega}^\star-x_{_\Omega}^i)\bigg|\leq     \gamma\Bigg(M_{\nabla g}\sqrt{nk}|\delta|+\sqrt{\frac{2k\varepsilon_0}{s}}\Bigg)\norm{x_{_\Omega}^\star-x_{_\Omega}^0}_{2}     \label{Ineq:BasicProbUpperBound1} \end{equation}
holds with probablity $p^k\big(1-2\exp(-\frac{\gamma^2}{2})\big)$. 
Substituting \eqref{Ineq:BasicProbUpperBound1} into \eqref{Eq:Theorem1} completes the proof of Theorem \ref{prop:2}.
\subsection{Proof of Theorem~\ref{prop:3}}
\label{Subsec:ProofEq:Theorem3}
Here $\epsilon_{2_\Omega}$ is bounded almost surely and has stationary mean. Specifically, we have $0 \leq \epsilon_{2_\Omega}^k \leq \varepsilon_0$, with probability $1$.  
By Hoeffding’s inequality (Lemma \ref{lem:4}), we can write,
\begin{equation}
    \text{Pr}\bigg(|\sum_{i=1}^{k}\epsilon_{2_\Omega}^i - E(\sum_{i=1}^{k}\epsilon_{2_\Omega}^i)|\geq t\bigg)\leq 2\exp\bigg(\frac{-2t^2}{k\varepsilon_0^2}\bigg),\quad \text{for all}\quad t > 0.
    \label{75}
\end{equation}
Defining the constant mean  $E(\epsilon_{2_\Omega}^k) = E(\epsilon_{2_\Omega})$ and substituting in \eqref{75} yields 
\begin{equation}
    \text{Pr}\bigg(|\sum_{i=1}^{k}\epsilon_{2_\Omega}^i - kE(\epsilon_{2_\Omega})|\geq t\bigg)\leq 2\exp\bigg(\frac{-2t^2}{k\varepsilon_0^2}\bigg),
    \quad
    \text{for all}\quad t > 0.  
\end{equation}
By choosing $t = \frac{\gamma\sqrt k \varepsilon_0}{2}$, for some $\gamma > 0$, we obtain
\begin{equation}
    \text{Pr}\bigg(|\sum_{i=1}^{k}\epsilon_{2_\Omega}^i - kE(\epsilon_{2_\Omega})|\geq \frac{\gamma\sqrt k \varepsilon_0}{2}\bigg)\leq 2\exp\bigg(\frac{-\gamma^2}{2}\bigg)
    \quad
    \text{for all}\quad \gamma > 0.
\end{equation}
Equivalently,
\begin{equation}
    \sum_{i=1}^{k}\epsilon_{2_\Omega}^i \leq kE(\epsilon_{2_\Omega}) + \frac{\gamma\sqrt k \varepsilon_0}{2}
    \label{e2leq}
\end{equation}
holds with probability at least $1 - 2\exp(-\frac{\gamma^2}{2})$. Using the last inequality \eqref{e2leq} in \eqref{Eq:Theorem2} and applying Lemma \ref{lem:sumprob} completes the proof of Theorem \ref{prop:3}. 
\subsection{Proof of Theorem~\ref{Theorem1}}
\label{Subsec:ProofEq:Theorem1}
Define now $G$ as
\begin{equation}
    G(x,y^{k}) =  g(y^{k}) + \nabla^{\epsilon_1^{k}} g(y^{k})^\top (x-y^{k})+\frac{1}{2s}\norm{x-y^{k}}_{2}^2+h(x),
    \label{G_acc}
\end{equation}
where
\begin{equation}
    y_k = (1+\beta_k)x^{k}-\beta_k x^{k-1},
\end{equation}
with momentum $\beta_k \in [0,  1]$. If we evaluate \eqref{G_acc} at iteration $k+1$ we obtain
\begin{equation}
    G(x^{k+1},y^{k}) =  g(y^{k}) + \nabla^{\epsilon_1^{k}} g(y^{k})^\top (x^{k+1}-y^{k})+\frac{1}{2s}\norm{x^{k+1}-y^{k}}_{2}^2+h(x^{k+1}),
\end{equation}
where 
\begin{equation}
    x^{k+1} \in \text{prox}_{s_k h}^{\epsilon_2^{k}}(y^{k} - s\nabla^{\epsilon_1^{k}} g(y^{k}))
\end{equation}
is the perturbed iterate. By completing the square, we can show that the latter is the $\epsilon_2^{k}$-suboptimal solution of the following optimization sub-problem,
\begin{equation}
    \underset{x \in \mathbb{R}^n} {\min} \quad  G(x,y^{k}).
    \label{Eq:MinGxy}
\end{equation}
Namely, expanding $G$ we have the following equivalent problems
\begin{align}
    &\underset{x \in \mathbb{R}^n} {\min} \quad g(y^{k}) + \nabla^{\epsilon_1^{k}} g(y^{k})^\top (x-y^{k})+\frac{1}{2s}\norm{x-y^{k}}_{2}^2+h(x)
    \label{Eq:MinGxyExpand1}
    \\
    &\underset{x \in \mathbb{R}^n} {\min} \quad g(y^{k}) + \nabla^{\epsilon_1^{k}} g(y^{k})^\top (x-y^{k})+\frac{1}{2s}\norm{x-y^{k}}_{2}^2+\frac{s}{2}\norm{\nabla^{\epsilon_1^{k}} g(y^{k})}_{2}^2+h(x)
    \label{Eq:MinGxyExpand2}
    \\
    &\underset{x \in \mathbb{R}^n} {\min} \quad \nabla^{\epsilon_1^{k}} g(y^{k})^\top (x-y^{k})+\frac{1}{2s}\norm{x-y^{k}}_{2}^2+\frac{s}{2}\norm{\nabla^{\epsilon_1^{k}} g(y^{k})}_{2}^2+h(x)
    \label{Eq:MinGxyExpand3}
    \\
    &\underset{x \in \mathbb{R}^n} {\min} \quad \nabla^{\epsilon_1^{k}} \frac{2s}{2s}g(y^{k})^\top (x-y^{k})+\frac{1}{2s}\norm{x-y^{k}}_{2}^2+\frac{s^2}{2s}\norm{\nabla^{\epsilon_1^{k}} g(y^{k})}_{2}^2+h(x)
    \notag
    \\
    &\underset{x \in \mathbb{R}^n} {\min} \quad \frac{\norm{x-y^{k}}_{2}^2+2s\nabla^{\epsilon_1^{k}} g(y^{k})^\top (x-y^{k})+s^2\norm{\nabla^{\epsilon_1^{k}} g(y^{k})}_{2}^2}{2s}+h(x),
    \notag
    \\
    &\underset{x \in \mathbb{R}^n} {\min} \quad \frac{\norm{x-y^{k}-s\nabla^{\epsilon_1^{k}} g(y^{k})}_{2}^2}{2s}+h(x),
    \label{Eq:MinGxyExpand6}
\end{align}
where from \eqref{Eq:MinGxyExpand1} to \eqref{Eq:MinGxyExpand2} we have added ${s}\norm{\nabla^{\epsilon_1^{k}} g(y^{k})}_{2}^2/{2}$ and from \eqref{Eq:MinGxyExpand2} to \eqref{Eq:MinGxyExpand3} we have dropped $g(y^k)$ from the objective function, since both terms do not depend on the variable $x$. The $\epsilon_2$-suboptimal solution of \eqref{Eq:MinGxyExpand6} is, by \eqref{Eq:ProxApproximate}, the approximate proximal operator evaluated at $y^{k} - s\nabla^{\epsilon_1^{k}} g(y^{k})$, i.e., $\text{prox}_{s_k h}^{\epsilon_2^{k}}(y^{k} - s\nabla^{\epsilon_1^{k}} g(y^{k}))$.
Note that $G$ is now centered around $y^{k}$ instead of $x^{k}$ as in the basic case. Similarly to \eqref{27}-\eqref{Eq:ProofTheorem1-BeforeConvex}, we can show that
\begin{equation}
    G(x^{k+1},y^k)-G(x,y^k) \leq \epsilon_2^{k} -\frac{1}{2s}\norm{x-\overline{x}^{k+1}}_{2}^2, \quad \text{for any}\quad x \in \mathbb{R}^n.
\end{equation}
where $\overline{x}^{k+1}$ is the noise-free (in terms of $\epsilon_2$) solution, i.e., solution of~\eqref{Eq:MinGxy}, and
$x^{k+1}$ is the actual, noisy iterate, i.e., the $\epsilon_2^{k}$-suboptimal solution of~\eqref{Eq:MinGxy}, respectively. Expanding $G(x^{k+1},y^k)$ gives
\begin{equation}
    \label{expandG}
    \begin{split}
    g(y^{k}) + \nabla^{\epsilon_1^{k}} g(y^{k})^\top (x^{k+1}-y^{k})+\frac{1}{2s}\norm{x^{k+1}-y^{k}}_{2}^2+h(x^{k+1})  - G(x,y^{k}) \\ \leq \epsilon_2^{k} -\frac{1}{2s}\norm{x-\overline{x}^{k+1}}_{2}^2.
    \end{split}
\end{equation}
By the \textit{Lipschitz} continuity of the gradient of $g$ we have
\begin{equation}
    g(y^{k}) + \nabla g(y^{k})^\top (x^{k+1}-y^{k})+\frac{1}{2s}\norm{x^{k+1}-y^{k}}_{2}^2 \geq g(x^{k+1}).
\end{equation}
for $s \leq \frac{1}{L}$. Adding $h(x^{k+1})$ to both sides of the inequality yields
\begin{equation}
    \begin{split}
        &g(y^{k}) + \nabla g(y^{k})^\top (x^{k+1}-y^{k})+\frac{1}{2s}\norm{x^{k+1}-y^{k}}_{2}^2+h(x^{k+1})\\ &\geq
        g(x^{k+1})+h(x^{k+1}) = f(x^{k+1}).
    \end{split}
\label{G_geq_f}
\end{equation}
Substituting for $\nabla^{\epsilon_1^{k}} g(y^{k})$ in \eqref{expandG} by $\nabla g(y^{k})+\epsilon_1^{k}$, and using~\eqref{G_geq_f} we obtain
\begin{equation}
    f(x^{k+1})-G(x,y^k) \leq \epsilon_2^{k} -\frac{1}{2s}\norm{x-\overline{x}^{k+1}}_{2}^2+{\epsilon_1^{k}}^\top (y^{k}-x^{k+1})
    \label{f-G}.
\end{equation}
Substituting for $G(x,y^k)$ yields
\begin{equation}
\label{f-G expand}
\begin{split}
    f(x^{k+1})-g(y^{k}) - \nabla^{\epsilon_1^{k}} g(y^{k})^\top (x-y^{k})-\frac{1}{2s}\norm{x-y^{k}}_{2}^2-h(x) \leq \epsilon_2^{k} -\frac{1}{2s}\norm{x-\overline{x}^{k+1}}_{2}^2
    \\
    +{\epsilon_1^{k}}^\top (y^{k}-x^{k+1}).
\end{split}
\end{equation}
Subtracting $g(x)$ from both sides of \eqref{f-G expand},
\begin{equation}
    \begin{split}
        f(x^{k+1})-g(y^{k}) - \nabla^{\epsilon_1^{k}} g(y^{k})^\top (x-y^{k})&-\frac{1}{2s}\norm{x-y^{k}}_{2}^2-h(x)-g(x)\leq \epsilon_2^{k}\\
        &\quad-\frac{1}{2s}\norm{x-\overline{x}^{k+1}}_{2}^2-g(x)    +{\epsilon_1^{k}}^\top (y^{k}-x^{k+1})
        \label{f-G-g}.
    \end{split}
\end{equation}
Substituting $h(x) + g(x)$ by $f(x)$ in \eqref{f-G-g} results in 
\begin{equation}
    \begin{split}
        f(x^{k+1})-f(x) &\leq \epsilon_2^{k} + {\epsilon_1^{k}}^\top (x-y^{k}) -\frac{1}{2s}\norm{x-\overline{x}^{k+1}}_{2}^2-g(x) +g (y^{k}) 
        \\
        &\quad+ \nabla g(y^{k})^\top (x-y^{k})+\frac{1}{2s}\norm{x-y^{k}}_{2}^2+{\epsilon_1^{k}}^\top (y^{k}-x^{k+1})
        \label{f-G-g 2}.
    \end{split}
\end{equation}
By the convexity of $g$, \eqref{f-G-g 2}  becomes
\begin{equation}
    f(x^{k+1})-f(x) \leq \epsilon_2^{k} + {\epsilon_1^{k}}^\top (x-x^{k+1}) -\frac{1}{2s}\norm{x-\overline{x}^{k+1}}_{2}^2+\frac{1}{2s}\norm{x-y^{k}}_{2}^2.
    \label{before-xBar}
\end{equation}
Recalling the definition of $r^{k+1} = x^{k+1} - \overline{x}^{k+1}$, we have
\begin{align}
    \norm{x-\overline{x}^{k+1}}_{2}^2 = \big\|x-x^{k+1}\big\|_2^2+
    \big\|r^{k+1}\big\|_2^2+
    2(r^{k+1})^\top(x-x^{k+1}).
    \label{xBar}
\end{align}
Using \eqref{xBar} in \eqref{before-xBar} yields 
\begin{equation}
    \begin{split}
        f(x^{k+1})-f(x) &\leq \epsilon_2^{k} + {\epsilon_1^{k}}^\top (x-x^{k+1}) -\frac{1}{2s}\big\|x-x^{k+1}\big\|_2^2-\frac{1}{2s}\big\|r^{k+1}\big\|_2^2
        \\
        &\quad-\frac{1}{2s}(r^{k+1})^\top(x-x^{k+1})+\frac{1}{2s}\norm{x-y^{k}}_{2}^2
        \\
        &\leq \epsilon_2^{k} + {\epsilon_1^{k}}^\top (x-x^{k+1}) -\frac{1}{2s}\big\|x-x^{k+1}\big\|_2^2
        \\
        &\quad-\frac{1}{2s}(r^{k+1})^\top(x-x^{k+1})+\frac{1}{2s}\norm{x-y^{k}}_{2}^2.
        \label{Eq:f-fAcceleratedMain}
    \end{split}
\end{equation}
Let us now substitute $y^k$ and $x$ by, 
\begin{align}
    y^k &= x^k +\beta_k(x^k-x^{k-1}) \label{Eq:yk-accelerated}\\
    x &= \alpha_k^{-1} x^\star + (1 - \alpha_k^{-1}) x^k \label{Eq:x-accelerated},
\end{align}
where \eqref{Eq:yk-accelerated} follows from the definition of the acceleration scheme \eqref{Eq:AcceleratedPG}, and~\eqref{Eq:x-accelerated} is a choice that we make to simplify the analysis.\footnote{Note that $y^k \rightarrow x^k$ as $x^k \rightarrow x^\star$.} $\{\alpha_k\}_{k\geq1}$ is a given parameter sequence that satisfies $\alpha_0 = 1$,  $\alpha_k \geq 1$ and $\beta_k = \frac{\alpha_{k-1}-1}{\alpha_k}$.  \eqref{Eq:f-fAcceleratedMain} can now be expanded as
\begin{equation}
    \begin{split}
        f(x^{k+1})-f(\alpha_k^{-1} x^\star + (1 - \alpha_k^{-1}) x^k) &\leq \epsilon_2^{k} + {\epsilon_1^{k}}^\top (\alpha_k^{-1} x^\star + (1 - \alpha_k^{-1}) x^k-x^{k+1})         
        \\                                                                                                                                                          &\quad-\frac{1}{2s}\norm{\alpha_k^{-1} x^\star + (1 - \alpha_k^{-1}) x^k-x^{k+1}}_{2}^2
        \\
        &\quad+\frac{1}{2s}\norm{\alpha_k^{-1} x^\star + (1 - \alpha_k^{-1}) x^k-y^{k}}_{2}^2
        \\
        &\quad-\frac{1}{2s}(r^{k+1})^\top(\alpha_k^{-1} x^\star + (1 - \alpha_k^{-1}) x^k-x^{k+1})
        \label{v-bvleq}.
    \end{split}
\end{equation}
Since $\alpha_k^{-1} \in ]0,1]$, and from the convexity of $f$, we have
\begin{equation}
    \begin{split}
        f(x^{k+1})-f(\alpha_k^{-1} x^\star + (1 - \alpha_k^{-1}) x^k) &\geq f(x^{k+1}) + (1-\alpha_k^{-1}) f(x^\star)
        \\
        &\quad- (1 - \alpha_k^{-1}) f(x^k) - f(x^\star)
        \\
        &= f(x^{k+1})- f(x^\star) - (1-\alpha_k^{-1}) ( f(x^k)-f(x^\star)). \label{v-bv}
    \end{split}
\end{equation}
Let us now define the new sequences $\{v^k\}$ and $\{u^k\}$ by
\begin{align}
    \label{u}
    & u^{k} := x^\star + (\alpha_k - 1) x^k-\alpha_k y^{k} = x^\star- (x^k+(\alpha_{k-1}-1)( x^{k} -x^{k-1}))
    \\
    \label{v}
    & v^{k} = f(x^k)-f(x^\star). 
\end{align}
From these we can obtain 
\begin{align}
    u^{k+1} := x^\star + (\alpha_k - 1) x^k-\alpha_k x^{k+1} =     x^\star-(x^{k+1} + (\alpha_k - 1) (x^{k+1}-x^k)),
    \label{u+}
\end{align}
by using $\beta_k = (\alpha_{k-1}-1)/\alpha_k$ and $y^k = (1+\beta_k)x^k-\beta_k x^{k-1}$.

Rewriting \eqref{v-bvleq} in terms of the newly defined sequences, $\{u^k\}$ and $\{v^k\}$, and using \eqref{v-bv} with $c_k := 1 - \alpha^{-1}$, as well as \eqref{u} and \eqref{u+} we obtain
\begin{equation}
    \begin{split}
        v^{k+1} - c_k v^{k} &\leq \epsilon_2^{k} + \frac{1}{\alpha_k}{\epsilon_1^{k}}^\top u^{k+1} -\frac{1}{2s\alpha_k^2}\norm{u^{k+1}}_{2}^2+\frac{1}{2s\alpha_k^2}\norm{u^{k}}_{2}^2 
        \\
        &\quad-\frac{1}{2s}\big\|r^{k+1}\big\|_2^2-\frac{1}{2s\alpha_k}(r^{k+1})^\top u^{k+1}. \label{v-bv(u)}
    \end{split}
\end{equation}
Rearranging \eqref{v-bv(u)} we obtain
\begin{equation}
    \begin{split}
        v^{k+1}+\frac{1}{2s}\big\|r^{k+1}\big\|_2^2 +\frac{1}{2s\alpha_k^2}\norm{u^{k+1}}_{2}^2 &\leq \epsilon_2^{k} + \frac{1}{\alpha_k}{\epsilon_1^{k}}^\top u^{k+1} + c_k v^{k}
        \\
        &\quad+\frac{1}{2s\alpha_k^2}\norm{u^{k}}_{2}^2-\frac{1}{2s\alpha_k}(r^{k+1})^\top u^{k+1}.
    \end{split}
\end{equation}
Multiplying both sides by $\alpha_k^2$,
\begin{equation}
    \begin{split}
        \alpha_k^2v^{k+1}
        +\frac{\alpha_k^2}{2s}\big\|r^{k+1}\big\|_2^2
        +\frac{1}{2s}\norm{u^{k+1}}_{2}^2 &\leq \alpha_k^2\epsilon_2^{k} + \alpha_k {\epsilon_1^{k}}^\top u^{k+1} + \alpha_k^2c_k v^{k}
        \\
        &\quad+\frac{1}{2s}\norm{u^{k}}_{2}^2-\frac{\alpha_k}{2s}(r^{k+1})^\top u^{k+1}.
        \label{eq:124}        
    \end{split}
\end{equation}
Applying \eqref{eq:124} recursively, and substituting
$\alpha_k^2c_k = \alpha_k^2-\alpha_k = \alpha_{k-1}$ yields

\begin{align}
   \alpha_k^2v^{k+1}
   +\frac{\alpha_k^2}{2s}\big\|r^{k+1}\big\|_2^2
   +\frac{1}{2s}\norm{u^{k+1}}_{2}^2 &\leq \alpha_k^2\epsilon_2^{k} + \alpha_k {\epsilon_1^{k}}^\top u^{k+1} + \alpha_{k-1} v^{k}  
   \\
   &\quad+\frac{1}{2s}\norm{u^{k}}_{2}^2-\frac{\alpha_k}{2s}(r^{k+1})^\top u^{k+1},\notag
   \\
   & \dots, \notag
   \\
   \alpha_1^2v^{2}
   +\frac{\alpha_1^2}{2s}\big\|r^{2}\big\|_2^2   
   +\frac{1}{2s}\norm{u^{2}}_{2}^2 &\leq \alpha_1^2\epsilon_2^{2} + \alpha_1 {\epsilon_1^{2}}^\top u^{2} + \alpha_{0} v^{1}  
   \\
   &\quad+\frac{1}{2s}\norm{u^{1}}_{2}^2-\frac{\alpha_1}{2s}(r^{2})^\top u^{2}.\notag
\end{align}
Adding both sides of all inequalities,
\begin{equation}
    \begin{split}
        &\alpha_k^2v^{k+1}
        +\sum_{i=1}^{k}\frac{\alpha_i^2}{2s}\big\|r^{i+1}\big\|_2^2
        +\frac{1}{2s}\norm{u^{k+1}}_{2}^2 + \sum_{i=1}^k (\alpha_{i-1}^2-\alpha_{i-1}) v^{i} 
        \\
        &\leq \sum_{i=1}^k \alpha_i^2\epsilon_2^{i}+\frac{1}{2s}\norm{u^{1}}_{2}^2
        + \sum_{i=1}^k \alpha_i {\epsilon_1^{i}}^\top u^{i+1}  + \alpha_{0} v^{1}
        -\sum_{i=1}^{k}\frac{\alpha_i}{2s}(r^{i+1})^\top u^{i+1}.
    \end{split}
\end{equation}
Substituting $\alpha_{i-1}^2-\alpha_{i-1} = \alpha_{i-2}^2$ and $\alpha_0 = 1$ gives,
\begin{equation}
    \begin{split}
        &\alpha_k^2v^{k+1}
        +\sum_{i=1}^{k}\frac{\alpha_i^2}{2s}\big\|r^{i+1}\big\|_2^2
        +\frac{1}{2s}\norm{u^{k+1}}_{2}^2 + \sum_{i=1}^k \alpha_{i-2} v^{i}
        \\
        &\leq \sum_{i=1}^k \alpha_i^2\epsilon_2^{i} + \sum_{i=1}^k \alpha_i {\epsilon_1^{i}}^\top u^{i+1}  + v^{1}\notag  
        +\frac{1}{2s}\norm{u^{1}}_{2}^2-\sum_{i=1}^{k}\frac{\alpha_i}{2s}(r^{i+1})^\top u^{i+1}.
    \end{split}
\end{equation}
For a positive sequence $\{\alpha_k\}_{k\geq 0}$ and because $x^\star$ is a (global) minimizer, $\sum \alpha_{i-2} v^{i}\geq 0$ is always satisfied; hence the following holds
\begin{equation}
    \begin{split}
        \label{130}
        \alpha_k^2v^{k+1} &\leq \alpha_k^2v^{k+1}
        +\sum_{i=1}^{k}\frac{\alpha_i^2}{2s}\big\|r^{i+1}\big\|_2^2
        +\frac{1}{2s}\norm{u^{k+1}}_{2}^2 + \sum_{i=1}^k \alpha_{i-2} v^{i} 
        \\
        &\leq \sum_{i=1}^k \alpha_i^2\epsilon_2^{i} + \sum_{i=1}^k \alpha_i \bigg({\epsilon_1^{i}} -\frac{1}{s} r^{i+1}\bigg)^\top u^{i+1}  + v^{1}  +\frac{1}{2s}\norm{u^{1}}_{2}^2.        
    \end{split}
\end{equation}
From \eqref{Eq:f-fAcceleratedMain} with $k=0$ and $x=x^\star$, we have
\begin{equation}
    \begin{split}
   \label{131}
        v^{1} = f(x^{1})  -f(x^\star) 
        &\leq \epsilon_2^{0} + \bigg({\epsilon_1^{0}}-\frac{1}{2s}{r^{1}}\bigg)^\top (x^\star-x^{1}) -\frac{1}{2s}\big\|x^\star-x^{1}\big\|_2^2
        \\
        &\quad+\frac{1}{2s}\norm{x^\star-x^{0}}_{2}^2, 
    \end{split}
\end{equation}
since $y^0 = x^0$. From the definition of $\{u^k\}$ in \eqref{u+} we have
\begin{equation}
    \begin{split}
        \frac{1}{2s}\norm{u^{1}}_{2}^2 &= \frac{1}{2s}\norm{x^\star + (\alpha_0 - 1) x^0-\alpha_0 x^{1}}_{2}^2,\\
         &= \frac{1}{2s}\norm{x^\star - x^{1}}_{2}^2\label{Eq:BeforeQuasiAcc},
    \end{split}
\end{equation}
where we have used the initialization $\alpha_0 = 1$. 
Substituting for $v^{k+1}$ and combining \eqref{131} and \eqref{Eq:BeforeQuasiAcc} with \eqref{130} yields
\begin{equation}
    \begin{split}
       \alpha_k^2(f(x^{k+1})-f(x^\star)) &\leq
       \sum_{i=1}^k \alpha_i^2\epsilon_2^{i} + \sum_{i=1}^k \alpha_i \bigg({\epsilon_1^{i}} -\frac{1}{s} r^{i+1}\bigg)^\top u^{i+1}  +\frac{1}{2s}\norm{u^{1}}_{2}^2
       \\
       &+ \epsilon_2^{0} + {\epsilon_1^{0}}^\top (x^\star-x^{1}) -\frac{1}{2s}\big\|x^\star-x^{1}\big\|_2^2-\frac{1}{2s}{r^{1}}^\top(x^\star-x^{1})
       \\
       &\quad+\frac{1}{2s}\norm{x^\star-x^{0}}_{2}^2,
       \\
       &=\sum_{i=1}^k \alpha_i^2\epsilon_2^{i} + \sum_{i=1}^k \alpha_i \bigg({\epsilon_1^{i}} -\frac{1}{s} r^{i+1}\bigg)^\top u^{i+1}
       \\
       &+ \alpha_0^2\epsilon_2^{0} + \alpha_0 \bigg({\epsilon_1^{0}} -\frac{1}{s} r^{1}\bigg)^\top u^{1}+\frac{1}{2s}\norm{x^\star-x^{0}}_{2}^2
       \\
       &=\sum_{i=0}^k \alpha_i^2\epsilon_2^{i} + \sum_{i=0}^k \alpha_i \bigg({\epsilon_1^{i}} -\frac{1}{s} r^{i+1}\bigg)^\top u^{i+1}
       \\
       &\quad+\frac{1}{2s}\norm{x^\star-x^{0}}_{2}^2.        
    \end{split}
\end{equation}
Dividing both sides by $\alpha_k^2$ completes the proof of the theorem.
\subsection{Proof of Corollary~\ref{corol:2}}
\label{Subsec:Proofcorol2}
Applying Cauchy-Schwarz inequality to \eqref{Eq:Theorem5} yields
\begin{align}
    f(x^{k+1})-f(x^\star) &\leq \frac{1}{\alpha_k^2} 
    \Bigg[\sum_{i=0}^k \alpha_i^2\epsilon_2^{i} + \Big[\sum_{i=0}^k \alpha_i \Big(\norm{\epsilon_1^{i}}_2  
    +\frac{1}{2s}\norm{r^{i+1}}_2 \Big)\Big]\norm{u^{k+1}}_2\\
    &\quad+\frac{1}{s}\norm{x^\star-x^{0}}_{2}^2\Bigg]\notag
    \label{ProofEq:Theorem1}.
\end{align}
Using the bound $\norm{r^{i+1}}_2 \leq \sqrt{2s\epsilon_2^{i}}$ from Lemma \ref{Lem:BackwardProxError} completes the proof of Corollary \ref{corol:2}.

We have by definition \ref{u} and \ref{u+}
\begin{align}
    & u^{k} = x^\star + (\alpha_k - 1) x^k-\alpha_k y^{k} = x^\star- (x^k+(\alpha_{k-1}-1)( x^{k} -x^{k-1})), \\
    & u^{k+1} = x^\star + (\alpha_k - 1) x^k-\alpha_k x^{k+1} =     x^\star-(x^{k+1} + (\alpha_k - 1) (x^{k+1}-x^k)).
\end{align}
By triangle inequality of the vector norm, we have
\begin{align}
    &\norm{u^{k}}_2\notag
    \leq \norm{(\alpha_k - 1)(x^k-x^\star)}_2+\alpha_k \norm{y^{k}-x^\star}_2,\notag
    \\
    &\norm{u^{k+1}}_2\notag
    \leq \vert \alpha_k - 1 \vert \norm{x^k-x^\star}_2+\alpha_k \norm{x^{k+1}-x^\star}_2\notag
\end{align}
By the nonexpansivity of the displacement operator, i.e., $\mathbf{I} - s \nabla g$, where $\mathbf{I}$ is the identity operator, we obtain
\begin{align}
    \norm{u^{k+1}}_{2}-\norm{u^{k}}_{2} 
    &\leq \alpha_k\bigg|\norm{x^{k+1}-x^\star}_2-\norm{y^k-x^\star}_2\bigg|,\\
    &\leq \alpha_k\bigg|\norm{r^{k+1}}_2+s_k \norm{\epsilon_1^{k}}_2 + C_{\rho,s_{k_0}}\bigg|, \quad \forall s_k \leq \frac{1}{L},\notag
\end{align}
where we have used inequality \eqref{Eq:x-xstar+E2}. Rearranging and taking into account that all the terms inside the absolute value are nonnegative, we obtain
\begin{align}
    \norm{u^{k+1}}_{2} &\leq \norm{u^{k}}_{2}+ \alpha_k\bigg(\norm{r^{k+1}}_2+s_k \norm{\epsilon_1^{k}}_2 + C_{\rho,s_{k_0}}\bigg), \quad \forall s_k \leq \frac{1}{L}.
\end{align}
By backward induction and by dropping second error terms we obtain the approximate bound of \eqref{corol2:2}.
\subsection{Proof of Theorem~\ref{Theorem2}}
\label{Subsec:Proofprop5}
This result is about the accelerated version of approximate PGD, but with random proximal computation error $\epsilon_{2_\Omega}$, component-wise bounded gradient error $\epsilon_{1_\Omega}$ and bounded residuals $\norm{x_{_\Omega}^k-x^\star}_2$. As the algorithm generates a sequence of random vectors $\{x_{_\Omega}^k\}$, the residual vector sequence $\{r_{_\Omega}^k\}$ will also be a random.
Let $\nu_{_\Omega} = {\epsilon_1^{i}} -\frac{1}{s} r^{i+1}$ and let $\{T_k\}$ denote the second error term in \eqref{Eq:Theorem1} [Theorem \ref{Theorem1}], i.e.,
\begin{equation}
T_k =
\left\{
\begin{array}{ll}
0, &\,\, k=0
\\
\sum_{i=1}^{k}\alpha_i{\nu_{_\Omega}^{i}}^\top u_{_\Omega}^i,&\,\,k = 1, 2, \ldots\,,
\end{array}
\right.
\label{138a}
\end{equation}
where
\begin{equation}
    u_{_\Omega}^i = x^\star- x_{_\Omega}^i + (1-\alpha_{i-1})(x_{_\Omega}^{i} -x_{_\Omega}^{i-1}).
    \label{138c}
\end{equation}
The first step is to show that $\{T_k\}$ is a martingale. Recall that a sequence of random variables $T_0,T_1,\dots$ is a martingale with respect to the sequence $X_0,X_1,\dots$ if, for all $k\geq0$, the following conditions hold:
\begin{itemize}
    \item $T_k$ is a function of $X_0,X_1,\dots,X_k$;
    \item $E(|T_k|) < \infty$;
    \item $E(T_{k+1}|X_0,X_1,\dots,X_k) = T_k$.
\end{itemize}
A sequence of random variables $T_0,T_1,\dots$ is called a martingale when it is a martingale with respect to itself. That is, $E(|T_k|) < \infty$, and $E(T_{k+1}|T_0,T_1,\dots,T_k) = T_k$.
We now show that Assumptions \ref{Assum:RandomVectors} and \ref{Assum:RandomVectors2} imply that $\{T_k\}_{k\geq 0}$ is a martingale. 
Specifically, \eqref{Eq:AssumptionRandomZeroMean} and \eqref{Eq:ResidAssumptionRandomZeroMean}, we have
\begin{equation}
\mathbb{E}\big[\nu_{_\Omega}^{k}\big\vert \nu_{_\Omega}^{1}\dots \nu_{_\Omega}^{k-1}\big] =  \mathbb{E}\big[\nu_{_\Omega}^{k}\big] = 0.\notag 
\end{equation}
And from \eqref{Eq:AssumptionRandomInnerProduct} and \eqref{Eq:ResidAssumptionRandomInnerProduct}, we have
\begin{equation}
\mathbb{E}\big[{\nu_{_\Omega}^{k}}^\top x_{_\Omega}^k\big\vert \nu_{_\Omega}^{1}\dots \nu_{_\Omega}^{k-1},x_{_\Omega}^{1}\dots x_{_\Omega}^{k-1}\big] = \mathbb{E}\big[{\nu_{_\Omega}^{k}}^\top x_{_\Omega}^k\big] = 0.\notag 
\end{equation}
We have from \eqref{138a},
\begin{equation}
        T_k = T_{k-1} + \alpha_k{\nu_{_\Omega}^{k}}^\top u_{_\Omega}^k.
\end{equation}
Substituting for $u_{_\Omega}^k$ using \eqref{138c} gives, 
\begin{equation}
        T_k = T_{k-1} + \alpha_k\alpha_{k-1}{\nu_{_\Omega}^{k}}^\top (x^\star-x_{_\Omega}^k)+\alpha_k(1-\alpha_{k-1}){\nu_{_\Omega}^{k}}^\top (x^\star-x^{k-1}).
\end{equation}
Taking the conditional expectation from both sides yields
\begin{align}
        \mathbb{E}\big[T_k|T_1\dots T_{k-1}\big] &= \mathbb{E}\big[T_{k-1} + \alpha_k\alpha_{k-1}{\nu_{_\Omega}^{k}}^\top (x^\star-x_{_\Omega}^k)
        \\
        &\quad\notag+\alpha_k(1-\alpha_{k-1}){\nu_{_\Omega}^{k}}^\top (x^\star-x^{k-1})|T_1\dots T_{k-1}\big]
        \\
         \notag
         &= \mathbb{E}\big[T_{k-1}|T_1\dots T_{k-1}\big]  + \mathbb{E}\big[\alpha_k\alpha_{k-1}{\nu_{_\Omega}^{k}}^\top (x^\star-x_{_\Omega}^k)
         \\
        &\quad\notag+\alpha_k(1-\alpha_{k-1}){\nu_{_\Omega}^{k}}^\top (x^\star-x^{k-1})|T_1\dots T_{k-1}\big]
        \\
        \notag
         &= T_{k-1}  + \mathbb{E}\big[\alpha_k\alpha_{k-1}{\nu_{_\Omega}^{k}}^\top (x^\star-x_{_\Omega}^k)|T_1\dots T_{k-1}\big]
         \\
        &\quad\notag+\mathbb{E}\big[\alpha_k(1-\alpha_{k-1}){\nu_{_\Omega}^{k}}^\top (x^\star-x^{k-1})|T_1\dots T_{k-1}\big]
    \\
        \label{4.92}
         &= T_{k-1}  + \alpha_k\alpha_{k-1}\mathbb{E}\big[{\nu_{_\Omega}^{k}}^\top (x^\star-x_{_\Omega}^k)|T_1\dots T_{k-1}\big]
         \\
        &\quad\notag+\alpha_k(1-\alpha_{k-1})\mathbb{E}\big[{\nu_{_\Omega}^{k}}^\top |T_1\dots T_{k-1}\big](x^\star-x^{k-1})
    \\
        \label{4.93}
         &= T_{k-1}  + \alpha_k\alpha_{k-1}\mathbb{E}\big[{\nu_{_\Omega}^{k}}^\top (x^\star-x_{_\Omega}^k)|T_1\dots T_{k-1}\big]
    \\
        \label{4.94}
         &= T_{k-1}  + \alpha_k\alpha_{k-1}\mathbb{E}\big[{\nu_{_\Omega}^{k}}^\top x^\star-{\nu_{_\Omega}^{k}}^\top x_{_\Omega}^k|T_1\dots T_{k-1}\big]
    \\
        \label{4.95}
         &= T_{k-1}  + \alpha_k\alpha_{k-1}\mathbb{E}\big[{\nu_{_\Omega}^{k}}^\top x^\star|T_1\dots T_{k-1}\big]
         \\
        &\quad\notag-\alpha_k\alpha_{k-1}\mathbb{E}\big[{\nu_{_\Omega}^{k}}^\top x_{_\Omega}^k|T_1\dots T_{k-1}\big]
    \\
        \label{4.96}
         &= T_{k-1}  + \alpha_k\alpha_{k-1}\mathbb{E}\big[\nu_{_\Omega}^{k}\big]^\top x^\star-\alpha_k\alpha_{k-1}\mathbb{E}\big[\mathbb{E}\big[\nu_{_\Omega}^{k}|x_{_\Omega}^k\big]^\top x_{_\Omega}^k\big]
    \\
        \label{4.97}
         &= T_{k-1},
\end{align}
From \eqref{4.92} to \eqref{4.93}, we factorised the deterministic vector difference, i.e., $x^\star-x^{k-1}$. From \eqref{4.96} to \eqref{4.97}, we used the zero mean error assumption, i.e, $E\big[\nu_{_\Omega}^{k}\big] = 0$ and the tower rule of conditional expectations, i.e., $\mathbb{E}\bigg[\mathbb{E}\big[\nu_{_\Omega}^{k}|x_{_\Omega}^k\big]^\top x_{_\Omega}^k\bigg]=\mathbb{E}\big[\nu_{_\Omega}^{k}\big]^\top x_{_\Omega}^k$. Therefore, $T_1,T_2,\dots,T_{k}$ is a martingale. 

In what follows, we establish upper bounds on the absolute value of the martingale $\{T_k\}$. By noticing that $\big\vert T_k - T_{k-1}\big\vert = \big\vert{\nu_{_\Omega}^k}^\top u_{_\Omega}^k \big\vert \leq \alpha_k\Big(\sqrt{n}\delta M_{\nabla g}+\sqrt{2\epsilon_2^k/s}\Big)\norm{u_{_\Omega}^k}_{2}$, where we have used Cauchy-Schwarz, etc. Lemma \ref{lem:3} then yields
\begin{align}
        |T_k|\leq \gamma|\delta|M_{\nabla g}\sqrt{n\sum_{i=1}^{k} i^{2}\norm{u_{_\Omega}^i}_{2}^2}+\gamma\sqrt{2s}\sqrt{\sum_{i=1}^{k} i^2\norm{u_{_\Omega}^i}_{2}^2\epsilon_2^i} \\\notag\leq \gamma|\delta|M_{\nabla g}\sqrt{n}\sum_{i=1}^{k} i\norm{u_{_\Omega}^i}_{2}+\gamma\sqrt{2s}\sum_{i=1}^{k} i\norm{u_{_\Omega}^i}_{2}\sqrt{\epsilon_2^i}
\end{align}
where $M_{\nabla g} = \underset{i \in \mathbb{N}_{+}} \sup\bigg\{\norm{\nabla g(x^i)}_{\infty}\bigg\}$ is the upper bound on the elements of the gradient.
Let $\{S_k\}$ denote the first error term in \eqref{Theorem1} [Theorem \ref{Theorem1}] i.e.,
\begin{equation}
        S_k = \sum_{i=0}^{k}\alpha_{i}^2\epsilon_{2_\Omega}^{i}.
\end{equation}
If $0 \leq \epsilon_{2_\Omega}^k \leq \varepsilon_0$ and $\alpha_k \leq k$, then applying Lemma \ref{lem:4} to $S_k = \sum_{i=0}^{k}\alpha_{i}^2\epsilon_{2_\Omega}^{i}$ with $0 \leq \epsilon_{2_\Omega}^k \leq \varepsilon_0$ and $\alpha_k \leq k$ yields
\begin{equation}
    S_k \leq \mathbb{E}\big[\sum_{i=0}^{k}\alpha_{i}^2\epsilon_{2_\Omega}^{i}\big] + \frac{\gamma}{2} \sqrt{\sum_{i=1}^{k}i^{4}(\epsilon_{2_\Omega}^i)^2}
    \leq
    \mathbb{E}\big[\sum_{i=0}^{k}\alpha_{i}^2\epsilon_{2_\Omega}^{i}\big] + \frac{\gamma}{2}\sum_{i=1}^{k}i^{2}\epsilon_{2_\Omega}^i,
\end{equation}
with probability at least $1-2\exp(-\frac{\gamma^2}{2})$. Applying Lemma \ref{lem:sumprob} completes the proof of Theorem \ref{Theorem2}.
\section{Experimental Results}
\label{Sec:Experiments}
We now experimentally assess the proposed bounds on an $l_1$ regularized model predictive control (MPC) and a synthesized LASSO problem. We consider a discrete linear time invariant (LTI) state space model of a spacecraft~\cite{hegrenaes2005spacecraft} and randomly generated data for the second experiment. For the MPC, approximation errors are simulated error sequences generated from a truncated Gaussian distribution. For the LASSO experiment, real round-off errors are generated using approximate algebra user-defined C++ functions that were previously developed and tested in~\cite{yun2020sspd}. Early termination errors are generated using adjustable CVX solver's tolerance parameters as described in~\cite{Grant11-CVX,grant2009cvx}.

We first describe the setup and experiments for the MPC problem.
\subsection{Model Predictive Control (MPC)}
\subsubsection{MPC problem formulation}
We consider a discrete-time state-space model
\begin{align}
    x(k+1) &= Ax(k) + Bu(k) \\ \label{discrete}
    y(k) &= Cx(k), \notag
\end{align}
where $x(k) \in \mathbb{R}^n$ is the state, $y(k) \in \mathbb{R}^m$ the observations and the control vector is given by $u(k) \in \mathbb{R}^p$. The matrices $A, B, C$ have dimensions $n \times n$, $n \times p$, and $ m \times n$, respectively. 
We assume that the state $x(k)$ is observable for all $k> 0$. Given an observation $x(k)$ at time $k$,, MPC algorithms search for a future control movement sequence of length $N_c$, i.e.,  $\{\Delta u(k+j)\}_{j=0}^{N_c-1}$, also known as the control horizon, that minimizes some defined error function between the predicted output sequence $\{y(k+l|k)\}_{l=1}^{N_p}$, of length $N_p$, and a defined set-point signal $r(t)$, where the latter is assumed to be constant $r(k) = r$ within the prediction time window $k \leq t \leq k + N_p$. $N_p$ is also known as the MPC prediction horizon. For $r(k)=0_{m}$, the MPC's objective is to drive the state of the system to the origin, a problem also known as classical regulator problem in the literature. Let $Y = [y(k+1|k) \quad y(k+2|k) \quad \dots \quad y(k+N_p|k)]$ denote the stacked output vector,
$X = [x(k+1) \quad x(k+2) \quad \dots \quad x(k+N_p)]$ denote the stacked state vector. The stacked input or future control vector, i.e., the solution of the MPC problem, is given by  $U = [\Delta u(k) \quad \Delta u(k+1) \quad \dots \quad \Delta u(k+N_c-1)]$. The sequentially calculated vector $Y \in \mathbb{R}^{m\times N_p}$ using the future control vector $U\in\mathbb{R}^{p\times (N_c -1)}$ satisfies the following equation,
\begin{equation}
    Y = \Psi x(k) + \Phi U \label{mpc-out},
\end{equation}
where
\begin{align}
    \Psi = \begin{bmatrix}
    CA\\CA^2\\CA^3\\ \vdots \\CA^{N_p}
    \end{bmatrix};
    \Phi = \begin{bmatrix}
    CB&0&0&\dots&0\\
    CAB&CB&0&\dots&0\\
    CA^2B&CAB&CB&\dots&0\\
    \vdots \\
    CA^{N_p-1}B&CA^{N_p-2}B&CA^{N_p-3}B&\dots&CA^{N_p-N_c}B
    \end{bmatrix}. \notag
\end{align}
The unconstrained optimization problem is therefore posed as minimizing
\begin{equation}
    J := (R_s - Y)^\top Q(R_s - Y) =  \norm{R_s - Y}_{Q}^2\label{MPC},
\end{equation}
where $R_s$ is an $m\times N_p$ matrix resulting from $N_p$ times stacking of the set-point vector $r(k)$, i.e.,  $R_s:=[r(k)\quad r(k)\quad \dots\quad r(k)]$, and  $Q$ is a constant matrix. Prior knowledge about the control vector $u$ is usually embedded into the problem in the form of a regularization function $h$ as follows,
\begin{equation}
    \min_{U \in\mathbb{R}^{p\times (N_c -1)}} \quad J := \norm{R_s - Y}_{Q}^2 + \lambda h(U)
\end{equation}
where $h: \mathbb{R}^{p\times N_c-1} \rightarrow \mathbb{R}$ is usually an $\ell_\zeta$-norm-like function and $\norm{.}_{Q}$ is the matrix weighted induced norm. 
The case when $\zeta = 0$ represents the $\ell_0$-regularized MPC problem or maximum support (also known as maximum hands-off \cite{nagahara2015maximum}) control, which is discontinuous and non-convex, and therefore difficult to solve.

If we set $\zeta = 1$ we obtain the $\ell_1$-regularized MPC problem, which is a non-smooth convex problem whose solution is a minimum fuel control sequence. The latter problem can be viewed as a relaxed $\ell_0$-regularized MPC problem which also yields the sparsest solution among other possible solutions. The discrete $\ell_1$-regularized MPC problem is equivalent to the following LASSO regression problem,
\begin{equation}
    \min_{U \in\mathbb{R}^{p\times (N_c -1)}} \quad J :=\norm{R_s - Y}_{Q}^2 + \lambda \norm{U}_1 \label{1-norm}
\end{equation}
where $U$ can be viewed as a sequence of control vector parameters and $\lambda$ is a  non-negative regularization parameter.

The case when $\zeta > 1$ yields a standard smooth optimal control problem. For instance, if the energy of the control signal is meant to be minimized to prevent overheating of some actuators or to reduce signal transmission cost, then regularizing \eqref{MPC} using the $\ell_2$-norm as the regularization function can be formulated as follows, 
\begin{equation}
    \min_{U \in\mathbb{R}^{p\times (N_c -1)}} \quad J :=\norm{R_s - Y}_{Q}^2 + \lambda \norm{U}_{R}^2 \label{2-norm}
\end{equation}
where $Q \in \mathbb{R}^{(n+m)\times (n+m)}$ and $R \in \mathbb{R}^{p\times p}$ are positive semi-definite weight matrices and $\lambda$ is a non-negative regularization parameter. \eqref{2-norm} is equivalent to the standard LQR (linear quadratic regulator) control problem,
\begin{equation}
    \min_{U \in\mathbb{R}^{p\times (N_c -1)}} \quad J :=(R_s - Y)^\top Q (R_s - Y) + U^\top R U,
\end{equation}
whose solution can be obtained in closed form:
\begin{equation}
    U = (\Phi^\top\Phi + R)^{-1}\Phi^\top(R_s-\Psi x(k)).
\end{equation}

Note that the solution to the following Elastic-Net MPC problem,
\begin{equation}
    \min_{U \in\mathbb{R}^{p\times (N_c -1)}} \quad J:= \norm{R_s - Y}_{Q}^2 + \lambda_{\ell_2} \norm{U}_{R}^2 + \lambda_{\ell_1} \norm{U}_1 \label{2-1-norm}
\end{equation}
yields a trade-off between minimum fuel and minimum energy solutions. However, the solution to \eqref{2-1-norm} as well as \eqref{1-norm} is usually not available in closed form; therefore, an iterative algorithm such as the introduced \textit{proximal gradient} algorithms [\eqref{Eq:PGNoiseless} and~\eqref{Eq:BasicPGNoiseless}] is a suitable choice to tackle both types of problems.

Let us now reformulate problem \eqref{2-1-norm} as a standard unconstrained LASSO. Choosing $\lambda_{\ell_2}=1$ and $\lambda_{\ell_1}=\lambda$, we know that the objective function of \eqref{2-1-norm} is equivalent to
\begin{subequations}
    \begin{equation}
            (R_s - Y)^\top Q (R_s - Y) +  U^\top R  U + \lambda \norm{ U}_1 \label{28a}
    \end{equation}
Substituting \eqref{mpc-out} in \eqref{28a}, we obtain
    \begin{equation}
    (R_s -\Psi x(k) - \Phi U)^\top Q (R_s - \Psi x(k) - \Phi U) +  U^\top R  U + \lambda \norm{ U}_1  
    \label{28b}.
    \end{equation}
Expanding and rearranging \eqref{28b} yields
    \begin{equation}
        \begin{split}
            (R_s -\Psi x(k))^\top Q(R_s - \Psi x(k)) - (R_s - \Psi x(k))^\top Q\Phi U\\ - U^\top \Phi^\top Q(R_s - \Psi x(k)) + U^\top( \Phi^\top Q\Phi + R) U + \lambda \norm{U}_1.
        \end{split}
    \end{equation}
Problem \eqref{2-1-norm} can then be written as a \textit{composite optimization problem}
\begin{equation}
    \begin{split}
            \min_{U \in\mathbb{R}^{p\times (N_c -1)}} F(U):=g(U)+h(U),
    \end{split}
    \label{Eq:MPCLASSO}
\end{equation}
where $g$ and $h$ are given in matrix induced norm notation as
\begin{equation}
    \begin{split}
            g(U) &:= \norm{\big(\Phi^\top Q\Phi 
            + R\big)^{\frac{1}{2}}U -\big(\Phi^\top Q\Phi + R\big)^{-\frac{1}{2}}\Phi^\top Q\big(R_s - \Psi x(k)\big)}_2^2;\\
            h(U) &:= \lambda \norm{U}_1.
            \label{(g(U),h(U))}
    \end{split}
\end{equation}
This regularized MPC problem is formulated as a  \textit{composite optimization problem}. This can also be viewed as a LASSO problem which will be solved iteratively using the \textit{inexact proximal gradient algorithm} \eqref{Eq:PGDApproximateRepeated} or its accelerated variant \eqref{Eq:AcceleratedPGRepeated}.
\end{subequations}
\subsubsection{Experimental setup for MPC}
We consider a discrete linear time invariant (LTI) state space model and we allow MPC to run for $N_p = N_c = 10$ time steps. We also let the proximal gradient algorithm perform $300$ iterations. For the accelerated case, we consider a more time-critical situation where we only allow for a maximum of $20$ iterations within the MPC horizon time window of length $N_p = N_c = 2$. Both proximal and gradient errors $\epsilon_1^k$ and $\epsilon_2^k$ are independent and identically distributed (i.i.d) sequences randomly generated from a univariate and multivariate truncated Gaussian distributions $\mathcal{N}(0,1)$ truncated to the interval $[0,\epsilon_0]$; and $\mathcal{N}(0,\mathcal{I}_n)$ truncated to the interval $[-\delta,\delta]$, respectively, where $\mathcal{I}_n$ is the identity matrix of dimension $n$. The truncation limits are assumed to be fixed and known throughout the simulation and with the same dimension as their corresponding error terms.

 In this experiment, we use the following simple spacecraft LTI discrete state-space model  $(A,B,C)$ \cite{hegrenaes2005spacecraft}
\begin{align}
    A &= \left[\begin{array}{ccccccc} 0 & 0 & 0.8416 & 0 & -1.267 & 0 & 0\\ 0 & 0 & 0 & 0 & 0 & -0.8107 & 0\\ -0.9763 & 0 & 0 & 0 & 0 & 0 & -0.04749\\ 0 & 0 & 0 & 0 & 0 & 0.8107 & 0\\ 0.5 & 0 & 0 & 0 & 0 & 0 & 0\\ 0 & 0.5 & 0 & 0 & 0 & 0 & 0\\ 0 & 0 & 0.5 & 0 & 0 & 0 & 0 \end{array}\right];\notag\\
    B &= \left[\begin{array}{cccc} 0.2353 & 0 & 0 & 0\\ 0 & 0.2306 & 0 & -0.2306\\ 0 & 0 & 0.2729 & 0\\ 0 & -0.2306 & 0 & 25000.0\\ 0 & 0 & 0 & 0\\ 0 & 0 & 0 & 0\\ 0 & 0 & 0 & 0 \end{array}\right];\notag\\
    C &= \mathcal{I}_{7}, \notag
\end{align}
with the following state and control vectors

\begin{align}
x(k) &= [\omega_1(k),\omega_2(k),\omega_3(k),\omega_w(k),\varepsilon_1(k),
\varepsilon_2(k),\varepsilon_3(k)]^\top \notag
\\
u(k) &=[\tau_1(k),\tau_2(k),\tau_3(k),\tau_w(k)]^\top.\notag
\end{align}
where $[\omega_1(k),\omega_2(k),\omega_3(k)]^\top$ are the angular velocities of the bodyframe relative to the orbit frame, $\omega_w(k)$ is the angular velocity of the wheels about their spin axes, and $\varepsilon_1(k),
\varepsilon_2(k),\varepsilon_3(k)$ are Euler parameters. 

\begin{figure}[H]
\centering
\includegraphics[width=10cm]{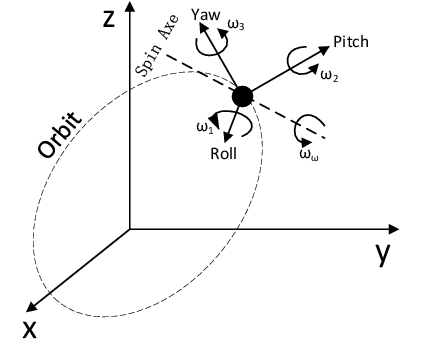}
\caption{Attitude Control \cite{hegrenaes2005spacecraft,yun2020sspd}: Seven states are considered here,  Roll, Pitch, Yaw, $\omega_1$, $\omega_2$, $\omega_3$, $\omega_w$, where Roll, Pitch, Yaw describe the rotating angles of the body frame relative to the orbit  frame, and $\omega_1$, $\omega_2$, $\omega_3$ are the corresponding angular velocities. $\omega_w$ is the angular velocity along the spin axis. The wheels are controlled by the input voltages, $\tau_1$, $\tau_2$, $\tau_3$, $\tau_w$ accordingly.}
\end{figure}

For the sake of simplicity, we relax the original parameter space and control constraints.

For simulation, we select the regularized MPC problem  matrices as follows,
\begin{align}
    Q &= \text{diag}(500.0, 500.0, 500.0, 10^{-7}, 1.0, 1.0, 1.0, 500.0, 500.0, 500.0, 10^{-7} , 1.0, 1.0, 1.0); \notag \\
    R &= \text{diag}(200.0, 200.0, 200.0, 1.0, 200.0, 200.0, 200.0, 1.0), \notag
\end{align}
and set the regularization parameter $\lambda = 16.79$. The Lipschitz constant of the quadratic term of problem \eqref{Eq:MPCLASSO} is $L = 8388$, and therefore, an initial stepsize of $\frac{1}{L}$ is adopted. The stepsize is then updated according to procedure B2 in \cite[Section 10.4.2]{beck2017first} with update parameter $\eta = 0.5$.\\ For the simulated errors, we use $\epsilon_{1_\Omega}^k = \nabla g(x^k)\odot \text{trand}(-\delta,\delta)$ where $g(x)$ is defined by~\eqref{(g(U),h(U))}, 
and $\text{trand}(a,b)$ is the doubly truncated normal distribution \cite{cha2013rethinking} with lower and upper truncation points $a$ and $b$, respectively. $\epsilon_2^{k} = \text{trand}(0,\epsilon_0)$ where $\delta$ and $\epsilon_0$ are variable scalar upper bounds on the gradient and proximal computation errors, respectively. The output of the distribution function $\text{trand}(l,u)$ is a vector randomly generated from the standard multivariate normal distribution truncated over the region $[l,u]$. 
\subsubsection{Results (Approximate PG-based MPC control of a spacecraft)}
\label{subsub:ResultsPG}
The deterministic and probabilistic bounds of Theorems \ref{prop:1} and \ref{prop:3} for the convex case are both plotted and superimposed with the bound  \eqref{schmidt1} of \cite{schmidt2011convergence} in \autoref{Figure2} and \autoref{Figure3}. The latter is denoted by \verb|Schmidt_1| and the proposed bounds are denoted by \verb|Thrm_1| and \verb|Thrm_2|, respectively. The dashed lines, \verb|Imprvm_Thrm_1| and \verb|Imprvm_Thrm_3|, represent the improvement of the proposed bounds over the bound given by \eqref{schmidt1}. Notice that we expect the effect of $\epsilon_1^k$ to be negligible near the optimum since the latter is proportional to the magnitude of the gradient. However, depending on the choice of the upper bound of $\epsilon_2^k$ in the proximal operation step \eqref{Eq:PGDApproximateRepeated}, the effect of the error $\epsilon_2$ can still be significant and sometimes permanent even near the optimum as we will see in the next few examples.

In the presence of small gradient and proximal computation errors ($|\epsilon_1^k| \leq 2.2\times 10^{-12}; \epsilon_2^k \leq 10^{-12}$), the bounds in Theorem \ref{prop:1}, Theorem \ref{prop:2} and~\eqref{schmidt1} practically coincide as shown in \autoref{Figure2}.
\begin{figure}[H]
\centering
\captionsetup{justification=centering}
\includegraphics[width=10cm]{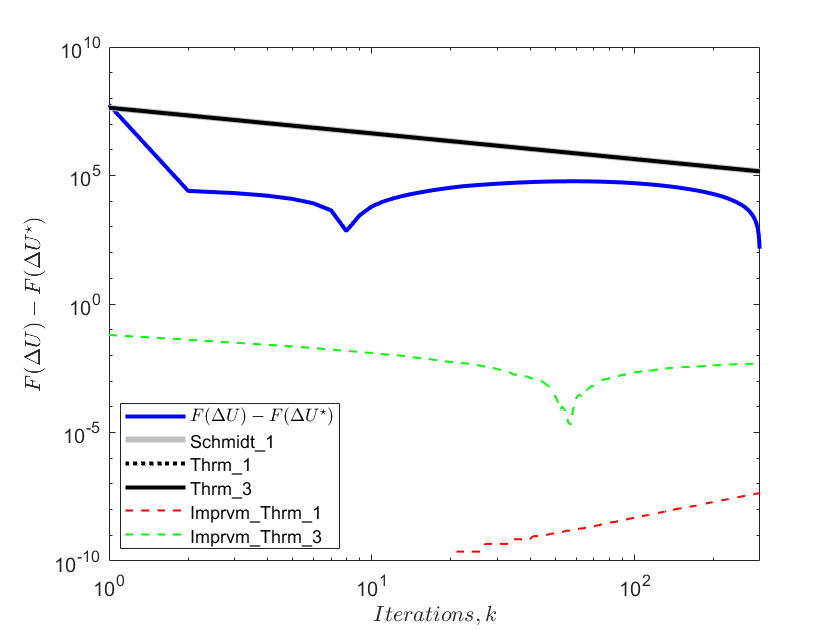}
\caption{Upper bounds based on Theorems \ref{prop:1} \& \ref{prop:3} vs Proposition 1 (\eqref{schmidt1}) in Schmidt et al. 2010 \cite{schmidt2011convergence} (with $\delta = 2.2\times 10^{-12}; \epsilon_0 = 10^{-12}$).\label{Figure2}}
\end{figure}

 In \autoref{Figure3} and \autoref{Figure4}, the simulated error magnitudes are larger and consequently a significant improvement can be seen. Notice how both proposed bounds (in Corollary \ref{corol:1} and Theorem \ref{prop:3}) become comparatively tighter as can be illustrated by the improvement that was achieved in both examples. 
\begin{figure}[H]
\centering
\captionsetup{justification=centering}
\includegraphics[width=10cm]{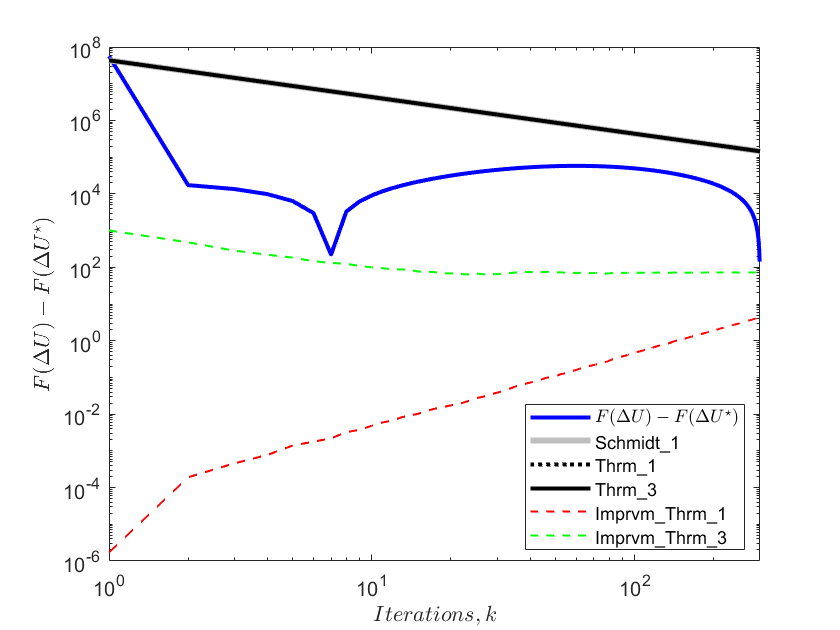}
\caption{Upper bounds based on Theorems \ref{prop:1} \& \ref{prop:3} vs Proposition 1 (\eqref{schmidt1}) in Schmidt et al. 2010 \cite{schmidt2011convergence} (with $\delta = 2.2\times 10^{-4}; \epsilon_0 = 10^{-4}$).\label{Figure3}}
\end{figure}
\begin{figure}[H]
\centering
\captionsetup{justification=centering}
\includegraphics[width=10cm]{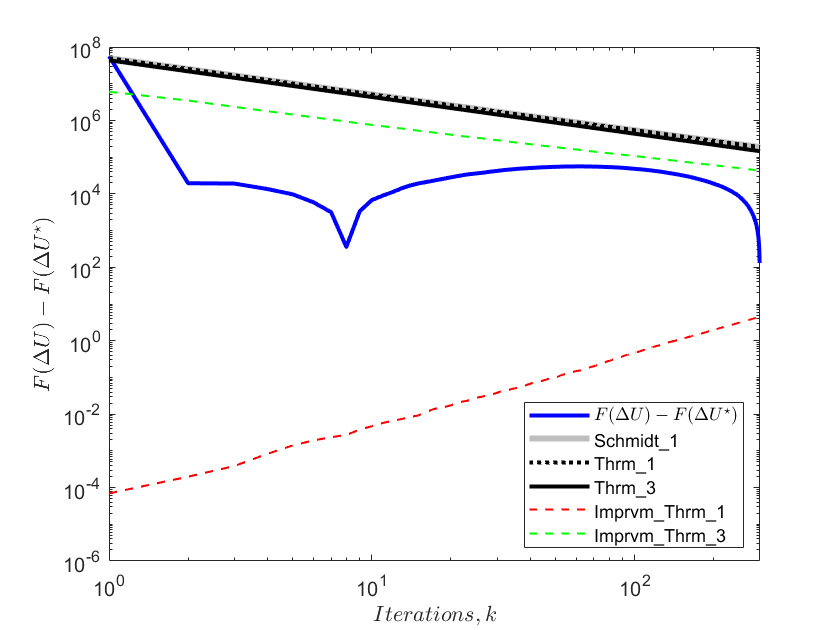}
\caption{Upper bounds based on Theorems \ref{prop:1} \& \ref{prop:3} vs Proposition 1 (\eqref{schmidt1}) in Schmidt et al. 2010 \cite{schmidt2011convergence} (with $ \delta = 2.2\times 10^{-1}; \epsilon_0 = 10^{-4}$).\label{Figure4}}
\end{figure}
The improvements in \autoref{Figure2} to \autoref{Figure4} suggest that our proposed bounds are tighter, and hence more accurate for application theoretical guarantees.  
The fact that probabilistic bounds are more stable than their deterministic counterparts is the consequence of the latter being adaptively computed from running error term and, therefore, continuously adjusted whilst the latter are \textit{a priori} bounds obtained from prior knowledge on the expected manifestation of error terms before any interaction with the computing machine. The latter can be calculated before running the algorithm, i.e, before error realizations are even generated, assuming machine precision ($\delta$) and solver tolerance ($\epsilon_0$) are specified beforehand together with a good estimate of the optimal solution. In other words, if an optimal solution $x^\star$ is known, then the bound of Theorem \ref{prop:3} becomes a function of the iteration counter $k$; therefore, a desired maximum number of iterations $k_0$ can be determined with probability $(p^k)\big(1-2\exp(-{\gamma^2}/{2})\big)$ if a specified level of suboptimaly or inefficiency $|f-f^\star|$ is to be tolerated, and vice versa.
\subsubsection{Results (Approximate Accelerated PG-based MPC control of a spacecraft)}
As a result of applying the accelerated PG \eqref{Eq:AcceleratedPG} to solve the MPC problem \eqref{Eq:MPCLASSO}, \autoref{Figure5} shows how both bounds of Theorems \ref{Theorem1} \& \ref{Theorem2} (denoted by \verb| Thrm_4| and \verb| Thrm_5| in Fig. 5  below) converge despite the amplified noise and remaining residual error. The bound \verb| Schmidt_2| refers to \eqref{schmidt2}.
\begin{figure}[H]
\centering
\includegraphics[width=10cm]{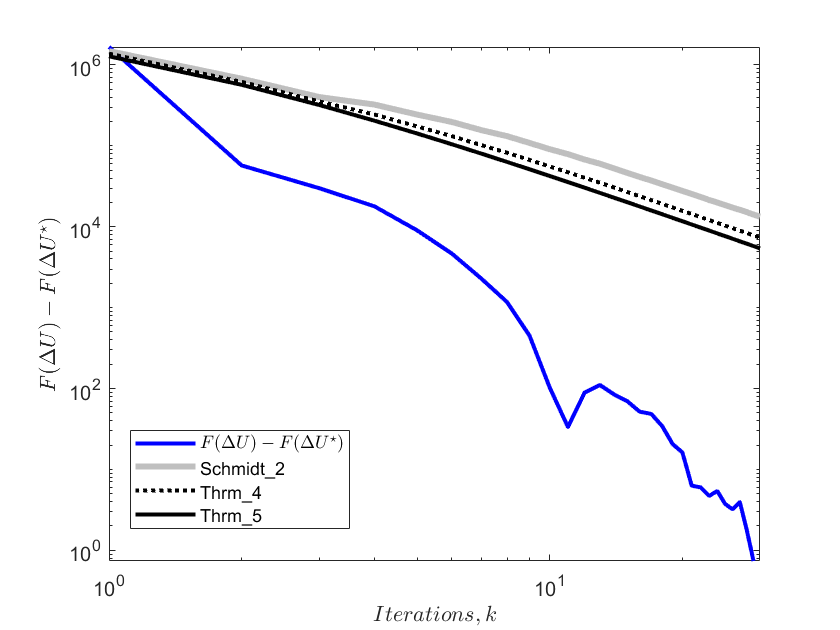}
\caption{Upper bounds based on Theorems \ref{Theorem1} \& \ref{Theorem2} vs Proposition 2 in Schmidt et al. 2010 \cite{schmidt2011convergence} (with $\delta = 2.2\times10^{-1}$; $\epsilon_0 = 10^{-4})$.\label{Figure5}}
\end{figure}
\autoref{Figure5} suggests that by using the results of Theorems \ref{Theorem1} and \ref{Theorem2} we can bound the suboptimality, i.e., $f-f^\star$, more tightly and the improvement is more significant in the accelerated case. However, the improvement is more remarkable in the early iterations of the accelerated proximal gradient algorithm and slightly drops over iterations which is different than the constant improvement in the nonaccelerated version. 

There is a notable oscillation pattern in the last iterations of the algorithm (blue line) that can be explained by the error amplification caused by acceleration.  

At this point, it is worth emphasizing the fact that bounds resulting from Theorem \ref{Theorem2} are \textit{a priori} probabilistic upper bounds which can be calculated before running the algorithm, and hence more robust to error fluctuations. In other words, if the optimal solution $x^\star$ is known, then the bound of Theorem \ref{Theorem2} becomes a single variable function of the iteration counter $k$; therefore, a desired $k_0$ can be determined with probability $1-2\exp(-{\gamma^2}/{2})$ if a specified level of suboptimaly or ineffeciency $|f-f^\star|$ is to be tolerated, and vice versa.

The improvements seen so far seem to be highly dependent on the subjective assumptions about the error models (truncated Gaussian). In order to present a more objective evidence, we apply the proposed bounds of Theorems \ref{prop:1}-\ref{prop:3} to assess the convergence of the approximate PG algorithm to the solution of LASSO problem under hardware and software errors as explained in the following section.
\newpage
\subsection{Synthetic LASSO}
\subsubsection{Experimental Setup}
We now apply the proposed bounds to analyze the convergence of the approximate
\textit{proximal gradient} algorithm when applied to solve randomly generated
LASSO problems:
\begin{equation*}
  \underset{x \in \mathbb{R}^n}{\text{minimize}}
  \,\,\,
  \frac{1}{2}\|Ax - y\|_2^2+\lambda \|x\|_1\,,
\end{equation*}
where $n = 100$ (dimension of $x$) and $A \in \mathbb{R}^{m \times n}$ has
$m=500$ rows. We run a total of $5$ random experiments for every algorithm
parameter selection. We mainly vary the bitwidth, the fraction width of the
fixed-point representation 
the CVX~\cite{Grant11-CVX} solver's precision
to approximate the proximal step \eqref{Eq:PGDApproximateRepeated}, and the
tolerance bound of the approximate proximal gradient (abstol). We record
and take the average over all $5$ experiments of the residual error in the
iterates $\|x-x^\star\|_2$, the error in the
function values $f-f^\star$, and the total number of iterations $k$. 
\subsubsection{Results}
\autoref{Figure6} to \autoref{Figure8} show the proposed convergence bounds in dashed (black) and continuous (black and gray) lines, the
error-free optimal bound (in green) as well as the original bound (red) in \eqref{schmidt1} for
the different tests. The parameter $\gamma$ is designed to generate $3$ probabilistic bounds which hold
with decreasing probabilities $1$, $0.5$ and $0.25$, respectively.  
\begin{figure}[H]
\centering
\captionsetup{justification=centering}
\label{fig:fig1}
\includegraphics[width=10cm]{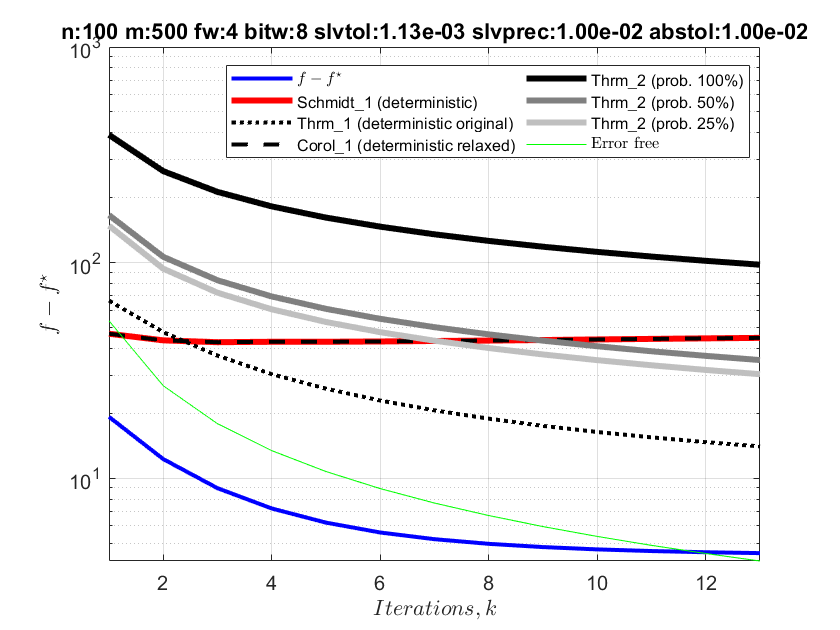}
\caption{Upper bounds based on Theorems \ref{prop:1} \& \ref{prop:3} and their corresponding corollaries vs \eqref{schmidt1}. slvprec: CVX solver's precision, bitw: bitwidth, fw: fraction width, \\abstol: overall tolerance of the PG algorithm.\label{Figure6}}
\end{figure}
\begin{figure}[H]
\centering
\captionsetup{justification=centering}
\label{fig:fig2}
\includegraphics[width=10cm]{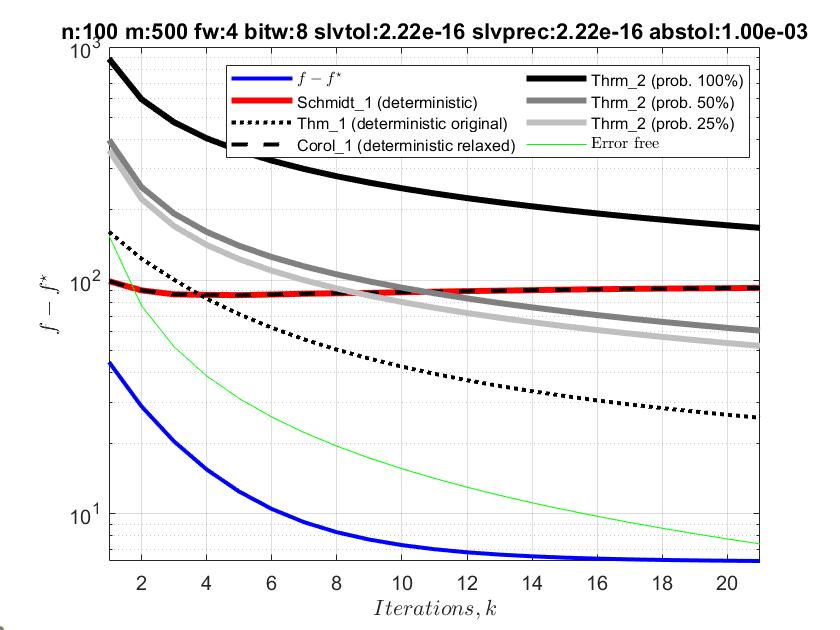}
\caption{Upper bounds based on Theorems \ref{prop:1} \& \ref{prop:3} and their corresponding corollaries vs \eqref{schmidt1}. slvprec: CVX solver's precision, bitw: bitwidth, fw: fraction width, \\abstol: overall tolerance of the PG algorithm.\label{Figure7}}
\end{figure}
\begin{figure}[H]
\centering
\captionsetup{justification=centering}
\includegraphics[width=10cm]{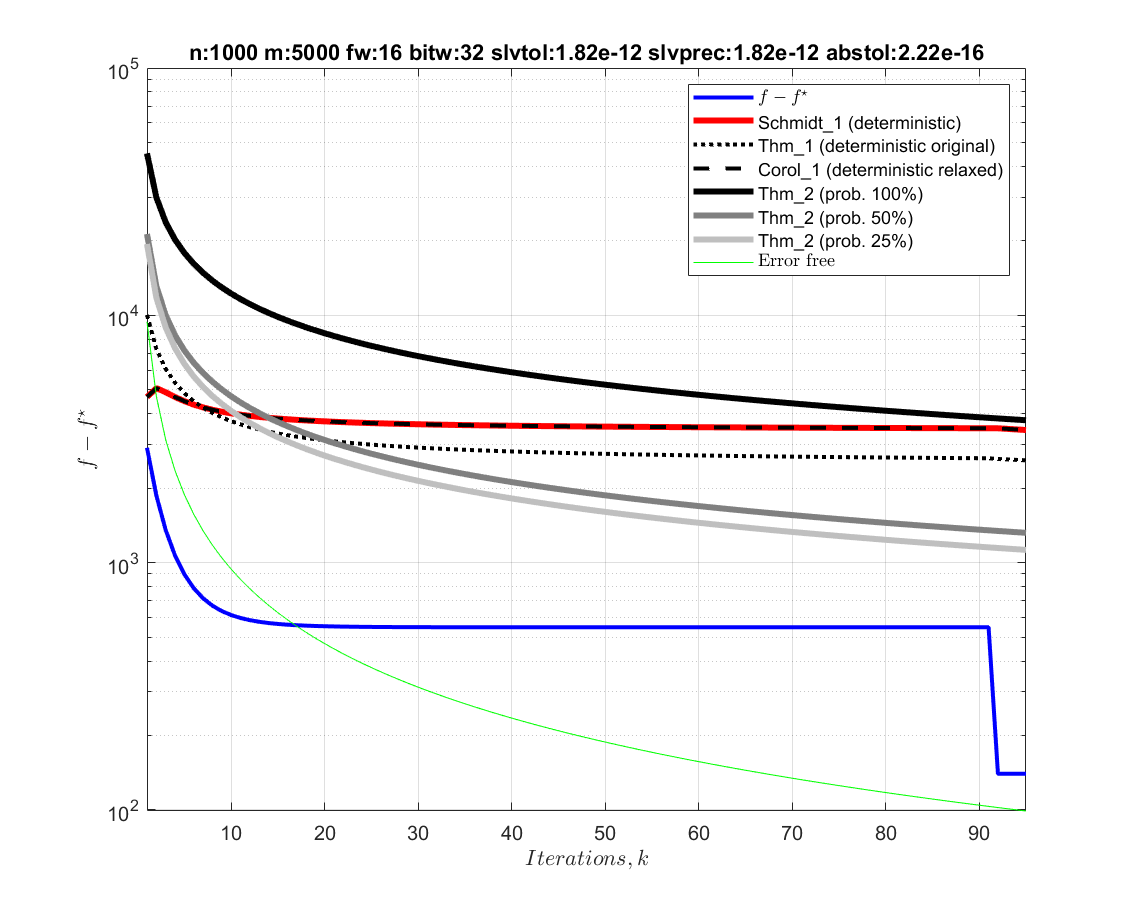}
\caption{Upper bounds based on Theorems \ref{prop:1} \& \ref{prop:3} and their corresponding corollaries vs \eqref{schmidt1}. slvprec: CVX solver's precision, bitw: bitwidth, fw: fraction width, \\abstol: overall tolerance of the PG algorithm.\label{Figure8}}
\end{figure}
\newpage
Overall, the proposed bounds of Theorems \ref{prop:1}-\ref{prop:3} give better approximations of the discrepancy caused by perturbations and they are more efficient asymptotically. Although the new deterministic bounds achieve better approximations (i.e., with smaller error terms) in \autoref{Figure6} and \autoref{Figure7}, probabilistic bounds are found to be more efficient in higher dimensions as can be seen from \autoref{Figure8}, where the LASSO problem was designed to have $n=1000$ variables and $m=5000$ examples and solved with lower machine epsilon of $2^{-16}$ and reduced solver tolerances.

As a necessary condition for convergence, we only
require the partial sums $\sum_{i=1}^{k}\epsilon_2^{i}$ and
$\sum_{i=1}^{k}\norm{\epsilon_1^{i}}_2$ to be in $o(k)$, in contrast to the
stronger condition $o(\sqrt{k})$ of \eqref{schmidt1}. For the probabilistic
bounds, we do not assume summability of the error terms but only require them
to be bounded. Consequently, the probabilistic bounds achieve better
approximations over iterations, they are less sensitive to error variations and
become tighter with decreasing  probability. If we relax our original bound of
Theorem \ref{prop:1} and use Lemma \ref{Lem:BackwardProxError} to bound the
sequence of the proximal residual error $\{r^k\}$, then our bound coincides
with the one in \eqref{schmidt1}, as depicted by the overlapping dashed and red
lines in \autoref{Figure6} to \autoref{Figure8}.



\section{Conclusions}
\label{Sec:Conclusion}
We assessed the convergence of the approximate proximal gradient algorithm in the presence of gradient
 and proximal computational
inaccuracies. We presented new tighter bounds and used them to verify a simulated (MPC) and a synthetic (LASSO) optimization examples solved on a
reduced-precision machine combined with reduced-precision solver. Following a probabilistic approach, we introduced new, and more robust, probabilistic upper bounds that can be used to verify the application before any interaction with the computing machine. Theoretically, we have also shown that some cumulative error terms follow a martingale property assuming error mean independence and data mean independence. In future works, we will try and relax the error assumptions to incorporate more general and more realistic perturbations into the analysis. We have also shown, in conformity with previous observations \cite{schmidt2011convergence}, that the standard momentum-based acceleration scheme has the potential drawback of noise amplification which can result in non-summable error terms; therefore, further work will need to be devoted to mitigate acceleration-error attenuation trade-off. 
\appendix
\section{Fixed-point representation}
\label{AppendixFixedPoint}
Let $\mathcal{F_{\text{u}}} \subseteq \mathbb{R}$ denote the unsigned fixed-point number system. A real number $x \not\in \mathcal{F_{\text{u}}}$ is rounded to an unsigned fixed-point number, i.e.,  $uI.F(x) \in \mathcal{F_{\text{u}}}$ as
\begin{equation}
    uI.F(x) = \sum_{i=0}^{W-1}b_i(x)2^{i-F},
\end{equation}
with $F, W\in \mathbb{N}_+$ and $F < W$. The corresponding dynamic range (DR) is given by 
\begin{equation}
    DR_{uI.F} = [0, \quad2^I-2^{-F}],
\end{equation}
where $I = W-F$.
The signed fixed-point representation I.F, or $sI.F(x) \in \mathcal{F_{\text{s}}}$ where $\mathcal{F_{\text{s}}} \subseteq \mathbb{R}$ can be obtained from $x\not\in \mathcal{F_{\text{s}}}$ by encoding the sign
of $x$ using one bit and this is typically done by taking the most
significant bit (MSB) of the integer part $I$. Although
this operation reduces the number of bits of the integer part from $I$ to $I-1$, the
two's complement approach handles negative numbers and therefore extends the DR
in the negative direction, i.e.,  $DR_{sI.F} = [-2^{I-1};2^{I-1}-2^{-F}]$, and the quantized value of x is now given by   
\begin{equation}
    sI.F(x) = \sum_{i=0}^{W-2}b_i2^{i-F} - b_{W-1}2^{I-1}.
    \label{Eq:Quantaz2}
\end{equation}
\section{Floating-point representation}

Let $\mathcal{F} \subseteq \mathbb{R}$ denote the floating-point number system. A real number \\$x=\pm(.d_1d_2d_3 \dots d_{t}d_{t+1}d_{t+2} \dots)\times \beta^e \in \mathbb{R}$ with $x \not\in \mathcal{F}$ is rounded to  $fl(x) \in \mathcal{F}$ as
\begin{equation}
    fl(x) = \pm m \times \beta^{e-t},
\end{equation}
with base $\beta$, precision $t$ and exponent $e$ satisfying $e_{\min} \leq e \leq e_{\max}$. $m$ is the \textit{mantissa} or also known as the \textit{significand} and satisfies $0 \leq m \leq \beta^t-1$. Numbers for which $m \geq \beta^{t-1}$ are called \textit{normalised numbers}. The following holds for \textit{round-to-nearest},
\begin{equation}
    \label{Eq:FpErrorBound}
    fl(x) = x(1+\epsilon), \quad |\epsilon| \leq \delta.
\end{equation}
\label{AppendixFloatingPoint}
\section{Some known results}
\renewcommand{\thesubsection}{\Alph{subsection}}
\begin{Theorem} \label{thm:1} 
Given a differentiable convex function $g: R^d \to R$ whose gradient is $L$-Lipschitz continuous and a convex proper (possibly non-smooth) function $h: R^d \to R \cup \{+\infty\}$, we define:
\begin{equation}
G(x,y) =  g(y) + \nabla g(y)^\top (x-y)+\frac{L}{2}\norm{x-y}_{2}^2+h(x)
\label{G1}
\end{equation}
for any $x \in \mathbb{R}^n$ and any $y \in dom(h)$. Then, for any fixed x,  $G(x, \cdot)$ is $L$-strongly convex; and for any fixed y, $G(\cdot, y)$ is $L$-strongly convex.
\end{Theorem}
\begin{proof}
See \cite{beck2017first}.
\end{proof}
\medskip
\begin{Theorem}\label{thm:2} 
  Let $g$: $E \rightarrow (-\infty, \infty]$ be a proper, closed, and
$c$-strongly convex function, where $c > 0$. Then,
\begin{itemize}
  \item $g$ has a unique minimizer $x^\star$; and
  \item $g(x)-g(x^\star) \geq \frac{c}{2}\|x-x^\star\|_2^2$, for all $x \in
    E$. 
\end{itemize}
\end{Theorem}
\begin{proof}
See \cite[Thm. \ 5.25]{beck2017first}.
\end{proof}
\medskip
\begin{Theorem}[\textbf{Fejer monotonicity of the sequence generated by the
proximal gradient method}] \label{Thm:Fejer} 
Let $\{x^k\}_{k\geq0}$ be the sequence generated by the proximal gradient
method for solving problem \eqref{Eq:ProblemRepeated}. Then for any $x^\star \in X^\star$ and $k\geq0$,
\begin{equation}
    \norm{x^{k+1} - x^\star}_2 \leq \norm{x^k - x^\star}_2
\end{equation}
\begin{equation}
    \norm{x^\star-x^i}_2 \leq \norm{x^\star-x^{i-1}}_2 \leq ... \leq \norm{x^\star-x^\star}_{2}
    \label{fejer}
\end{equation}
\end{Theorem}
\begin{proof} See \cite[Thm.\ 10.23]{beck2017first}\end{proof}
\medskip
\begin{Definition}[\textbf{Quasi-F\'ejer monotonicity of a sequence \cite[Def.\ 1.1]{combettes2001quasi}}] \label{Def:QuasiFejer} 
Relative to a nonempty target set $X^\star \in \mathbb{R}^n$, a sequence $\{x^k\}_{k \geq 0} \in \mathbb{R}^n$ is quasi-F\'ejer if, for any $k \geq 0$, the following inequality holds
\begin{equation}
    \norm{x^{k+1}-x^\star}_2 \leq \norm{x^{k}-x^\star}_2 + \varepsilon^k,
    \label{Eq:QuasiFejerDef}
\end{equation}
where $\{\varepsilon^k\}_{k \geq 0}$ is a positive absolutely summable sequence.
\end{Definition}
\medskip
\begin{Lemma}[\textbf{Azuma-Hoeffding inequality}] \label{lem:3} 
Let $E_1,\dots, E_n$ be a martingale such that $|E_k - E_{k-1}| \leq c_k$ almost surely, for $k = 2, \dots ,n$. Then for any $\gamma > 0$,
\begin{equation}
    \text{Pr}\bigg(|E_k-E_0|> \gamma\sqrt{\sum_{i=1}^kc_k^2}\bigg)\leq 2\exp(-\frac{\gamma^2}{2}).
\end{equation}
\end{Lemma}
\begin{proof}
See \cite[p.\ 36]{wainwright2019high}
\end{proof}
\medskip
\begin{Lemma}[\textbf{Hoeffding bound}] \label{lem:4} 
Suppose that the random variables $X_i$, $i= 1,\dots,n$ are
independent, and $X_i$ has mean $\mu_i$ and sub-Gaussian parameter $\sigma_i$. If we define $S = \sum_{i=1}^{k}X_i$ then for all $t \geq 0$, we have
\begin{equation}
    \text{Pr}\bigg(|S - \mathbb{E}\big[S\big]|> t\bigg)\leq 2\textup{exp}\bigg(\frac{-t^2}{2\sum_{i=1}^{k}\sigma_i^2}\bigg).    
\end{equation}
In particular, if $X_i \in [a, b]$ for all $i = 1, 2,\dots,n$, then
\begin{equation} 
    \text{Pr}\bigg(|S - \mathbb{E}\big[S\big]|\geq t\bigg)\leq 2\textup{exp}\bigg(\frac{-2t^2}{k(b-a)^2}\bigg).    
\end{equation}
\end{Lemma}
\begin{proof}
See \cite[p.\ 24]{wainwright2019high}.
\end{proof}
\begin{Lemma}
\label{lem:sumprob}
Let $(\Omega, \mathcal{F}, Pr)$ be a probability space and $T_i, i=1,\dots, m$, events in $\mathcal{F}$. Let $t_i$ be some function of a scalar variable $\gamma$. If we have
\begin{equation} 
    \text{Pr}\bigg(T_i\geq t_i(\gamma)\bigg)\leq P_i(\gamma),    
\end{equation}
for all $i=1,\dots, m$, then the following holds
\begin{equation} 
    \text{Pr}\bigg(\cup_{i=1}^m T_i\geq t_i(\gamma)\bigg)\leq \sum_{i=1}^m P_i(\gamma).    
\end{equation}
Equivalently we have
\begin{equation} 
    \text{Pr}\bigg(\cup_{i=1}^m T_i\leq t_i(\gamma)\bigg)\geq 1-\sum_{i=1}^m P_i(\gamma).    
\end{equation}
\end{Lemma}
\begin{Lemma}\label{Lem:BackwardProxError} 
Consider problem \eqref{Eq:ProblemRepeated} and let Assumption \ref{Ass:OptimizationProblem} hold. For $L, s > 0$, define $G$: $\mathbb{R}^n \times \mathbb{R}^n \rightarrow (-\infty, \infty]$ as the proper, closed, and $L$-strongly convex function
\begin{equation}
    G\big(y,\, x\big):=  g(y) + \nabla g(y)^\top (x-y)+\frac{1}{2s}\norm{x-y}_{2}^2+h(x), \notag
\end{equation}
Define $\widehat{y}^\star := \arg \min G\big(y,\, x\big)$ as the minimizer of $G$ with respect to $y$ when $x$ is fixed, and $y^\star \in \{y : G(y,x)-G(\widehat{y}^\star,x) \leq \epsilon_2\}$ as an $\epsilon_2$-approximate solution of the same problem. Then, 
\begin{equation}
     \big\|\widehat{y}^\star - y^\star\big\|_2 \leq \sqrt{2s\epsilon_2}.\notag  
\end{equation}
\end{Lemma}
\begin{proof}
 By Assumption \ref{Ass:OptimizationProblem} and \cite{beck2017first}, the function $G$ defined by \eqref{G0} is $1/s$-strongly convex in the first argument and, by Theorem \ref{thm:2}, it satisfies,
 \begin{equation}
     \label{lem11}
     G(y,x)-G(y^\star,x) \geq \frac{1}{2s}\big\|y-y^\star\big\|_2^2,
 \end{equation}
 for any $y$. The error $\epsilon_2$ associated with the suboptimal solution ${y}^\star$ satisfies \begin{equation}
     \label{lem12}
     G(\widehat{y}^\star,x)-G(y^\star,x) \leq \epsilon_2. \notag
 \end{equation}
 Applying \eqref{lem11} with $y=\widehat{y}^\star$ yields
 \begin{equation}
     \big\|\widehat{y}^\star-y^\star\big\|_2 \leq \sqrt{2s\epsilon_2}. \notag
\end{equation}
\end{proof}
\section{Supplementary results}
\begin{Theorem}[\textbf{Quasi-Fejer monotonicity of the sequence generated by the
proximal gradient method}] 
\label{Thm:QuasiFejerPG} 
Let $\{x^k\}_{k\geq0}$ be the sequence generated by the approximate proximal gradient \eqref{Eq:PGDApproximateRepeated} for solving problem \eqref{Eq:ProblemRepeated} under Assumption \ref{Ass:OptimizationProblem} and with $s_k \leq \frac{1}{L}$. Assume that, for $k \geq k_0$, we have $\epsilon_2^{k} \leq c_2 \norm{{x}^{k+1}-{x}^{k}}_2 \leq c_2\rho$ and $\epsilon_2^{k} \leq c_1 \norm{\nabla g(x^{k+1})-\nabla g(x^{k})}_2$. Then for any $x^\star \in X^\star$ and $k\geq0$ we have
\begin{equation}
    \norm{x^{k+1} - x^\star}_2 \leq \norm{x^k - x^\star}_2+ \norm{r^{k+1}}_2+s_k \norm{\epsilon_1^{k}}_2+C_\rho,
    \label{Eq:QuasiFejer}
\end{equation}
where $C_\rho = \sqrt{2Lc_2 \rho} + c_1 L \rho$. If $E^{k+1} := \norm{r^{k+1}}_2+s_k \norm{\epsilon_1^{k}}_2$ is a positive and absolutely summable sequence, 
then $\{x^{k}\}_{k\geq0}$ is a quasi-F\'ejer sequence.
\end{Theorem}
\begin{proof}
Writing $\text{prox}^{\epsilon_2^k}_{s_k h}(x)$ as $ \text{prox}_{s_k h}(x) + r^k$ and $\nabla^{\epsilon_1^k} g(x)$ as $\nabla g(x) + \epsilon_1^k$ for any optimal solution $x^\star$ of \eqref{Eq:ProblemRepeated}, we obtain
\begin{equation}
    \begin{split}
        \norm{{x}^{k+1} - x^\star}_2 &= \bigg\|\text{prox}_{s_k h}(x^{k}- s_k\nabla g(x^{k})-s_k \epsilon_1^{k}) \\&\quad- \text{prox}_{s_{k_0} h}(x^\star - s_{k_0}\nabla g(x^\star)-s_{k_0} \epsilon_1^{k_0}) + r^{k+1}-r^{k_0}\bigg\|_2.
    \end{split}
\end{equation}
By assumption we have $\epsilon_2^{k} \leq c_2 \norm{{x}^{k+1}-{x}^{k}}_2$, or equivalently,\\ $\norm{r^{k+1}}_2$ $\leq$ $\sqrt{{2c_2 \norm{{x}^{k+1}-{x}^{k}}_2}/{s}}$ and $\epsilon_2^{k}$ $\leq c_1 \norm{\nabla g(x^{k+1})-\nabla g(x^{k})}_2$\\$\leq  c_1 L \norm{x^{k+1}-x^{k}}_2$ for $k \geq k_0$. By the triangle inequality we have
\begin{equation}
    \begin{split}
        \norm{{x}^{k+1} - x^\star}_2 &\leq \bigg\|\text{prox}_{s_k h}(x^{k}- s_k\nabla g(x^{k})-s_k \epsilon_1^{k}) 
        \\
        &\quad- \text{prox}_{s_{k_0} h}(x^\star - s_{k_0}\nabla g(x^\star)-s_{k_0} \epsilon_1^{k_0})\bigg\|_2 + \norm{r^{k+1}}_2+\norm{r^{k_0+1}}_2
        \\
        &\leq \bigg\|\text{prox}_{s_k h}(x^{k}- s_k\nabla g(x^{k})-s_k \epsilon_1^{k})
        \\
        &\quad - \text{prox}_{s_{k_0} h}(x^\star - s_{k_0}\nabla g(x^\star)-s_{k_0} \epsilon_1^{k_0})\bigg\|_2 + \norm{r^{k+1}}_2+ \sqrt{\frac{2c_2 \rho}{s}}        
    \end{split}
\end{equation}
where we have used $\norm{{x}^{k_0+1} - {x}^{k_0}}_2 \leq \rho$.

By the nonexpansivity of the proximal operator we have
\begin{equation}
    \begin{split}
        \norm{{x}^{k+1} - x^\star}_2&\leq \bigg\|[x^{k}- s_k\nabla g(x^{k})] - [x^\star - s_{k_0}\nabla g(x^\star)]\bigg\|_2 + \norm{r^{k+1}}_2+ \sqrt{\frac{2c_2 \rho}{s_{k_0}}} 
        \\
        &\quad+ s_k \norm{\epsilon_1^{k}}_2+s_{k_0} \norm{\epsilon_1^{k_0}}_2
        \\
        &\leq \bigg\|[x^{k}- s_k\nabla g(x^{k})] - [x^\star - s_{k_0}\nabla g(x^\star)]\bigg\|_2 + \norm{r^{k+1}}_2 + s_k \norm{\epsilon_1^{k}}_2
        \\
        &\quad+ \sqrt{\frac{2c_2 \rho}{s_{k_0}}} +s_{k_0} c_1 L \rho
    \end{split}
    \label{Eq:QuasiFejerProofExpansive}
\end{equation}
By the nonexpansivity of the gradient descent operator, i.e., $\mathbf{I} - s \nabla g$, we obtain
\begin{align}
      \norm{{x}^{k+1} - x^\star}_2 &\leq \norm{x^{k}- x^\star}_2 + \norm{r^{k+1}}_2+s_k \norm{\epsilon_1^{k}}_2 + C_\rho, \quad \forall s_k \leq \frac{1}{L}\\
    &= \norm{x^k- x^\star}_2 + E^{k+1}  + C_\rho.
    \label{Eq:x-xstar+E1}
\end{align}
where $C_\rho = \sqrt{\frac{2c_2 \rho}{s_{k_0}}} +s_{k_0} c_1 L \rho$ and $E^{k+1} = \norm{r^{k+1}}_2+s_k \norm{\epsilon_1^{k}}_2$. From \eqref{Eq:x-xstar+E1} and by Definition \ref{Def:QuasiFejer}, the sequence $\{{x}^{k}\}_{k\geq 1}$ is quasi-F\'ejer relative to the set $X^\star$ if $\{E^{k}\}_{k\geq 1}$ is positive and absolutely summable.\\ 
\end{proof}
\begin{Theorem}[\textbf{Quasi-Fejer monotonicity of the sequence generated by the 
accelerated proximal gradient method}] 
\label{Thm:QuasiFejerAcceleratedPG} 
Let $\{x^k\}_{k\geq0}$ be the sequence generated by the approximate accelerated proximal gradient \eqref{Eq:AcceleratedPGRepeated} for solving problem \eqref{Eq:ProblemRepeated} under Assumption \ref{Ass:OptimizationProblem} and with $s_k \leq \frac{1}{L}$. Assume we have summable iterative displacements $\norm{x^{k}-x^{k-1}}_2$ and that, for $k \geq k_0$, we have $\epsilon_2^{k} \leq c_2 \norm{{x}^{k+1}-{x}^{k}}_2 \leq c_2\rho$ and $\epsilon_2^{k} \leq c_1 \norm{\nabla g(x^{k+1})-\nabla g(x^{k})}_2^\top$, then for any $x^\star \in X^\star$ and $k\geq0$ we have
\begin{equation}
      \norm{{x}^{k+1} - x^\star}_2 \leq
      \norm{x^{k}- x^\star}_2+ \norm{x^{k}-x^{k-1}}_2 + E^{k+1}  + C_\rho
\end{equation}
where $C_\rho = \sqrt{2Lc_2 \rho} + c_1 L \rho$. If $E^{k+1} := \norm{r^{k+1}}_2+s_k \norm{\epsilon_1^{k}}_2$ is a positive and absolutely summable sequence, 
then $\{x^{k}\}_{k\geq0}$ is a quasi-F\'ejer sequence.
\end{Theorem}
\begin{proof}
For any optimal solution $x^\star$ of \eqref{Eq:ProblemRepeated}, we have
\begin{align}
    \norm{{x}^{k+1} - x^\star}_2 &= \norm{\text{prox}^{\epsilon_2^k}_{s_k h}(y^{k}- s_k\nabla^{\epsilon_1^k} g(y^{k})) - \text{prox}^{\epsilon_2^{k_0}}_{s_{k_0} h}(x^\star - s_{k_0}\nabla^{\epsilon_1^{k_0}} g(x^\star))}_2.
\end{align}
Rewriting $\text{prox}^{\epsilon_2^k}_{s_k h}(y)$ as $ \text{prox}_{s_k h}(y) + r^k$ and $\nabla^{\epsilon_1^k} g(y)$ as $\nabla g(y) + \epsilon_1^k$ we obtain
\begin{equation}
    \begin{split}
        \norm{{x}^{k+1} - x^\star}_2 &= \bigg\|\text{prox}_{s_k h}(y^{k}- s_k\nabla g(y^{k})-s_k \epsilon_1^{k})
        \\
        &\quad- \text{prox}_{s_{k_0} h}(x^\star - s_{k_0}\nabla g(x^\star)-s_{k_0} \epsilon_1^{k_0}) + r^{k+1}-r^{k_0}\bigg\|_2.
    \end{split}
\end{equation}
By assumption we have $\epsilon_2^{k} \leq c_2 \norm{{x}^{k+1}-{x}^{k}}_2$ and \\ $\epsilon_2^{k} \leq c_1 \norm{\nabla g(x^{k+1})-\nabla g(x^{k})}_2 \leq  c_1 L \norm{x^{k+1}-x^{k}}_2$ for $k \geq k_0$. By the triangle inequality we have
\begin{equation}
    \begin{split}
        \norm{{x}^{k+1} - x^\star}_2 &\leq \bigg\|\text{prox}_{s_k h}(y^{k}- s_k\nabla g(y^{k})-s_k \epsilon_1^{k}) 
        \\
        &\quad- \text{prox}_{s_{k_0} h}(x^\star - s_{k_0}\nabla g(x^\star)-s_{k_0} \epsilon_1^{k_0})\bigg\|_2 + \norm{r^{k+1}}_2+\norm{r^{k_0+1}}_2
        \\
         &\leq \bigg\|\text{prox}_{s_k h}(y^{k}- s_k\nabla g(y^{k})-s_k \epsilon_1^{k}) 
         \\
         &\quad- \text{prox}_{s_{k_0} h}(x^\star - s_{k_0}\nabla g(x^\star)-s_{k_0} \epsilon_1^{k_0})\bigg\|_2 + \norm{r^{k+1}}_2+ \sqrt{\frac{2c_2 \rho}{s}}
    \end{split}
\end{equation}
where we have used $\norm{{x}^{k_0+1} - {x}^{k_0}}_2 \leq \rho$.

By the nonexpansivity of the proximal operator we have
\begin{equation}
    \begin{split}
        \norm{{x}^{k+1} - x^\star}_2&\leq \norm{[y^{k}- s_k\nabla g(y^{k})] - [x^\star - s_{k_0}\nabla g(x^\star)]}_2 + \norm{r^{k+1}}_2
        \\
        &\quad+\sqrt{\frac{2c_2 \rho}{s_{k_0}}} + s_k \norm{\epsilon_1^{k}}_2+s_{k_0} \norm{\epsilon_1^{k_0}}_2
        \\
        &\leq \norm{[y^{k}- s_k\nabla g(y^{k})] - [x^\star - s_{k_0}\nabla g(x^\star)]}_2 + \norm{r^{k+1}}_2 
        \\
        &\quad+s_k \norm{\epsilon_1^{k}}_2+ \sqrt{\frac{2c_2 \rho}{s_{k_0}}} +s_{k_0} c_1 L \rho
    \end{split}
    \label{Eq:QuasiFejerProofExpansive}
\end{equation}
By the nonexpansivity of the gradient descent operator, i.e., $\mathbf{I} - s \nabla g$, we obtain
\begin{equation}
    \begin{split}
        \norm{{x}^{k+1} - x^\star}_2 &\leq \norm{y^{k}- x^\star}_2 + \norm{r^{k+1}}_2+s_k \norm{\epsilon_1^{k}}_2 + C_{\rho,s_{k_0}}, \quad \forall s_k \leq \frac{1}{L}\\
        &= \norm{x^{k}- x^\star+\beta_k(x^{k}-x^{k-1})}_2 + E^{k+1}  + C_{\rho,s_{k_0}}\\
        &= \norm{x^{k}- x^\star}_2+\norm{x^{k}-x^{k-1}}_2 + E^{k+1}  + C_{\rho,s_{k_0}},
    \end{split}
    \label{Eq:x-xstar+E2}
\end{equation}
where $C_{\rho,s_{k_0}} = \sqrt{\frac{2c_2 \rho}{s_{k_0}}} +s_{k_0} c_1 L \rho$, $E^{k+1} = \norm{r^{k+1}}_2+s_k \norm{\epsilon_1^{k}}_2$ and we used $\beta_k \leq 1$. From \eqref{Eq:x-xstar+E2} and by Definition \ref{Def:QuasiFejer}, the sequence $\{{x}^{k}\}_{k\geq 1}$ is quasi-F\'ejer relative to the set $X^\star$ if $\{E^{k}\}_{k\geq 1}$ is positive and absolutely summable provided we have summable iterative displacements $\norm{x^{k}-x^{k-1}}_2$.
\end{proof}
\addcontentsline{toc}{section}{Acknowledgment}
\section*{Acknowledgements}
This work was supported by the Engineering and Physical Research Council (EPSRC) grant EP/S000631/1 and the MOD University Defence Research Collaboration (UDRC).
\bibliographystyle{siamplain}
\bibliography{references.bib}
\end{document}